\documentclass{amsart}
\usepackage[utf8]{inputenc}
\usepackage[T1]{fontenc}
\usepackage{lmodern}
\usepackage{amsmath}
\usepackage{amssymb}
\usepackage{mathtools}
\usepackage{latexsym}
\usepackage[lite]{amsrefs}
\usepackage{nicefrac}
\usepackage{microtype}
\usepackage{color} 
\usepackage{tikz-cd}
\usepackage{MnSymbol}
\usepackage{enumitem} 
\setlist[enumerate,1]{label=\textup{(\arabic*)}}

\usepackage[all]{xy}
\newdir{ >}{{}*!/-5pt/@{>}}
\usepackage[pdftitle={Dualisable categories in NCG},
pdfauthor={Devarshi Mukherjee},
pdfsubject={Mathematics}
]{hyperref}

 \usepackage{multicol}
 \usepackage{tabularx}
 
\theoremstyle{plain}
\newtheorem{theorem}{Theorem}[section]
\newtheorem{lemma}[theorem]{Lemma}
\newtheorem{corollary}[theorem]{Corollary}
\newtheorem{proposition}[theorem]{Proposition}
\newtheorem{defprop}[theorem]{Definition/Proposition}
\theoremstyle{remark}
\newtheorem{remark}[theorem]{Remark}
\theoremstyle{question}

\newcommand{\op}{\mathrm{op}}
\theoremstyle{definition}
\newtheorem{definition}[theorem]{Definition}
\newtheorem{example}[theorem]{Example}
\newtheorem{exercise}[theorem]{Exercise}
\newtheorem{conjecture}[theorem]{Conjecture}
\newtheorem{warning}[theorem]{Warning}
\numberwithin{theorem}{section}

\newcommand{\coma}{\widehat}
\newcommand\haotimes{\mathbin{\coma{\otimes}}}

\newcommand{\defeq}{\mathrel{:=}} 

\usepackage{tikz-cd}
\newcommand{\bA}{\mathbf{A}}
\newcommand{\cA}{\mathcal{A}}
\newcommand{\Q}{\mathbb{Q}}
\newcommand{\Z}{\mathbb{Z}}
\newcommand{\C}{\mathbb{C}}
\newcommand{\R}{\mathbb{R}}
\newcommand{\N}{\mathbb{N}}

\newcommand{\bD}{\mathbf{D}}
\newcommand{\bC}{\mathbf{C}}
\newcommand{\bE}{\mathbf{E}}
\newcommand{\cX}{\mathcal{X}}
\newcommand{\colim}{\mathrm{colim}}
\newcommand{\bdd}{\mathcal{B}}
\newcommand{\jhat}{\widehat{\jmath}}

\newcommand{\Prld}{\mathbf{Pr}^\mathrm{L}_\mathrm{dual}}
\newcommand{\xto}{\xrightarrow}
\newcommand{\calC}{\mathbf{C}}
\newcommand{\calD}{\mathbf{D}}
\newcommand{\calE}{\mathbf{E}}
\newcommand{\Calk}{\mathbf{Calk}}

\newcommand{\Ind}{\mathrm{Ind}}
\newcommand{\Shv}{\mathbf{Shv}}
\newcommand{\Sp}{\mathbf{Sp}}
\newcommand{\dirlim}{\underrightarrow{\mathrm{colim}}\,}
\newcommand{\NcMot}{\mathbf{NcMot}}
\newcommand{\Fun}{\mathbf{Fun}}
\newcommand{\map}{\mathrm{map}}
\newcommand{\THH}{\mathrm{THH}}
\newcommand{\coS}{\widehat{\mathbf{coShv}}}
\newcommand{\alg}{\mathrm{alg}}

\theoremstyle{remark} 
\newtheorem*{ack}{Acknowledgements}

\begin{document}

\title{Dualisable categories in noncommutative geometry}

\author{Devarshi Mukherjee}
\email{Devarshi.Mukherjee@abdn.ac.uk}
\address{University of Aberdeen, Institute of Mathematics, Aberdeen, UK}

\author{Thomas Nikolaus}
\email{nikolaus@uni-muenster.de}
\address{Universit\"at M\"unster\\ Mathematics M\"unster, M\"unster, Germany}

\begin{abstract}
In this exposition based on a lecture series at the CIMPA School on Operator Algebras and \(K\)-theory, 2025, we highlight modern perspectives on noncommutative geometry using dualisable categories. To this end, we show that several categories naturally arising in functional analysis (derived categories of bornological and condensed modules), operator algebras (\(E\)-theory) and stable categories (noncommutative motives) are dualisable. This in particular leads to a definition of algebraic \(K\)-theory for analytic spaces. Furthermore, the dualisability of \(E\)-theory and motives  leads us to pursue an analogy between topological and algebraic \(K\)-theory, inspired by assembly map (Baum-Connes and Farrell-Jones) conjectures.
\end{abstract}

\maketitle

\section*{Introduction}

Noncommutative geometry studies ```noncommutative spaces'' through their invariants. The precise meaning of such a space depends on context, with each perspective extending a classical notion of geometry.

From the \emph{operator-algebraic} viewpoint, a noncommutative space is modelled by a (not necessarily commutative) $C^*$-algebra. This is motivated by the Gelfand-Naimark theorem, which identifies commutative $C^*$-algebras with algebras of continuous functions on locally compact Hausdorff spaces.

From the \emph{algebro-geometric} perspective, one assigns to a scheme $X$ its derived $\infty$-category of quasi-coherent sheaves. Under suitable hypotheses, this stable $\infty$-category determines $X$, suggesting that a ``noncommutative scheme'' should be understood as a certain stable~$\infty$-category. This idea, developed by Manin and others, forms the basis of noncommutative algebraic geometry.

Both approaches lead to universal invariants: analytic theories such as $KK$-theory and $E$-theory for $C^*$-algebras, and algebraic invariants such as algebraic $K$-theory for presentable stable~$\infty$-categories. This document surveys recent progress in these areas, including Efimov's advances using dualisable categories and new analytic techniques arising from bornological and condensed perspectives. These developments are closely linked. Our goal is to build a dictionary between the analytic and categorical frameworks, interpreting constructions such as assembly maps in both languages. This philosophy is motivated by major conjectures such as those of Baum--Connes and Farrell--Jones. The two perspectives can be summarised as follows:
\begin{itemize}
\item \textbf{Analytic viewpoint.} \emph{Noncommutative spaces are $C^*$-algebras.} The Gelfand duality identifies locally compact Hausdorff spaces with commutative \(C^*\)-algebras, leading to the point of view that noncommutative \(C^*\)-algebras be treated as models for noncommutative spaces. The prototypical examples of such spaces are topological groupoids, which includes group actions on spaces and graphs. The way one probes \(C^*\)-algebras is through topological $K$-theory, and trace-invariants mapping out of \(K\)-theory such as local cyclic homology. Both these invariants are special cases of \emph{bivariant} theories - $KK$-theory, and $E$-theory, which are the universal homology theories of \(C^*\)-algebras, and \emph{bivariant local cyclic homology}, the best cyclic homological approximation. These invariants coincide (rationally) on a subcategory of \(C^*\)-algebras called the \emph{UCT-class}.

\item \textbf{Categorical viewpoint.} \emph{A noncommutative space is a stable~$\infty$-category.} In analogy with the analytic viewpoint, we may associate to a topological space or scheme \(X\), the category \(\mathsf{Shv}(X)\) of sheaves, or \(\mathsf{QCoh}(X)\) of quasi-coherent sheaves on a scheme, prompting one to view noncommutative spaces as stable \(\infty\)-categories. In this setup, one analyses such spaces through structural notions such as dualisability, compact objects, symmetric monoidal structures, and Verdier localisations. The natural analogue of topological \(K\)-theory is algebraic \(K\)-theory, which is again the special case of a bivariant theory for stable \(\infty\)-categories, called \emph{noncommutative motives}.

\end{itemize}
\medskip

A $C^*$-algebra gives rise to a stable~$\infty$-category of modules, while dualisable stable categories are often equivalent to module categories. These assignments are not inverse equivalences in any reasonable sense, as the algebraic framework forgets analytic structure, but many methods can nonetheless be transferred between the two worlds. 
For example invariants such as $KK$-theory and algebraic $K$-theory may be viewed in parallel as functors from noncommutative spaces to spectra, characterised by excision, additivity, and, in favourable cases, $A^1$-invariance. The following dictionary outlines corresponding concepts in each framework.

\medskip
\begin{table}
\begin{tabularx}{0.8\textwidth} { 
  | >{\raggedright\arraybackslash}X 
  | >{\centering\arraybackslash}X 
  | >{\raggedleft\arraybackslash}X | }
 \hline
  \textbf{Operator algebras} & \textbf{Stable \(\infty\)-categories} \\
 \hline
   \(\C\) &  \(\mathbf{Sp}\)\\
  \hline
  Commutative \(C^*\)-algebras \(C_0(X,\C)\)  & \(\mathbf{Shv}(X, \mathbf{Sp})\)  \\
\hline
(Discrete) Group \(C^*\)-algebras \(C^*(G)\) &  \(\mathbf{Sp}^{BG}\) \\
\hline
Crossed products by discrete groups \(G \rtimes C_0(X,\C)\) & Equivariant sheaves \(\mathbf{Shv}^G(X, \mathbf{Sp})\)\\
\hline
Noncommutative \(C^*\)-algebra & Dualisable category \\
\hline
Topological \(K\)-theory & Algebraic \(K\)-theory \\
\hline
Local cyclic homology & Refined topological periodic homology \\
\hline
\(E\)-theory & Noncommutative motives \\
\hline
\end{tabularx}
\caption{A dictionary between operator algebras and stable \(\infty\)-categories}
\label{ncg-table}
\end{table}

The analogies of \ref{ncg-table} may appear somewhat superficial at first glance; to argue why this is not the case, recent work \cite{bredon} shows that for a finite group \(G\) acting on a locally compact Hausdorff space \(X\), we have \[K^\mathrm{top}(G \ltimes C_0(X,\C)) \simeq \Gamma_c^G(X, \mathbf{K}_G^{\mathrm{top}}),\] where the right hand side is a sheaf-theoretic version of Bredon cohomology, and \(\mathbf{K}_G^{\mathrm{top}}\) is the presheaf on the orbit category of \(G\) defined by \(G/H \mapsto K^{\mathrm{top}}(\C[H])\). In particular, for the trivial group, we have \(K^{\mathrm{top}}(C_0(X,\C)) \simeq \Gamma_c(X, \mathbf{KU})\), where the right hand side is compactly supported cohomology. On the algebraic \(K\)-theory side, we have similarly have
\[
K(\mathbf{Shv}^G(X, \mathbf{Sp})) \simeq \Gamma_c^G(X, K_G), 
\] where \(K_G\) is again the presheaf \(G/H \mapsto K(\mathbf{Sp}^{BH})\). This generalises Efimov's computation (\cite{efimov2024k}) of algebraic \(K\)-theory of sheaves on a locally compact Hausdorff space as compactly supported sheaf cohomology of the space. These computations follow from the characterisation of (equivariant) cohomology theories of locally compact Hausdorff (\(G\)-)spaces valued in dualisable categories, in terms of presheaves on the orbit category. The dualisability of the categories of noncommutative motives and \(E\)-theory applied to the cohomology theories \(X \mapsto C_0(X)\) and \(X \mapsto \mathbf{Shv}^G(X)\) then implies the computations above.

\section{Overview of the Contents}

\paragraph{Section 2 - Presentable, dualisable, and rigid categories.} 
We begin with presentable \(\infty\)-categories: accessibility via \(\Ind_\kappa\), compact objects, and the adjoint functor theorem. We discuss constructions from model categories and characterisations by Bousfield localisations. Within \(\mathrm{Pr}^L\) we recall the closed symmetric monoidal structure and define the tensor product \(C \otimes D\) by separately colimit-preserving functors. We then specialise to \emph{stable} \(\infty\)-categories, reviewing suspension/loop functors, triangulated homotopy categories, and stabilisation \(\Sp(C)\). A key point is the characterisation of \emph{dualisable} objects in \(\mathrm{Pr}^L_{\mathrm{st}}\) and their \emph{rigid} refinements, including the role of compact and trace-class morphisms and the ``dualisable core'' construction that produces a rigidification \(C \mapsto C^{\mathrm{rig}}\).

\paragraph{\emph{Section 3 - Functional analysis revisited: bornologies and condensed mathematics.}}
We formalise bornological sets and bornological \(R\)-modules (for a Banach ring \(R\)), introduce completeness and separation, and explain the \emph{dissection} functor into inductive systems of normed/Banach modules. This yields presentable, closed monoidal categories such as \(\Ind(\mathrm{Ban}_R)\) or its ``formal bornological'' subcategory, which remedy the lack of infinite (co)limits in classical Banach categories. We also indicate how condensed mathematics provides an alternative test-object formalism compatible with these constructions. The emphasis is on the tensor product, internal Hom, and the functorial control required for later \(K\)-theoretic and motivic arguments.

\paragraph{\emph{Section 4 - Higher-categorical perspectives on operator algebras.}}
We review \(KK\)-theory and \(E\)-theory from a higher-categorical vantage: exactness and stability properties, mapping objects, and bifunctoriality. The goal is to express familiar constructions in a language where adjunctions and monoidal structures in \(\mathrm{Pr}^L_{\mathrm{st}}\) can be applied directly.

\paragraph{\emph{Section 5 - Definition of algebraic \(K\)-theory.}}
We present algebraic \(K\)-theory for appropriate stable \(\infty\)-categories/rings, including the positive \(K\)-groups, the \emph{Calkin category}, and \emph{Efimov \(K\)-theory}. The treatment is designed to interface with both exact sequences (Verdier) and filtered colimits, and to prepare for comparisons with homotopy \(K\)-theory and excision.

\paragraph{\emph{Section 6 - Properties of \(K\)-theory.}}
We establish exactness for Verdier sequences, treat non-unital rings, formulate general excision, and introduce homotopy \(K\)-theory. The structural leitmotif is that the good categorical behaviour (exactness, colimit preservation) of the input categories translates into the expected long exact sequences and localisation theorems in \(K\)-theory.

\paragraph{\emph{Section 7 - Noncommutative motives.}}
We define noncommutative motives in a presentable, stable setting, discuss sheaves as motives, \(A^1\)-invariant motives, and \(A^1\)-acyclicity. The focus is on universal properties and the role of dualisability in ensuring that motivic functors interact well with tensor products and localisations.

\paragraph{\emph{Section 8 -  Assembly maps.}}
We formulate assembly maps via algebraic \(KK\)-theory and to relate them to functorial constructions such as cosheaves and functorial assembly. The chapter explains how the motivic and \(KK\)-theoretic assembly pictures can be compared inside the higher-categorical framework established earlier, clarifying the connection with major conjectures (Farrell-Jones) at a structural level.

\begin{ack}
Both authors were funded by Germany’s Excellence Strategy EXC 2044 390685587, Mathematics Münster: Dynamics-Geometry-Structure. The first named author was also supported by a Marie-Curie-Postdoctoral Fellowship, carried out at the Mathematical Institute, University of Oxford.
\end{ack}

\section{Presentable, dualisable and rigid categories}

In this section, we setup terminology that will be relevant to the narrative of the subsequent sections. We note at this point that this section is not meant to be ``yet another introduction to \(\infty\)-categories", and we assume that the reader is familiar with some form of higher category theory. The results stated in this section hold in most known models of higher category theory, though when we provide proofs, we mostly use  quasicategories. 
 The material in this section is assembled from various sources, notably \cite{nkp,Cnossen}.

 Let \(\kappa\) be a regular cardinal. A \emph{\(\kappa\)-finite} \(\infty\)-category is an \(\infty\)-category that can be obtained from taking \(\kappa\)-finite coproducts and pushouts of the categories \([0]\) and \([1]\) in \(\mathbf{Cat}_\infty\). A nonempty \(\infty\)-category \(\bC\) is said to be \emph{\(\kappa\)-filtered} if every diagram \(I \to \bC\) with \(I\) \(\kappa\)-finite admits a cocone. Finally, a \emph{\(\kappa\)-filtered diagram} \(J \to \bC\) is a functor whose indexing category is \(\kappa\)-filtered.

\subsection{Presentable \(\infty\)-categories}

\begin{definition}\label{def:compact}
Let \(\bC\) be an \(\infty\)-category with all small colimits. We call an object \(X \in \bC\) \(\kappa\)-\emph{compact} if \(\mathbf{Hom}_\bC(X,-)\) commutes with \(\kappa\)-filtered colimits. That is, for any \(\kappa\)-filtered diagram \(Y \colon I \to \bC\), the natural map \[\colim_{i \in I} \mathbf{Hom}_\bC(X, Y_i) \to \mathbf{Hom}_\bC(X,\colim_i Y_i)\] is an equivalence in anima.  Denote by \(\bC^\kappa\) the full subcategory of \(\kappa\)-compact objects.  
\end{definition}

Let \(\bC\) be a small \(\infty\)-category. Consider the \emph{\(\infty\)-category \(\mathbf{Fun}(\bC^\op, \mathbf{An})\) of anima-valued presheaves on \(\bC\)}; this has all colimits. In fact, it is the universal way to cocomplete an \(\infty\)-category in the sense that if \(\bD\) is an \(\infty\)-category with all small colimits, left Kan extension along the (fully faithful) Yoneda embedding \[j \colon \bC \to \mathbf{Fun}(\bC^\op, \mathbf{An}), \quad X \mapsto \mathbf{Hom}_\bC(-, X)\] induces an equivalence \begin{equation}\label{eq1-yoneda}
\mathbf{Fun}^{\colim}(\mathbf{Fun}(\bC^\op, \mathbf{An}), \bD) \to \mathbf{Fun}(\bC, \bD),
\end{equation} of \(\infty\)-categories. In what follows, we recall how to add \(\kappa\)-filtered colimits to \(\bC\) for any regular cardinal \(\kappa\). 

\begin{definition}\label{def:ind-kappa}
For a small \(\infty\)-category \(\bC\), we define \(\mathbf{Ind}_\kappa(\bC) \subseteq \mathbf{Fun}(\bC^\op, \mathbf{An})\) as the smallest \(\infty\)-category containing \(j(\bC)\) and is closed under \(\kappa\)-filtered colimits.  
\end{definition}

Let \(\mathbf{Fun}^{\colim_\kappa}\) denote the subcategory of the functor category consisting of \(\kappa\)-filtered colimit preserving functor. Then for an \(\infty\)-category with \(\kappa\)-filtered colimits, one has an equivalence 

\[\mathbf{Fun}^{\colim_\kappa}(\mathbf{Ind}_{\kappa}(\bC),\bD) \to \mathbf{Fun}(\bC, \bD)\] induced by restriction along \(\bC \to \mathbf{Ind}_\kappa(\bC)\) and left Kan extension.

\begin{lemma}\label{lem:compact-generation}\cite[Proposition 2.1.8]{nkp}
Let \(\bC\) be a small \(\infty\)-category and \(\bD\) an \(\infty\)-category with \(\kappa\)-filtered colimits. Let \(f \colon \bC \to \bD\) be a functor and \(F \colon \mathbf{Ind}_\kappa(\bC) \to \bD\) its Yoneda extension. 
\begin{enumerate}
\item If \(f\) is fully faithful, and for every object \(c \in \bC\), \(f(c)\) is \(\kappa\)-compact in \(\bD)\), then  \(F\) is fully faithful;
\item If the image of \(f\) generates \(\bD\) under \(\kappa\)-filtered colimits, then \(F\) is an equivalence of \(\infty\)-categories. 
\end{enumerate}
\end{lemma}

We specialise the situation of Lemma \ref{lem:compact-generation} to the case where \(\bC\) is an \(\infty\)-category with \(\kappa\)-filtered colimits. Then its full subcategory \(\bC^\kappa \subseteq \bC\) of \(\kappa\)-compact objects induces a fully faithful functor \[k \colon \mathbf{Ind}_\kappa(\bC^\kappa) \to \bC,\] which is an equivalence if \(\bC^\kappa\) generates \(\bC\). This motivates the following:

\begin{definition}\label{def:presentable}
Let \(\bC\) be an \(\infty\)-category with small colimits. We call \(\bC\) \emph{presentable} if there is a regular cardinal \(\kappa\) such that the canonical functor \(k \colon \mathbf{Ind}_\kappa(\bC^\kappa) \to \bC\) is an equivalence.
\end{definition}

In what follows, we consider several examples of presentable \(\infty\)-categories. 

\subsubsection{Presentable \(1\)-categories to presentable \(\infty\)-categories}

Let \(\mathsf{C}\) be a locally presentable \(1\)-category. Recall that by definition, this means that \(\mathsf{C}\) is locally small, has all small colimits, and that it is \emph{accessible} in the sense that it is equivalent to \(\mathsf{Ind}_\kappa(D)\) for some small category \(D\) and regular cardinal \(\kappa\).  

\begin{proposition}\label{prop:nerve-presentable}
The nerve of a locally presentable \(1\)-category is a presentable \(\infty\)-category.  
\end{proposition}

\begin{proof}
Viewing ordinary categories as constant simplicial categories, we may apply  \cite[Theorem 4.2.4.1]{HTT}, to conclude that the colimit of a diagram \(F \colon I \to \mathsf{C}\) (which coincides with the homotopy colimit of \(\mathsf{C}\) viewed as a constant simplicial category) is equivalent to the \(\infty\)-categorical colimit of the diagram \(N(F) \colon N(I) \to N(\mathsf{C})\).  Furthermore, the \(\infty\)-functor \(N(\mathsf{Set}) \to N_\Delta(\mathsf{Kan}) \simeq \mathbf{An}\) induced by the functor \(\mathsf{Set} \to \mathsf{Kan}\), assigning to \(X\) the constant simplicial set at \(X\), preserves filtered colimits. Now the accessibility of \(\mathsf{C}\) in particular says that \(\mathsf{Ind}_\kappa(D) \simeq \mathsf{C}\) for some regular cardinal \(\kappa\), and some small category \(D\). In other words, \(\mathsf{C}\) is generated under \(\kappa\)-filtered colimits inside \(\mathsf{Fun}(D^\op, \mathsf{Set})\) by objects in \(j(D)\). The nerve functor preserves these colimits, which are now taken in \(\mathbf{Fun}(N(D)^\op, N(\mathsf{Set}))\). Finally, since the functor \(N(\mathsf{Set}) \to \mathbf{An}\) preserves filtered colimits, we have \(N(\mathsf{C})\) is a full subcategory of \(\mathbf{Ind}_\kappa(N(D))\), which is accessible, so that \(N(\mathsf{C})\) is itself accessible.  
\end{proof}

Proposition \ref{prop:nerve-presentable} gives us many examples of presentable \(\infty\)-categories. 

\begin{example}
The nerve of the category of sets is presentable. As a \(1\)-category this is clear as any set can be written as a filtered union of finite sets. 
\end{example}

\begin{example}\label{ex:Ban}
The category \(\mathsf{Ban}_\C^{\leq 1}\) of Banach spaces with contracting linear maps is presentable. The coequaliser of a parallel pair of bounded linear maps \(f,g \colon V \rightrightarrows W\) is given by the quotient map \(W \to W/\overline{\{f(x) - g(x) \vert x \in W\}}\). For a family \((X_i)_{i \in I}\) of Banach spaces, the coproduct is given by the \(l^1\)-direct sum \[\bigoplus_{i \in I} X_i = \{(x_i) \in \prod_{i \in I} X_i \vert \sum_{i \in I} \vert x_i \vert_i < \infty \}.\] This shows that \(\mathsf{Ban}_\C^{\leq 1}\) has all colimits. Now observe that \(\mathsf{Hom}_{\mathsf{Ban}_\C^{\leq 1}}(\C,X) \cong B_X\), where \(B_X\) is the unit ball of \(X\), and that the unit ball functor preserves \(\omega_1\)-filtered colimits. This shows that \(\C\) is \(\omega_1\)-compact. To see that \(\C\) is a strong generator, let \(f\), \(g\) be two distinct maps \(X \to Y\). Choose \(x\) such that \(f(x) \neq g(x)\). After possibly rescaling, we may assume that \(\vert x \vert \leq 1\). Then \(u \colon \C \to X\), \(1 \mapsto x\) is a contraction satisfying \(f \circ u \neq g \circ u\), showing that \(\C\) is a generator. It is strong as if \(f \colon X \to Y\) induces an isomorphism \(\mathsf{Hom}_{\mathsf{Ban}_\C^{\leq 1}}(\C,X) \cong \mathsf{Hom}_{\mathsf{Ban}_\C^{\leq 1}}(\C,Y)\), then \(B_X \cong B_Y\) implying that \(f\) is an isomorphism in \(\mathsf{Ban}^{\leq 1}\). 
\end{example}

\subsubsection{Presentable \(\infty\)-categories through model categories}

The classical way to come up with presentable \(\infty\)-categories has been to use \emph{model category} structures on suitable \(1\)-categories. Suppose we have a simplicial model category \(\mathsf{C}\). Let \(\mathsf{C}_{cf}\) denote its full subcategory of fibrant-cofibrant objects. Then taking its homotopy coherent nerve \(N_\Delta(\mathsf{C}_{cf})\) is an \(\infty\)-category, which we call the \emph{underlying \(\infty\)-category} of a model category. When the model structure is nice enough, the following result says that all presentable \(\infty\)-categories arise this way:

\begin{theorem}\label{thm:model-simplicial}\cite[A.3.7.6]{HTT}
An \(\infty\)-category is presentable if and only if it is equivalent to the underlying \(\infty\)-category of combinatorial, (simplicial) model category.
\end{theorem}

Using Theorem \ref{thm:model-simplicial}, we can come up with many interesting examples:

\begin{example}\label{ex:anima}
The category of simplicial sets with the \emph{Kan-Quillen model structure} is a combinatorial model category. Its fibrant-cofibrant objects are precisely the Kan complexes; taking the homotopy-coherent nerve yields anima. 
\end{example}

\begin{example}\label{ex:modules}
The derived \(\infty\)-category of modules over a ring is presentable. To see this, one equips the category of chain complexes over the Grothendieck abelian category of \(R\)-modules with the \emph{projective model structure}. It has kernels and cokernels, and is therefore idempotent complete; being Grothendieck abelian ensures implies that it is locally presentable. As a consequence, the model structure is combinatorial. We shall return to this example in greater detail. 
\end{example}

\subsection{Presentability and adjunctions}

Recall that any left adjoint functor \(F \colon \bC \to \bD\) between two \(\infty\)-categories preserves colimits. The main feature of presentable \(\infty\)-categories is that the converse also holds. More precisely, we have the following:

\begin{theorem}[Adjoint functor theorem]\label{thm:adjoint-functor-thm}\cite[Corollary 5.5.2.9, 5.5.2.10]{HTT}
Let \(\bC\) be presentable and \(\bD\) be any \(\infty\)-category and $F \colon \bC \to \bD$ be any functor.
\begin{enumerate}
\item \(F\) admits a right adjoint if and only if it preserves small colimits;
\item If \(\bD\) is additionally presentable, then \(F\) admits a left adjoint if and only if it preserves limits and \(\kappa\)-filtered colimits for some \(\kappa\).
\end{enumerate}
\end{theorem}

We denote by \(\mathbf{Pr}^L\) the subcategory of \(\mathbf{Cat}_\infty\) generated by presentable \(\infty\)-categories with colimit-preserving (or equivalently, left adjoint) functors as morphisms. 

\subsubsection{Presentability via Bousfield localisations}

We end this section with a characterisation of presentable \(\infty\)-categories in terms of a generators and relations-type identification. Let \(\bC\) be a presentable \(\infty\)-category, and \(W\) a small set of morphisms. We call an object \(X \in \bC\) \emph{\(W\)-local} if for any \(f \colon Y \to Z \in W\), the map \[\mathbf{Hom}_\bC(f,X) \colon \mathbf{Hom}_\bC(Z,X) \to  \mathbf{Hom}_\bC(Y,X)\] is an equivalence. Let \(\bC_W\) be the full subcategory of \(W\)-local objects in \(\bC\). Then the inclusion \(\bC_W \to \mathbf{C}\) has a left adjoint \(L \colon \bC \to \bC_W\) that sends morphisms in \(W\) to equivalences. More generally, we have the following: 

\begin{definition}\label{def:Bousefield}
A \emph{left Bousfield localisation} of an \(\infty\)-category \(\bC\) is an adjoint functor pair \(L \colon \bC \leftrightarrows \bD \colon R\), such that \(R \colon \bD \to \bC\) is fully faithful.  
\end{definition}

\begin{proposition}\label{prop:presentable-char}\cite[Corollary 2.1.28]{nkp}
An \(\infty\)-category \(\bC\) is presentable if and only if it is a left Bousfield localisation of \(\mathbf{Fun}(\bC_0^\op, \mathbf{An})\) for some small \(\infty\)-category \(\bC_0\).  
\end{proposition}

As a corollary, we get another important class of examples which will appear later in the course:

\begin{example}\label{ex:sheaves}
Let \(X\) be a topological space and \(\mathsf{Op}(X)\) the category of its open subsets and inclusions between open subsets. Then the \(\infty\)-category \(\mathbf{Fun}(\mathsf{Op}(X)^\op, \mathbf{An})\) of \emph{presheaves on \(X\)} is presentable by Proposition \ref{prop:presentable-char}. The \(\infty\)-category \(\mathbf{Shv}(X, \mathbf{An})\) of \emph{sheaves on \(X\)} is a left Bousfield localisation of  \(\mathbf{Fun}(\mathsf{Op}(X)^\op, \mathbf{An})\) with respect to the local objects relative to the morphisms \(\emptyset \to j(\emptyset)\), \(j(U) \coprod_{j(U \cap V)} j(V) \to j(U \cup V)\) and \(\mathrm{colim}_{i \in I} j(U_i) \to j(\bigcup_{i \in I} U_i)\) of representable presheaves, where \(I\) is a filtered category. 
\end{example}

\subsection{The category of presentable \(\infty\)-categories}

We end this section with some properties of the category of presentable \(\infty\)-categories. 

\begin{corollary}\cite[Corollary 2.1.30]{nkp}\label{cor:fun-present}
Let \(\bC\) and \(\bD\) be presentable \(\infty\)-categories. Then the category of colimit-preserving functors \(\mathbf{Fun}^L(\bC,\bD)\) is itself presentable. 
\end{corollary}

As a consequence of Corollary \ref{cor:fun-present}, we deduce that the category of presentable \(\infty\)-categories \(\mathbf{Pr}^L\) is enriched over itself. In what follows, we show that it also has a tensor product, making \(\mathbf{Pr}^L\) a \emph{closed} symmetric monoidal category. 

Let \(\bC\), \(\bD\) be presentable \(\infty\)-categories; denote by \(\mathbf{Fun}^{bL}(\bC \times \bD, \mathbf{E})\) full subcategory of functors \(\bC \times \bD \to \mathbf{E}\) that preserve colimits separately in each input, that is, separately continuous functors. Then the \emph{tensor product} is a presentable \(\infty\)-category \(\bC \otimes \bD\) with a separately continuous functor \(\bC \times \bD \to \bC \otimes \bD\) such that precomposition induces an equivalence \[\mathbf{Fun}^L(\bC \otimes \bD, \mathbf{E}) \to \mathbf{Fun}^{bL}(\bC \times \bD, \mathbf{E})\] of \(\infty\)-categories. It is an easy exercise to check that the tensor product is unique if it exists. The following proposition shows that the tensor product indeed exists:

\begin{proposition}\label{prop:tensor-product}\cite[Proposition 4.8.1.7]{HA}
We have \(\bC \otimes \bD \simeq \mathbf{Fun}^{\lim}(\bC^\op, \bD)\), where the right hand side denotes the full subcategory of limit-preserving functors \(\bC^\op \to \bD\). 
\end{proposition}

We note that a functor \(\bC^\op \to \bD\) is limit preserving precisely if it is right adjoint. To see this apply the adjint functor theorem to the functor considered as a functor $\bC \to \bD^\op$.

\subsection{Stable \(\infty\)-categories}

We now come to \(\infty\)-categories that are sufficiently additive, rendering the possibility to do homological algebra in higher categorical settings. But before we get there, we consider a precursor formalism commonly used in noncommutative geometry, called \emph{triangulated categories}. 

We start with an example. Let \(\mathsf{C}\) be an additive category and \(\mathsf{HoKom}(\mathsf{C})\) its homotopy category of chain complexes. Morphisms in this category are chain maps \(f \colon X \to Y\) up to chain homotopy. For any such chain map, we may form a diagram \[A \overset{f}\to B \to \mathsf{cone}(f) \to A[1],\] where \(\mathsf{cone}(f)\) is the \emph{mapping cone} of \(f\), and \(A[1]\) is the chain complex \(A\) shifted by \(1\)-degree. The diagram above induces a long exact sequence in bivariant homology \[\cdots \to H_i(A,D) \to H_i(B,D) \to H_i(\mathsf{cone}(f), D) \to H_{i-1}(A,D) \to \cdots,\] where \(D\) is an arbitrary chain complex. An immediate implication is that a chain map \(f \colon A \to B\) is a chain homotopy equivalence if and only if the mapping cone is contractible, showing that the mapping cone plays the role of a joint ``kernel" and ``cokernel" in the homotopy category. Summarising, we have 

\begin{enumerate}
\item an additive category \(\mathsf{HoKom}(\mathsf{C})\);
\item an auto equivalence \(\Sigma \colon \mathsf{HoKom}(\mathsf{C}) \to \mathsf{HoKom}(\mathsf{C})\), namely, the \emph{shift} \(X[1]\) of a chain complex \(X\);
\item distinguished diagrams of the form \(A \to B \to C \to \Sigma(A)\), where \(C\) is called the \emph{cofibre} of the map \(A \to B\); 
\item functors (such as \(\mathsf{Hom}_\mathsf{C}(D,-)\)) taking a class of distinguished triangles to a long exact sequence.
\end{enumerate}    

In the setting of \(1\)-categories, when the data above is subject to certain axioms, we arrive at the notion of a \emph{triangulated category}. As mentioned above, several categories in noncommutative (algebraic) geometry come equipped with triangulated category structures. This includes the (bounded) derived category of a scheme, the stable homotopy category of CW complexes, Kasparov's bivariant \(K\)-theory and related homology theories of (topological) algebras. While triangulated categories provide a reasonable setting for homological algebra (see for instance, \cite{meyer2007homological}) in noncommutative-geometric settings, they are inconvenient for (at least) the following reasons:

\begin{enumerate}
\item Exact triangles are \emph{extra structure} (rather than a \emph{property} of the category);
\item In the case of classical relative homological algebra, we may start with a symmetric monoidal category \(\mathsf{C}\). For an algebra object \(A \in \mathsf{Alg}(\mathsf{C})\), we may consider the category of \(A\)-modules \(\mathsf{Mod}_A(\mathsf{C})\) and do homological algebra (that is, define projective resolutions and derived functors) within this category. If one tries to do the same thing with a triangulated category by considering \(A\)-module objects in a triangulated category, the resulting category does not have an obvious triangulated category structure. A similar problem appears when one wants to add colimits to a triangulated category, say by taking an ind-completion. The resulting category does not inherit any canonical triangulated category structure from the original triangulated category. This applies concretely to bivariant algebraic \(K\)-theory (\cite{Cortinas-Thom:Bivariant_K}), which does not seem to have infinite colimits. 
\item Homotopy (co)limits of diagrams are neither always functorial, nor are they unique in a triangulated category. 
\end{enumerate}

To correct these shortcomings, we work with a suitable homotopy theoretic enhancement called \emph{stable \(\infty\)-categories}.

\begin{definition}\label{def:pointed}
An \(\infty\)-category \(\bC\) is called \emph{pointed} if it has an initial and a terminal object and they both coincide. We call this the zero object and denote it by \(0\).  
\end{definition}

We call a diagram \(A \overset{f}\to B \overset{g}\to C\) in a pointed \(\infty\)-category \(\bC\) an \emph{extension} if it is part of a commuting square 
\[
\begin{tikzcd}
A \arrow{r}{f} \arrow{d}{} & B \arrow{d}{g} \\
0 \arrow{r}{} & C
\end{tikzcd}
\] in \(\bC\). We call an extension a \emph{fibre sequence} (resp. \emph{cofibre sequence}) if it is a pullback square (resp. pushout square).

\begin{definition}\label{def:stable}
A pointed \(\infty\)-category \(\bC\) with pushouts and pullbacks is called \emph{stable} if every pushout square is a pullback square (and vice-versa).
\end{definition}

To understand stable \(\infty\)-categories and relate them to triangulated categories, in any pointed \(\infty\)-category \(\bC\) with fibres and cofibres, we may consider the pushout diagram 
\[
\begin{tikzcd}
A \arrow{r}{} \arrow{d}{} & 0 \arrow{d}{g} \\
0 \arrow{r}{} & \Sigma A
\end{tikzcd}
\] along the canonical map \(A \to 0\) to the zero object. The object \(\Sigma A\) is called the \emph{suspension} of \(A\). Moreover, by the functoriality of pushouts in \(\infty\)-categories, we get an endofunctor \(\Sigma \colon \bC \to \bC\). Dually, the pullback along the unique map \(0 \to A\) induces a functor \(\Omega \colon \bC \to \bC\) called the \emph{loop functor}.

\begin{lemma}\label{def:suspension-loop-adjunction}
Let \(\bC\) be a pointed \(\infty\)-category with fibres and cofibres. Then \(\Omega \colon \bC \to \bC\) is right adjoint to the suspension functor \(\Sigma\). 
\end{lemma}

\begin{proof}
Consider the mapping space \(\mathbf{Hom}_\bC(X, \Omega Y)\). By the universal property of pullbacks, maps to the pullback are equivalent to the anima of natural transformations \(\mathsf{const}_X \to F\) from the constant diagram to the diagram \(F\) corresponding to the morphisms \(0 \to Y \leftarrow 0\) in \(\bC\). This in turn is equivalent to commuting diagrams 
\[
\begin{tikzcd}
X \arrow{r}{} \arrow{d}{} & 0 \arrow{d}{} \\
0 \arrow{r}{} & Y.
\end{tikzcd}
\] Dualising the argument for the pushout of the maps \(0 \leftarrow X \to 0\) in \(\bC\) shows that the mapping anima \(\mathbf{Hom}_\bC(\Sigma X, Y)\) is equivalent to the same commuting diagram above.  
\end{proof}

\begin{proposition}\label{prop:stable-equiv}
The following statements are equivalent for a pointed \(\infty\)-category \(\bC\):
\begin{enumerate}
\item \(\bC\) is a stable \(\infty\)-category;
\item \(\bC\) has fibres and cofibres, and every fibre sequence is a cofibre sequence (and vice-versa);
\item  \(\bC\) has fibres and the loop functor \(\Omega \colon \bC \to \bC\) is an equivalence;
\item \(\bC\) admits cofibres and the suspension functor \(\Sigma \colon \bC \to \bC\) is an equivalence.
\end{enumerate}
\end{proposition}

\begin{proof}
(1) clearly implies (2). For the converse, the fibre sequence 
\[
\begin{tikzcd}
\Omega X \arrow{r}{} \arrow{d}{} & 0 \arrow{d}{} \\
0 \arrow{r}{} & X
\end{tikzcd}
\] is also a cofibre sequence. But this implies that the natural map \(\Sigma \Omega X \to X\) is an equivalence. Symmetrically, the natural map \(\Omega \Sigma X \to X\) is an equivalence, showing that \(\Sigma\) and \(\Omega\) are mutual inverses. Now consider the fibre \(Q = \mathsf{fib}(X \overset{0}\to \Sigma(Y))\). Using this, we have for \(T \in \bC\), \[\mathbf{Hom}_{\bC}(T, Q) \simeq \mathbf{Hom}_{\bC}(T,X) \times \Omega \mathbf{Hom}_{\bC}(T, \Sigma(Y)) \simeq \mathbf{Hom}_{\bC}(T,X) \times \mathbf{Hom}_{\bC}(T,Y).\] This establishes the existence of products. Similarly, applying \(\mathbf{Hom}_{\bC}(-,T)\) to the cofibre sequence \(0 \to \Sigma^{-1}(X) \overset{0}\to Y)\) gives coproducts. The coincidence of fibre and cofibre sequences implies that binary products and coproducts coincide. To construct arbitrary pullbacks, we consider the cospan \(B \overset{h}\to D \overset{k}\leftarrow C\), and construct the map \([h,-k] \colon B \oplus C \to D\). Define by \(Q = \mathsf{fib}([h,-k])\); then for any \(T\),

$$ \mathbf{Hom}_{\bC}(T,Q) \simeq * \times_{\mathbf{Hom}_{\bC}(T,D)} \mathbf{Hom}_{\bC}(T,B\oplus C). $$
Using the biproduct, we have

$$ \mathbf{Hom}_{\bC}(T,B\oplus C) \simeq \mathbf{Hom}_{\bC}(T,B)\times \mathbf{Hom}_{\bC}(T,C), $$

and the map to \(\mathbf{Hom}_{\bC}(T,D)\) is $ (u,v)\longmapsto h\circ u-k\circ v$.

Thus

$$\mathbf{Hom}_{\bC}(T,Q) \simeq \mathbf{Hom}_{\bC}(T,B) \times_{\mathbf{Hom}_{\bC}(T,D)} \mathbf{Hom}_{\bC}(T,C).$$

Hence \(B \times_D C \simeq \mathsf{fib}(B \oplus C \overset{[h,-k]}\to D)\). Similarly, pushouts are constructed by taking the span \(B \overset{f}\leftarrow A \overset{g}\to C\) and forming the cofibre of the map \((f,-g) \colon A \to B \oplus C\). It is now an easy exerise to see that the pushouts and pullbacks agree. This proves that (2) and (1) are equivalent. The proof of the statement that (2) implies (1) shows that (2) implies (3) and (4). We leave the proof of (3) and (4) implying (2) as an exercise (see \cite[Theorem 4.2.2]{Cnossen}).
\end{proof}

Just as we saw with presentable \(\infty\)-categories, a rich supply of examples of stable \(\infty\)-categories comes from \emph{stable, simplicial} model categories. 

\begin{theorem}\label{thm:stable-model}
Let \(\mathsf{C}\) stable simplicial model category. Then its underlying \(\infty\)-category \(\bC = N_\Delta(C_{cf})\) is stable.  
\end{theorem}

\begin{example}
Recall that a spectrum is a collection of pointed spaces \((E_n)_{n \in \N}\) and linking homeomorphisms \(E_n \overset{\simeq}\to \Omega E_{n+1}\) for each \(n\), where \(\Omega\) is the loop space functor. The collection of spectrum objects form an \(\infty\)-category, called \emph{spectra}. One way to see its stability is to use any of the stable model structures on a model of the \(1\)-category of spectra (such as symmetric spectra), and appealing to Theorem \ref{thm:stable-model}. It turns out that the underlying \(\infty\)-categories of each of these model categories are equivalent. But instead of providing any further details here, we will return to the \(\infty\)-category of spectra from a different perspective.
\end{example}

\begin{theorem}\label{thm:stable-triangulated}
Let \(\bC\) be a stable \(\infty\)-category. Then its homotopy category \(\mathbf{Ho}(\bC)\) is triangulated. 
\end{theorem}

\begin{proof}
To show that the homotopy category is triangulated, we need an autoequivalence on \(\bC\) and a class of distinguished triangles satisfying the required axioms of a triangulated category. Since \(\bC\) is stable, the functor \(\Sigma \colon \bC \to \bC\) is an autoequivalence, so that \(\mathsf{Ho}(\Sigma) \colon \mathsf{Ho}(\bC) \to \mathsf{Ho}(\bC)\) is an equivalence. 
For the distinguished triangles, consider the pushout diagram \[
\begin{tikzcd}
A \arrow{r}{f} \arrow{d}{} & B \arrow{d}{} \\
0 \arrow{r}{} & \mathsf{cofib}(f).
\end{tikzcd}
\] This induces a diagram \(A \to B \to \mathsf{cofib}(f) \to \Sigma(A)\) in \(\mathsf{Ho}(\bC)\).
\end{proof}

\begin{definition}\label{def:exact-fun}
A functor \(F \colon \bC \to \bD\) between stable \(\infty\)-categories is called \emph{exact} if \(F(0) \simeq 0\), and \(F\) preserves fibre (or equivalently cofibre) sequences.
\end{definition}

We denote by \(\mathbf{Cat}_{\infty}^{st} \subset \mathbf{Cat}_\infty\) the subcategory spanned by stable \(\infty\)-categories and exact functors. 

\begin{theorem}\label{thm:lim-underlying}
The category \(\mathbf{Cat}_{\infty}^{st}\) has all small limits, and the inclusion \(\mathbf{Cat}_{\infty}^{st} \to \mathbf{Cat}_{\infty}\) preserves them. 
\end{theorem}

\subsubsection{Stabilisation of \(\infty\)-categories}

Let \(\bC\) be an \(\infty\)-category with a terminal object \(*\). We may then define the \(\infty\)-category of \emph{pointed objects of} \(\bC\) as the pullback 

\[
\begin{tikzcd}
\bC_* \arrow{r}{} \arrow{d}{} & \mathbf{Fun}([1], \bC) \arrow{d}{\mathrm{ev}_0} \\
* \arrow{r}{} & \bC 
\end{tikzcd}
\] in \(\mathbf{Cat}_\infty\). Loosely, \(\bC_*\) comprises of objects \(X \in \bC\) with a morphism \(* \to X\) from the terminal object. 

\begin{lemma}\cite[Lemma 4.1.9]{Cnossen}\label{lem:pointed-char}
Let \(\bC\) be an \(\infty\)-category with a terminal object. Then \(\bC_*\) is pointed. Furthermore, \(\bC\) is pointed if and only if the forgetful functor \(\bC_* \to \bC\) is an equivalence. 
\end{lemma}

To see that \(\bC_*\) is the universal way to turn an \(\infty\)-category with a terminal object into a pointed \(\infty\)-category, consider the subcategory \(\mathbf{Cat}_\infty^* \subset \mathbf{Cat}_\infty\) of \(\infty\)-categories with a terminal object and terminal object preserving functors, and the full subcategory \(\mathbf{Cat}_\infty^{\mathrm{pt}} \subset \mathbf{Cat}_\infty^*\) of pointed \(\infty\)-categories. 

\begin{lemma}\cite[Lemma 4.1.13]{Cnossen}\label{lem:adding-point}
The inclusion \(\mathbf{Cat}_\infty^{\mathrm{pt}} \to \mathbf{Cat}_\infty^*\) admits a right adjoint \[\mathbf{Cat}_\infty^* \to \mathbf{Cat}_\infty^{\mathrm{pt}}, \quad \bC \mapsto \bC_*\] with counit given by the forgetful functor \(\bC_* \to \bC\).   
\end{lemma}

To get a stable \(\infty\)-category from a pointed \(\infty\)-category, we additionally need to ensure that the \(\infty\)-category we start with has pullbacks. This is equivalent to the category we start with having finite limits. 

\begin{definition}\label{def:stabilisation}
Let \(\bC\) be an \(\infty\)-category with finite limits. The \emph{stabilisation of \(\bC\)} is defined as the limit \[\mathbf{Sp}(\bC) := \lim (\cdots \overset{\Omega}\to C_* \overset{\Omega}\to C_*)\] taken in \(\mathbf{Cat}_{\infty}\).  
\end{definition}

When we start with the presentable \(\infty\)-category \emph{anima}, which in particular has finite limits, the stabilisation \[\mathbf{Sp}(\mathbf{An}) = \lim (\cdots \mathbf{An}_* \overset{\Omega}\to \mathbf{An}_*)\] is called the \(\infty\)-category of \emph{spectra}, and is denoted simply by \(\mathbf{Sp}\). We call the projection to the first factor \(\Omega_*^\infty \colon \mathbf{Sp} \to \mathbf{An}_*\) the \emph{infinite loop space}. This functor has a left adjoint \[\Sigma_*^\infty \colon \mathbf{An}_* \to \mathbf{Sp}\] called the \emph{suspension spectrum} functor. Postcomposing \(\Omega_*^\infty\) with the forgetful functor \(\mathbf{An}_* \to \mathbf{An}\), we get a functor \[\Omega^\infty \colon \mathbf{Sp} \to \mathbf{An},\] which is left adjoint to the functor \[\mathbb{S}[-] \colon \mathbf{An} \to \mathbf{Sp}, \quad X \mapsto \Sigma_*^\infty(X_+),\] where \(X_+ = X \sqcup \{*\}\). In particular, \(\mathbb{S} := \mathbb{S}[\mathrm{pt}]\) is called the \emph{sphere spectrum}.

We end this subsection with a characterisation of stable \(\infty\)-categories that are additionally presentable, and a procedure to stabilise presentable \(\infty\)-categories. 

\begin{theorem}\cite[Corollary 2.8.12]{nkp}\label{thm:stable-presentable}
\begin{enumerate}
\item The functor \(\mathbb{S}[-] \colon \mathbf{An} \to \mathbf{Sp}\) induces an equivalence \(\mathbf{Sp} \simeq \mathbf{Sp} \otimes \mathbf{Sp}\). The inverse exhibits \(\mathbf{Sp}\) as a commutative algebra in \(\mathbf{Pr}^L\);
\item A presentable \(\infty\)-category \(\bC\) is stable if and only if \(\bC \simeq \mathbf{Sp} \otimes \bC\).  
\end{enumerate}
\end{theorem}

\begin{corollary}\label{cor:stabilisation-presentable}
Let \(\bC\) be a presentable \(\infty\)-category. Then \(\bC \otimes \mathbf{Sp}\) is stable. 
\end{corollary}

We denote by \(\mathbf{Pr}_{st}^{L} \subset \mathbf{Pr}^L\) the full subcategory of presentable, stable \(\infty\)-categories with left adjoint functors as morphisms. Note that an immediate consequence of Theorem \ref{thm:stable-presentable} and Corollary \ref{cor:stabilisation-presentable} is that the tensor product of presentable, stable \(\infty\)-categories is again stable. Therefore, the Lurie tensor product restricts to a closed symmetric monoidal structure on \(\mathbf{Pr}_{st}^L\) with \(\mathbf{Sp}\) as the tensor unit, and internal Hom given by \(\mathbf{Fun}^L(\mathbf{C}, \bD)\).  

Finally, let \(\kappa\) be a fixed regular cardinal. Consider the subcategory \(\mathbf{Pr}_{st}^{L, \kappa} \subset \mathbf{Pr}_{st}^L\) of \(\kappa\)-compactly generated, stable \(\infty\)-categories, with colimit-preserving functors \(F \colon \bC \to \bD\) restricting to a functor \(\bC^\kappa \to \bD^\kappa\). Then given \(\bC \in \mathbf{Pr}_{st}^{L,\kappa}\), its full subcategory \(\bC^\kappa\) of \(\kappa\)-compact objects is a \(\kappa\)-small, stable \(\infty\)-category. Denote by \(\mathbf{Cat}_\infty^{\kappa, \mathrm{perf}}\) the category of small stable \(\infty\)-categories with \(\kappa\)-small colimits and \(\kappa\)-small colimit preserving functors, spanned by the idempotent complete categories.

\begin{theorem}\label{thm:compact-gen-cat-inf}
The functor \[\mathbf{Pr}_{st}^{L,\kappa} \to \mathbf{Cat}_\infty^{\kappa, \mathrm{perf}}, \quad \bC \mapsto \bC^\kappa\] constitutes an equivalence of \(\infty\)-categories, with inverse given by \(\bC \mapsto \mathbf{Ind}_\kappa(\bC)\). 
\end{theorem}

We denote the special case of \(\kappa = \omega\) by \(\mathbf{Cat}_\infty^{\mathrm{perf}}\). 

\subsubsection{Homological algebra}

Let \(\cA\) be an abelian category. We associate to \(\cA\) a stable \(\infty\)-category \(\bD(\cA)\), whose homotopy category is the usual derived category over \(\cA\); the latter being defined as the Verdier localisation of the homotopy category of unbounded chain complexes over \(\cA\) at the quasi-isomorphisms. 

Recall that a \emph{chain complex} \((X,d)\) is a diagram \[\cdots X_n \overset{d_n}\to X_{n-1} \overset{d_{n-1}}\to \cdots \to X_0 \to X_{-1} \to \cdots,\] with \(X_n \in \cA\) and \(d_* \colon X_* \to X_{*-1}\) a morphism in \(\cA\), where \(d_{n-1}\circ d_n = 0\) for each \(n \in \Z\). We often leave the maps \(d_*\) out of the notation, and simply shorten the notation of a chain complex \((X,d)\) to \(X\). A \emph{chain map} \(f \colon X \to Y\) are morphisms \(f_n \colon X_n \to Y_n\) in \(\cA\) such that the following diagram 
\[
\begin{tikzcd}
X_n \arrow{d}{f_n} \arrow{r}{d_n^X} & X_{n-1} \arrow{d}{f_{n-1}} \\
Y_n \arrow{r}{d_n^Y} & Y_{n-1} 
\end{tikzcd}
\] commutes for each \(n\). Chain complexes and chain maps in an additive category themselves form a category that we denote by \(\mathsf{Ch}(\cA)\). The full subcategory \(\mathsf{Ch}_{\geq 0}(\cA)\) of chain complexes \((X,d)\) with \(X_n \cong 0\) for \(n <0\) is called the category of \emph{connective chain complexes}. The inclusion \(\mathsf{Ch}_{\geq 0}(\cA) \subset \mathsf{Ch}(\cA)\) is left adjoint to the \emph{truncation} functor \[\tau_{\geq 0} \colon \mathsf{Ch}(\cA) \to \mathsf{Ch}_{\geq 0}(\cA), \quad (X,d) \mapsto (\cdots \to X_2 \to X_1 \to \ker(d_0)\to 0 \to \dotsc).\] In what follows, we specialise to the additive category \(\cA = \mathsf{Mod}_\Z\) of abelian groups; in this case, we have the following well-known result:

\begin{lemma}[Dold-Kan Correspondence]\label{lem:Dold-Kan}
The category \(\mathsf{Ch}_{\geq 0}(\mathsf{Mod}_\Z)\) of connective chain complexes of abelian groups is equivalent to the category of simplicial abelian groups. Furthermore, the functor \(\mathsf{DK} \colon \mathsf{Ch}_{\geq 0}(\mathsf{Mod}_\Z) \to \mathsf{Fun}(\Delta^\op, \mathsf{Mod}_\Z)\) implementing the equivalence is lax monoidal.  
\end{lemma}

\begin{proof}
In one direction, we define the functor \(\mathsf{DK} \colon \mathsf{Ch}_{\geq 0}(\mathsf{Mod}_\Z) \to \mathsf{Fun}(\Delta^\op, \mathsf{Mod}_\Z) \) that takes a non-negatively graded chain complex \((A_n, d_n)\) to the simplicial abelian group defined by \[[n] \mapsto \bigoplus_{[n] \twoheadrightarrow [k]} A_k,\] where \([n] \twoheadrightarrow [k] \) is an epimorphism in \(\Delta\). For \([n'] \to [n]\) in \(\Delta\), the corresponding map \(\bigoplus_{[n] \twoheadrightarrow [k]} A_k \to \bigoplus_{[n'] \to [k']} A_{k'}\) is defined by a matrix of maps \(f_{k,k'} \colon A_k \to A_{k'}\) as follows:

\begin{itemize}
\item if \(k = k'\) and the diagram 
\[
\begin{tikzcd}
 {[n']} \arrow{r}{} \arrow{d}{} & {[n]} \arrow{d}{} \\
{[k']} \arrow{r}{\mathrm{id}} & {[k]}
\end{tikzcd}
\] commutes, the entries \(f_{k,k'} = \mathrm{id}\);
\item if \(k' = k-1\), and the diagram  
\[
\begin{tikzcd}
 {[n']} \arrow{r}{} \arrow{d}{} & {[n]} \arrow{d}{} \\
{[k'] = \{1,\dotsc, k\}} \arrow{r}{\mathrm{id}} & {[k]}
\end{tikzcd}
\] commutes, the entries \(f_{k,k'} = d\);
\item \(0\) otherwise.
\end{itemize}

In the other direction, we define the normalised chain complex functor \[N \colon \mathsf{Fun}(\Delta^\op, \mathsf{Mod}_\Z)  \to \mathsf{Ch}_{\geq 0}(\mathsf{Mod}_\Z)\] that associates to a simplicial abelian group \((A_\bullet)\), the chain complex whose entries are \[N_n(A) = \ker(A_n \overset{(d_{1 \leq i \leq n})}\to \bigoplus_{1 \leq i \leq n} A_{n-1}).\] For \(n >0\), the face map \(d_0 \colon N_n(A) \to N_{n-1}(A)\), so that we get a chain complex \(\cdots \to N_2(A) \to N_1(A) \to N_0(A) \to 0 \to \cdots\). These two functors are mutually inverse to each other.
\end{proof}

We now consider a simplicial enhancement of the category \(\mathsf{Ch}(\cA)\) of complexes over an additive category. To this end, we first observe that for any two chain complexes \(X\) and \(Y\), we may consider the \emph{internal mapping complex} \(\underline{\mathsf{Hom}}(X,Y)\) defined as \begin{multline*}
\underline{\mathsf{Hom}}(X,Y)_n = \prod_{k \in \Z} \mathsf{Hom}_\cA(X_k, Y_{n+k}), \\ \delta_n \colon \underline{\mathsf{Hom}}(X,Y)_n \to \underline{\mathsf{Hom}}(X,Y)_{n-1}, (f_k) \mapsto (d_{k+n}^Y \circ f_k - (-1)^n f_{k-1}\circ d^X).
\end{multline*}

\noindent For two fixed complexes \(X\) and \(Y\), the internal mapping complex \((\underline{\mathsf{Hom}}(X,Y), \delta)\) is a chain complex of abelian groups. The elements of \(\mathrm{ker}(\delta_n)\) are precisely chain maps \(X[n] \to Y\), where \(X[n]\) is the shifted chain complex defined by \(X[n]_k = X_{k - n}\) and differentials \(d[n]_k = (-1)^n d_{n-k}^X\). In a precise sense, \(\mathsf{Ch}(\cA)\) is enriched over the monoidal category of chain complexes of abelian groups. Base-changing along the  truncation functor \(\tau_{\geq 0} \colon \mathsf{Ch}(\mathsf{Mod}_\Z) \to \mathsf{Ch}_{\geq 0}(\mathsf{Mod}_\Z)\), and subsequently the Dold-Kan functor \(\mathsf{Ch}_{\geq 0}(\mathsf{Mod}_\Z) \to \mathsf{Fun}(\Delta^\op, \mathsf{Mod}_\Z)\), we may view the category \(\mathsf{Ch}(\cA)\) as enriched over simplicial abelian groups. Finally, using the forgetful functor from simplicial abelian groups to simplicial sets, and the fact that any simplicial abelian group is automatically a Kan-complex (see \cite[Remark 1.2.3.14]{HA}), the category \(\mathsf{Ch}(\cA)\) is a fibrant simplicial category. Applying the homotopy coherent nerve functor, we get an \(\infty\)-category \(\mathbf{Ch}(\cA)\) of chain complexes.

\noindent To define the derived \(\infty\)-category, we require the underlying category \(\cA\) to have more structure (for instance, to define homology and quasi-isomorphisms conveniently). For simplicity, let \(\cA\) be a Grothendieck abelian category; that is, it is an abelian category which is presentable, and the filtered colimit of monomorphisms is a monomorphism.

\begin{example}\label{ex:R-mod}
Let \(R\) be a commutative, unital ring. Then the category \(\mathsf{Mod}_R\) is a Grothendieck abelian category. 
\end{example}

\noindent We now take the Dwyer-Kan localisation at the chain maps \[W= \{f\colon X \to Y\}\vert\{H_n(f) \colon H_n(X) \overset{\cong}\to H_n(Y) \text{ for all }n\}\] inducing isomorphism in homology (ie, the \emph{quasi-isomorphisms}), and get an \(\infty\)-category \[\bD(\cA) = \mathbf{Ch}(\cA)[W^{-1}]\] that we call the \emph{derived \(\infty\)-category} of an abelian category \(\cA\).  Note in particular that the definition did not require that \(\cA\) be a \emph{Grothendieck} abelian category.

\begin{proposition}\label{prop:D-stable}
Let \(\cA\) be an abelian category. Then the derived \(\infty\)-category \(\bD(\cA)\) is a stable \(\infty\)-category. If furthermore \(\cA\) is a Grothendieck abelian category, then \(\bD(\cA)\) is a presentable \(\infty\)-category. 
\end{proposition} 

\begin{proof}
The \(\infty\)-category \(\mathbf{Ch}(\cA)\) with the distinguished classes \(W\) and \(I\) of degreewise monomorphisms (respectively, epimorphisms) defines an \(\infty\)-category with weak equivalences and cofibrations (respectively, fibrations). By Cisinski, \(\bD(\cA)\) has finite limits and colimits, and the localisation functor \(L \colon \mathbf{Ch}(\cA) \to \bD(A)\) is right and left exact. Consequently, \(\bD(\cA)\) inherits the structure of an additive category from \(\mathbf{Ch}(\cA)\).  To see that it is stable, we first observe that we have a pushout square 
\[
\begin{tikzcd}
L(X) \arrow{r}{} \arrow{d}{} & 0 \arrow{d}{} \\
0 \arrow{r}{} & L(\Sigma(X)),
\end{tikzcd}
\] which in turn follows from the fact that \(L\) preserves pushouts and the shift functor is by definition the \(\infty\)-categorical pushout of the diagram \(0 \leftarrow X \to 0\) in \(\mathbf{Ch}(\cA)\). This in particular shows that \(L(\Sigma(X)) \simeq \Sigma(L(X))\). By a dual argument involving pullbacks, we get \(L(\Omega(X)) \simeq \Omega(L(X))\). This shows that the unit \(L(X) \to \Omega\Sigma(L(X))\) and counit \(\Sigma \Omega(L(X)) \to L(X)\) are equivalences in \(\bD(A)\), as required. 

Finally, it remains to show that \(\bD(A)\) is presentable. The easiest way to see this is to use the fact that when \(\cA\) is a Grothendieck abelian category, the category \(\mathsf{Ch}(\cA)\) with weak equivalences as quasi-isomorphisms, cofibrations as termwise monomorphic chain maps, and fibrations as chain maps with the right lifting property with respect to acyclic cofibrations is a combinatorial model category. The underlying \(\infty\)-category of such a model category is presentable, and may be identified with \(\bD(\cA)\) (see \cite[1.3.5.13, 1.3.5.15, 1.3.4.22]{HA}).   
\end{proof}

\subsection{Dualisable and rigid categories}

In this subsection, we introduce the most important class of categories from the viewpoint of noncommutative geometry, namely, \emph{dualisable categories}. Recall that in any symmetric monoidal (\(\infty\))-category \(\mathcal{X}\), we call an object \(c \in \mathcal{X}\) \emph{dualisable} if there exists an object \(c^\vee \in \mathcal{X}\) with morphisms \[\mathrm{ev} \colon c^\vee \otimes c \to 1_{\cX}, \quad \mathrm{coev} \colon 1_{\cX} \to c \otimes c^\vee\] such that the compositions \[c \overset{\mathrm{coev} \otimes 1}\to c \otimes c^\vee \otimes c \overset{1 \otimes \mathrm{ev}}\to c, \quad c^\vee \overset{1 \otimes \mathrm{coev}}\to c^\vee \otimes c \otimes c^\vee \overset{\mathrm{ev} \otimes 1}\to c^\vee\] are homotopic to the identity. When \(\cX\) is \emph{closed} symmetric monoidal, then the dual object is forced by the tensor-Hom adjunction to be \(c^\vee \simeq \underline{\mathbf{Hom}}(c,1_\cX)\).

Recall that by Proposition \ref{prop:tensor-product}, the category \(\mathbf{Pr}_{st}^L\) has a closed symmetric monoidal structure.

\begin{definition}\label{def:bad-def-dualisable}
A dualisable \(\infty\)-category \(\bC\) is a dualisable object in the symmetric monoidal \(\infty\)-category \(\mathbf{Pr}_{st}^L\).  
\end{definition} 

Since \(\mathbf{Pr}_{st}^L\) is also closed, the dual \(\bC^\vee \simeq \mathbf{Fun}^L(\bC, \mathbf{Sp})\). Unfortunately, the definition is extrinsic in the sense that it requires working in the \(2\)-category of \(\infty\)-categories to witness the natural transformation witnessing the homotopy in the triangle identities. The following provides an \emph{intrinsic} characterisation of dualisable categories:

\begin{theorem}[Lurie]\label{thm:dualisable-char}
Let \(\bC \in \mathbf{Pr}_{st}^L\). Then the following are equivalent:

\begin{enumerate}
\item \(\bC\) is a dualisable \(\infty\)-category;
\item The colimit functor \(k \colon \mathbf{Ind}(\bC) \to \bC\) has a left adjoint;
\item \(\bC\) is \(\omega_1\)-compactly generated and the colimit functor \(k \colon \mathbf{Ind}(\bC^{\omega_1}) \to \bC\) has a fully faithful left adjoint;
\item \(\bC\) is the retract in \(\mathbf{Pr}^L\) of a compactly generated stable \(\infty\)-category.
\end{enumerate}

\end{theorem} 

\begin{proof}
We first prove the equivalence between (1) and (4). Let \(\bC\) be a dualisable category, and \(\bC^{\vee}\) its dual object in \(\mathbf{Pr}_{st}^L\). Since \(\bC\) is presentable, there is in particular a \(\kappa\) for which it is \(\kappa\)-compactly generated. As \(\bC\) has all colimits, we may take the ind-extension \(\mathbf{Ind}(\bC^\kappa) \to \bC\) of the inclusion \(\bC^\kappa \subseteq \bC\). This is a left Bousfield localisation, with right adjoint given by the \(\mathbf{Ind}_\kappa\)-extension of \(j \colon \bC^\kappa \to \mathbf{Ind}(\bC^\kappa)\) (see \cite[Proposition 2.1.26]{nkp}). Then \(\mathbf{Ind}(\bC^\kappa) \otimes \bC^\vee \to \bC \otimes \bC^\vee\) is also a left Bousfield localisation, so that it is in particular essentially surjective. By the universal property of spectra as the initial object in \(\mathbf{CAlg}(\mathbf{Pr}_{st}^L\), the coevaluation \(\mathbf{Sp} \to \bC \otimes \bC^\vee\) lifts to a functor \(\mathbf{Sp} \to   \mathbf{Ind}(\bC^\kappa) \otimes \bC^\vee\). Tensoring by \(\bC\) and using duality, we get a functor \(\bC \to \mathbf{Ind}(\bC^\kappa)\) lifting the identity on \(\bC\). This shows that \(\bC\) is the retract of a compactly generated category. Conversely, we first show that a compactly generated stable \(\infty\)-category is dualisable. To see this, consider \(\mathbf{Ind}(\bC_0)\) for a small, stable \(\infty\)-category \(\bC_0\). Then the \(\mathbf{Ind}\)-extension \[\mathbf{Ind}(\bC_0^\op) \otimes \mathbf{Ind}(\bC_0) \simeq \mathbf{Ind}(\bC_0^\op \times \bC_0) \to \mathbf{Sp}\] of \(\mathbf{Hom}_{\bC_0}(-,-) \colon \bC_0^\op \times \bC_0 \to \mathbf{Sp}\) gives the evaluation map.   For the coevaluation map \(\mathbf{Sp} \to \mathbf{Ind}(\bC_0^\op \times \bC_0)\), since the target is stable, a functor is the same as a choice of object in \(\mathbf{Ind}(\bC_0^\op \otimes \bC_0)\). These are in turn is equivalent to finite colimit-preserving functors \(\mathbf{An} \to \bC_0^\op \otimes \bC_0\), or equivalently, finite limit-preserving functors \((\bC_0 \times \bC_0^\op)^\op \to \mathbf{An}\), or equivalently, finite limit-preserving functors \(\bC_0 \times \bC_0^\op \to \mathbf{An}\) in both variables, which may be identified with the mapping space bifunctor \(\mathbf{Hom}_{\bC_0}(-,-)\). 

We now prove the equivalence between (4), (2) and (3). Suppose \(\bC\) is a compactly generated stable \(\infty\)-category, so that it is of the form \(\mathbf{Ind}(\bC_0)\) for some small \(\bC_0\). Then a left adjoint to \(\mathbf{Ind}(\mathbf{Ind}(\bC_0)) \to \mathbf{Ind}(\bC_0)\) is obtained by taking the ind-extension of the functor \(\bC_0 \to \mathbf{Ind}(\mathbf{Ind}(\bC_0))\). The general case of a retract of a compactly generated category reduces to the latter (see \cite[Lemma 2.3.22]{nkp}). This shows that (4) implies (2). The implication (2) to (3) is slightly more complicated: the idea is that if \(k \colon \mathbf{Ind}(\bC) \to \bC\) admits a left adjoint, then \(\bC\) is generated under colimits by objects of the form \[X \simeq \colim (X_0 \to X_1 \to \cdots),\] where each \(X_i \to X_{i+1}\) is a compact morphism (in the sense of Definition \ref{def:compact} above). By \cite[Lemma 2.3.18]{nkp}, the filtered colimits are exact in \(\bC\). Putting these facts together, \cite[Lemma 2.3.12]{nkp} implies that \(\bC\) is \(\omega_1\)-compactly generated, so that \(k \colon \mathbf{Ind}(\bC) \to \bC\) and therefore its left adjoint factor through \(\mathbf{Ind}(\bC^{\omega_1})\). Finally, it remains to prove that (3) implies (4). To see this, since \(\bC\) is \(\omega_1\)-compactly generated, the restricted Yoneda embedding \(\bC \to \mathbf{Ind}(\bC^{\omega_1})\) is a fully faithful right adjoint of the colimit functor \(k \colon \mathbf{Ind}(\bC^{\omega_1}) \to \bC\). It is now a general fact that if we have an adjoint triple, such as \((\hat{j}, k, j)\), then \(\hat{j}\) is fully faithful if and only if \(j\) is. This shows that \(k \hat{j} \simeq 1\), so that \(k \colon \mathbf{Ind}(\bC^{\omega_1}) \to \bC\) is a retraction.   
\end{proof}

We remark that there are further equivalent conditions, but we only provide these for what will follow later. To define the right class of functors between dualisable categories, we consider the following:

\begin{definition}\label{def:dualisable-functor}
Let \(F \colon \bC \to \bD\) be a left adjoint functor between dualisable categories, with right adjoint denoted by \(G\). We call \(F\) \emph{dualisable} if \(G\) commutes with filtered colimits. 
\end{definition}

Denote by \(\mathbf{Pr}^{L,\mathrm{dual}}\) the (non full) subcategory of \(\mathbf{Pr}_{st}^L\) generated by dualisable categories and dualisable functors as morphisms.

We end this section by describing a way to force a closed symmetric monoidal presentable \(\infty\)-category to be dualisable. In fact, this procedure takes such a category to a particularly well-behaved class of dualisable categories called \emph{rigid categories}. The definition of a rigid category will take some preparation:

\begin{definition}
Let \(\bC\) be a presentable stable \(\infty\)-category. 
\begin{enumerate}
\item A morphism \(f \colon X \to Y\) in \(\bC\) is called \textit{compact} if for any morphism \(Y \to Z = colim_{i \in I} Z_i\) into a filtered colimit, the composite \(X \to Z\) factorises through a finite stage \(X \to Z_i\). 
\item Assume additionally that \(\bC\) is symmetric monoidal (with unit denoted by \(1_\bC\)). We call a morphism \(f \colon X \to Y\) \textit{trace-class} if there is morphism \(1_\bC \to X^{\vee} \otimes Y\) such that the composition \(X \to X \otimes X^{\vee} \otimes Y \to Y\) is homotopic to \(f\). 
\end{enumerate}
\end{definition}

As we shall later see, this definition of trace-class morphism is intimately related to the notion of a trace-class operator between Banach or Hilbert spaces, used in operator algebra theory.

\begin{definition}\label{def:rigid}
Let \(\bC\) be a stable, presentable \(\infty\)-category with a closed symmetric monoidal structure. We call \(\bC\) \emph{rigid} if 
\begin{enumerate}
\item \(\bC\) is \emph{dualisable};
\item compact morphisms agree with trace-class morphisms. 
\end{enumerate}
\end{definition}


We end this section by describing a way to force categories to be rigid (and therefore dualisable). Denote the full subcategory of rigid \(\infty\)-categories in \(\mathbf{CAlg}(\mathbf{Pr}_{st}^L)\) by \(\mathbf{Rig}\). We now associate to a presentable stable \(\infty\)-category \(\bC\) a rigid category, using a very general construction due to Clausen-Gaitsgory-Rozenblyum. Let \(S\) be the collection of all compact morphisms in \(\bC\). 

\begin{proposition}[Clausen]\label{prop:dualisable-core}
There exists a dualisable \(\infty\)-category \((\bC, S)^{\mathrm{dual}}\) with a colimit-preserving functor to \(\bC\) such that 
\[
\mathbf{Fun}^{\mathrm{cs}}(\bD, (\bC, S)^{\mathrm{dual}}) \overset{\simeq}\to \mathbf{Fun}^{\mathrm{cpt}}(\bD, \bC) 
\] for all \(D \in \mathbf{Pr}_{st}^{L, \mathrm{dual}}\), where the left hand side denotes strongly continuous functors, and the right hand side colimit-preserving functors taking compact morphisms in \(\bD\) to \(S\).  
\end{proposition}

\begin{proof}
Take the full subcategory of \(\mathbf{Ind}(\bC)\) generated under colimits by basic objects given by systems of maps \(X_\lambda \to X_\mu\) in \(S\) such that \(\lambda < \mu\). This is \((\bC, S)^{\mathrm{dual}}\). 
\end{proof}

\begin{proposition}\cite[Theorem 4.4.17]{nkp}
The inclusion of rigid categories \(\mathbf{Rig} \subseteq \mathbf{CAlg}(\mathbf{Pr}_{st}^L)\) into stable, presentable symmetric monoidal \(\infty\)-categories admits a right adjoint, namely, the functor \[\bC \mapsto \bC^{\mathrm{rig}} := (\bC, S_t)^{\mathrm{dual}},\] where \(S_t\) denotes the collection of trace-class morphisms. 
\end{proposition}

\begin{proof}
The construction in Proposition \ref{prop:dualisable-core} may be done with any collection \(S\) of morphisms that form a so-called precompact ideal. It is easy to see that trace-class maps form such an ideal. Furthermore, the dualisable core construction forces trace-class maps to be compact, thereby forcing rigidity. 
\end{proof}

\section{Functional analysis revisited - bornologies and condensed mathematics}\label{chapter3}

In this chapter, we introduce two frameworks for functional analysis, starting with very different foundations. The first of these formalises the notion of \emph{bounded subsets}, capturing the idea that specifying a sequence of seminorms defining the topology of a locally convex vector space is equivalent to the unit balls of the seminorms. The second formalism (\emph{condensed mathematics}) builds on the idea that the topology on a vector space is constructed out of maps of test profinite sets into it.    

\subsection{Bornological analysis}

The most general setup of bornological starts with a set \(X\) with a collection of subsets \(\mathcal{B}_X \subseteq 2^X\) called bounded sets satisfying the following axioms:

\begin{enumerate}
\item finite subsets are bounded;
\item if \(S\) and \(T \in \bdd_X\), then \(S \cup T\in \bdd_X\);
\item if \(S \subseteq T\), and \(T \in \bdd_X\), then \(S \in \bdd_X\).  
\end{enumerate}

A set-map \(f \colon X \to Y\) between bornological sets is called \emph{bounded} if \(f(S) \in \bdd_Y\) for each \(S \in \bdd_X\). Bornological sets with bounded maps form a category denoted by \(\mathsf{Born}(\mathsf{Set})\), which we may identify with filtering unions of sets. More precisely, let \(\mathsf{Ind}^m(\mathsf{Set})\) be the full subcategory of \(\mathsf{Ind}(\mathsf{Set})\) consisting of filtered diagrams of sets that admit a presentation with injective transition maps.   

\begin{proposition}\label{prop:bornological-sets}
We have an equivalence \(\mathsf{Born}(\mathsf{Set}) \simeq \mathsf{Ind}^m(\mathsf{Set})\). 
\end{proposition}

\begin{proof}
Let \((X,\bdd_X)\) be a bornological set. Then the collection of its bounded subsets \(\bdd_X\) with inclusions as morphisms is a filtered category, so that we get a functor \[\mathsf{Born}(\mathsf{Set}) \to \mathsf{Ind}(\mathsf{Set}), \quad (X,\bdd_X) \mapsto ``colim_{S \in \bdd_X}" S.\] This functor is fully faithful, with essential image inductive systems with injective structure maps. The functor in the other direction is taking inductive limits with the bornology where a subset is bounded if and only if it is contained in one of the terms of the inductive system. 
\end{proof}

Proposition \ref{prop:bornological-sets} shows that the category of bornological sets is in some sense too large for functional analysis. For instance, complex vector spaces with a collection of bounded subsets satisfying the axioms above include non-convex topological vector spaces, which are unwieldy. To rectify this, one requires two further axioms on a complex vector space \(X\), namely:

\begin{itemize}
\item if \(c>0\), then \(c \cdot S \in \bdd_X\) whenever \(S \in \bdd_X\);
\item the disked hull of a bounded set is bounded. 
\end{itemize}

A \(\C\)-vector space with a bornology \(\bdd_X\) and the additional axioms above is called a \emph{convex} bornological vector space. Since we mostly consider convex bornologies anyway, we will drop this adjective. In fact, the theory of bornological \(\C\)-vector spaces extends to a theory of functional analysis over any base ring \(R\) that is complete and normed\footnote{For us a norm is a function \(\rho \colon R \to R_{\geq 0}\) that is nondegenerate, satisfies the triangle inequality \(\rho(a+b) \leq \rho(a) + \rho(b)\), and \(\rho(ab) \leq C\rho(a) \rho (b)\), for some \(C > 0\).}, that is, a \emph{Banach ring}.

We now provide a synthetic description of bornological modules over a Banach ring \(R\), following \cite{Kelly-M}. To fix intuition, one should think of three classes of examples: (1) \(\R\) and \(\C\) with the Euclidean norm (\emph{archimedean fields}); (2) \(\Q_p\) and \(\C_p\) with the \(p\)-adic norm (\emph{nonarchimedean fields}); (3) \(\Z\) with the Euclidean norm and \(\Z_p\) with the \(p\)-adic norm, which we think of as \emph{global} base rings for analytic geometry.  A \emph{semi-normed} (resp. \emph{normed}) \(R\)-module is an \(R\)-module with a semi-norm (resp. norm). We call a normed \(R\)-module a \emph{Banach} \(R\)-module if it is complete with respect to its norm. We call a (semi)-normed (resp. Banach) \(R\)-module \emph{nonarchimedean} if its norm satisfies the  ultra-metric property. Implicitly, whenever the base Banach ring \(R\) is nonarchimedean, we will assume that the Banach modules over such rings are nonarchimedean as well. An \(R\)-linear map \(f \colon M \to N\) between semi-normed \(R\)-modules is called \emph{bounded} if there is a \(C>0\) such that \(\rho_{N}(f(m)) \leq C\rho_M(m)\) for all \(m \in M\), where \(\rho_M\) and \(\rho_N\) are the seminorms on \(M\) and \(N\), respectively.  \footnote{To avoid set-theoretic issues and to eventually get a presentable category, we will assume that the underlying sets of our (semi)-normed and Banach spaces are of size at most a strongly inaccessible cardinal \(\kappa\).}

Denote by \(\mathsf{Norm}_R^{1/2}\) (resp \(\mathsf{Norm}_R\) and \(\mathsf{Ban}_R)\) the categories of semi-normed, normed and Banach \(R\)-modules, and bounded \(R\)-linear maps. These are symmetric monoidal categories with the (completed) projective tensor product \(- \haotimes - \colon \mathsf{Ban}_R \times \mathsf{Ban}_R \to \mathsf{Ban}_R\) defined as the (completion of the) algebraic tensor product \(M \otimes_R N\) with (respect to) the seminorm \[\varrho(x) \defeq 
\begin{cases}
\inf \{\sum_{i=0}^n \varrho(m_i)_M \varrho(n_i)_N : x = \sum_{i=0}^n m_i \otimes n_i\} \quad \text{ in the archimedean case} \\
\inf \max\{\varrho(m_i)_M \varrho(n_i)_N  : x = \sum_{i=0}^n m_i \otimes n_i\} \quad \text{ in the nonarchimedean case}. 
\end{cases} 
\] The (completed) projective tensor product is universal for bounded bilinear maps into a Banach \(R\)-module, and there is an internal Hom given by the \(R\)-module \[\mathsf{Hom}(M,N) = \{T \colon M \to N \text{ }R\text{ -linear }:\varrho(T)< \infty\}\] with the operator norm. The functor \(M \haotimes -\) is left adjoint to \(\mathsf{Hom}(M,-)\), or in other words, \(\mathsf{Norm}_R^{1/2}\), \(\mathsf{Norm}_R\) and \(\mathsf{Ban}_R\) are \emph{closed} symmetric monoidal categories. 

The categories of (semi-)normed and Banach \(R\)-modules are finitely complete and finitely cocomplete, but lack infinite products and coproducts. Consequently, they are insufficient for the purposes of analytic geometry. For instance, smooth functions and overconvergent analytic functions on a smooth manifold or complex analytic spaces do not lie in the category of Banach spaces. One way to address the lack of (co)limits is to take the filtered cocompletion \(\mathsf{Ind}(\mathsf{Ban}_R)\) of Banach (or (semi-)normed) \(R\)-modules. This is still a closed symmetric monoidal category, but is now a presentable category. However, the category \(\mathsf{Ind}(\mathsf{Ban}_R)\) is still larger than what is strictly \emph{needed}. Indeed, it is sometimes more convenient to work with a smaller \emph{concrete} category \(\mathsf{Ind}^s(\mathsf{Ban}_R)\) of inductive systems of Banach \(R\)-modules with \emph{monomorphic} structure maps. This category is closed symmetric monoidal, and continues to remain complete and cocomplete. We call this the category of \emph{formal}, complete bornological \(R\)-modules. When the base Banach ring is a nontrivially valued Banach field, we recover bornological \(R\)-modules defined as modules with a bornology in the following sense:

\begin{definition}\label{def:bornological-module}
Let \(R\) be a Banach ring with a nontrivial norm. A \emph{bornological \(R\)-module} \(M\) is a module with a bornology such that the addition and scalar multiplication maps \[+ \colon M \times M \to M, \quad R \times M \to M\] are bounded functions. 
\end{definition} 

To define complete bornological modules, we first describe how bornologies arise from modules over a Banach ring. For simplicity, we only do this for Banach rings that are fields which are complete with respect to a nontrivial norm. Given a module \(M\) over such a field, we may define for each nonempty subset (which we may without loss of generality take to be discs)  \(S \in \bdd_M\), the seminorm \[\rho_S(x) \defeq \inf_{\lambda \in R}\{\vert \lambda \vert : x \in \lambda S\},\] called the \emph{gauge seminorm}. The \(R\)-module \(M_S\) generated by \(S\) is a semi-normed \(R\)-module, whose unit ball is precisely \(S\).  We call a bornology on an \(R\)-module \emph{convex} if the collection of bounded discs is cofinal in the original bornology (ordered by inclusion). We shall always assume that our bornologies are convex, and drop this adjective from here onwards. A bornological \(R\)-module is called \emph{separated} if for every bounded disc \(S\), the associated seminorm \(\rho_S\) is a norm. Finally, a bornological \(R\)-module \(M\) is called \emph{complete} if for every bounded disc \(S\), the associated seminormed space \((M_S, \rho_S)\) is a Banach \(R\)-module. Denote by \(\mathsf{Born}_R\) and \(\mathsf{CBorn}_R\) the categories of separated and complete bornological \(R\)-modules.

To relate the concrete definition of bornologies above with inductive systems, we observe that there are canonical functors 

\[
\mathsf{diss} \colon \mathsf{Born}_R \to \mathsf{Ind}(\mathsf{Norm}_R), \quad \mathsf{CBorn}_R \to \mathsf{Ind}(\mathsf{Ban}_R), 
\] taking a (complete) bornological \(R\)-module \(M\) to the filtered system \(``\mathrm{colim}" (M_S, \rho_S)\) of (complete) normed \(R\)-modules. We call this the \emph{dissection} functor.

\begin{theorem}\cite[Proposition 2.14]{kelly2022analytic}\label{thm:dissection}
Let \(R\) be a nontrivially valued Banach field. Then the dissection functor \(\mathrm{diss} \colon \mathsf{Born}_R \to \mathsf{Ind}(\mathsf{Norm}_R)\) is fully faithful, and its essential image may be identified with formal bornological \(R\)-modules. Similar claims hold for \(\mathsf{CBorn}_R\) and \(\mathsf{Ind}(\mathsf{Ban}_R)\).  
\end{theorem}

Using Theorem \ref{thm:dissection}, we will identify the categories \(\mathsf{CBorn}_R\) and \(\mathsf{Ind}^m(\mathsf{Ban}_R)\) from this point onwards. Next, we see that these categories contain the categories \(\mathsf{TVS}_R\) and \(\mathsf{Fr}_R\) of metrisable and complete metrisable locally convex vector spaces as full subcategories, when \(R = \C\) and \(\Q_p\). Consider the functor \(\mathsf{vN} \colon \mathsf{TVS}_R \to \mathsf{Born}_R\), that associates to a locally convex topological vector space the bornology where a subset \(S\) is bounded if and only if \(\varrho(S)\) is a bounded subset of \(\R_{>0}\), for every seminorm \(\varrho\) defining the topology. We then have

\begin{proposition}\label{prop:fully-faithful}\cite{meyer2003bornological}
The functor \(\mathsf{vN} \colon \mathsf{TVS}_R \to \mathsf{Born}_R\) is faithful. It restricts to a fully faithful functor \[\mathsf{vN} \colon \mathsf{Fr}_R \to \mathsf{CBorn}_R.\]
\end{proposition}

We end this subsection by describing the various categories thus introduced:

\[
\begin{tikzcd}
\mathsf{CBorn}_R \arrow{r}{\mathrm{diss}} \arrow{d}{\subset} & \mathsf{Ind}(\mathsf{Ban}_R) \arrow{d}{\subset} \\
\mathsf{Born}_R \arrow{r}{\mathrm{diss}}  & \mathsf{Ind}(\mathsf{Norm}_R),
\end{tikzcd}
\] where inclusions are right adjoint to the completion functor \(\mathsf{Born}_R \to \mathsf{CBorn}_R\) (resp. \(\mathsf{Ind}(\mathsf{Norm}_R) \to \mathsf{Ind}(\mathsf{Ban}_R)\)), defined by taking Banach space completions in the representation of a bornological \(R\)-module as a filtered union of normed \(R\)-modules as described above.

\subsubsection{Homotopy theory for bornological modules}

Recall that an additive category with kernels and cokernels is called \emph{quasi-abelian} if it is stable under pullbacks (resp. pushouts) of cokernels (resp. kernels) by arbitrary morphisms. It is easy to see that \(\mathsf{Norm}_R\), \(\mathsf{Ban}_R\), \(\mathsf{Born}_R\), \(\mathsf{CBorn}_R\), \(\mathsf{Ind}(\mathsf{Norm}_R)\), and \(\mathsf{Ind}(\mathsf{Ban}_R)\) are quasi-abelian categories. Being a quasi-abelian category implies that these categories have well-behaved derived categories, defined by taking the Verdier quotient of the triangulated homotopy category of unbounded chain complexes by the thick subcategory of exact chain complexes. Furthermore, being quasi-abelian, they have left and right \(t\)-structures, whose left hearts we denote by \(\mathsf{LH}(-)\).  Recall that a quasi-abelian category is said to have \emph{enough projectives} if for every object \(M\), there is a cokernel \(P \to M\) from a projective object \(P\).

\begin{theorem}\label{thm:elementary}\cite{Kelly-M}
Let \(R\) be a Banach ring. 
\begin{enumerate}
\item The category \(\mathsf{Ban}_R\) has enough projectives, and the tensor product of projectives is projective;
\item The category \(\mathsf{Ind}(\mathsf{Ban}_R)\) is an elementary quasi-abelian category (in the sense of \cite{Schneiders:Quasi-Abelian});
\item The categories \(\mathsf{LH}(\mathsf{Ind}(\mathsf{Ban}_R))\) and \(\mathsf{LH}(\mathsf{Ind}^m(\mathsf{Ban}_R))\) are Grothendieck abelian categories;
\item We have a monoidal equivalence of categories \[\mathsf{D}(\mathsf{Ind}(\mathsf{Ban}_R)) \simeq \mathsf{D}(\mathsf{LH}(\mathsf{Ind}(\mathsf{Ban}_R))) \simeq \mathsf{D}(\mathsf{LH}(\mathsf{Ind}^s(\mathsf{Ban}_R))) \simeq \mathsf{D}(\mathsf{Ind}^s(\mathsf{Ban}_R)).\] 
\end{enumerate}
\end{theorem}

\begin{proof}
Let \(M\) be a Banach \(R\)-module. Then the Banach \(R\)-module \(l^1(M)\) is a projective Banach \(R\)-module with an epimorphism to \(M\), proving (1). To see that the completed projective tensor product of two projectives is projective, we first use the standard identification \(l^1(X) \haotimes l^1(Y) \cong l^1 (X \times Y)\), and then the fact that any projective object in \(\mathsf{Ban}_R\) is a retract of an \(l^1\)-space, and the fact that the tensor product commutes with finite colimits. Statement (2) follows formally from the first part of (1). That the left hearts are Grothendieck abelian is a consequence of \cite[Definition 2.2]{bode2021six}. The monoidal derived equivalence between a symmetric monoidal quasi-abelian category and its left heart is purely formal (see \cite[Corollary 2.11]{bambozzi2020sheafyness}). The monoidal derived equivalence between inductive systems and bornologies is induced by the adjoint pair \[\varinjlim \colon \mathsf{Ind}(\mathsf{Ban}_R) \rightleftarrows \mathsf{CBorn}_R \colon \mathrm{diss},\] whose proof is spelled out in \cite[Proposition 3.16]{bambozzi2020sheafyness}. 
\end{proof}





We end this section with the \(\infty\)-categorical enhancements of the derived categories defined so far. By Theorem \ref{thm:elementary} and \cite[Theorem 5.3.58]{kelly2016homotopy}, since \(\mathsf{Ind}(\mathsf{Ban}_R)\) is an elementary quasi-abelian category, the category \(\mathsf{Ch}(\mathsf{Ind}(\mathsf{Ban}_R))\) of unbounded complexes of ind-Banach \(R\)-modules has a \emph{projective model structure}, whose weak equivalences and fibrations are the quasi-isomorphisms and termwise cokernels, respectively. Furthermore, the model structure is stable, simplicial, combinatorial and projectively monoidal. Consequently, localising at the weak equivalences yields \[\bD(\mathsf{Ind}(\mathsf{Ban}_R)) := N(\mathsf{Ch}(\mathsf{Ind}(\mathsf{Ban}_R)))[W^{-1}]\] a presentably stable symmetric monoidal \(\infty\)-category.

\subsubsection{Categories of nuclear modules}

For a ring \(R\), one defines localising invariants \(E(R) := E(\mathsf{Perf}(R))\). Now given a Banarch ring \(R\) and a normed set \(X\), one has a split short exact sequence 
\[
0 \to l^1(\N,R) \to l^1(\N,R) \to R \to 0,
\] so that the functor \[F \colon \mathsf{Ban}_R \to \mathsf{Ban}_R, \quad A \mapsto l^1(\N,A)\] satisfies \(F(A) = A \oplus F(A)\), leading to Eilenberg-Swindle (so that \(K\)-theory vanishes). Consequently, we cannot use the category of compact objects in the derived category \(\bD(R)\) of a Banach ring as the input category for localising invariants. The right replacement of perfect complexes in functional analytic contexts is the category of suitably \emph{nuclear modules}, which is a rigid category.

\begin{definition}\label{nuclear}
Let \(R\) be a Banach ring. 
\begin{itemize}
\item A map \(f \colon X \to Y\) of Banach \(R\)-modules is called \emph{trace-class} if it lies in the image of the map \(\mathsf{Hom}_R(R,X^\vee \haotimes Y) \to \mathsf{Hom}_R(X,Y)\). Unravelling the definition of the completed projective tensor product, one immediately sees that a map is trace-class if and only if it is of the form \(f(x) = \sum_{n=0}^\infty \lambda_n x_n'(x) y_n\) for an \(l^1\)-summable sequence \((\lambda_n)\), and bounded sequences \((x_n') \in X^\vee\) and \((y_n)\in Y\);
\item A bornological \(R\)-module is called \emph{nuclear} if it is an \(\N\)-indexed colimit of Banach \(R\)-modules with injective, trace-class structure maps;
\item A map \(f\colon X_0 \to X_1\) is called \emph{factorisably} trace-class if it extends over a \([0,1] \cap \Q\)-indexed diagram, all of whose nonidentity morphisms are trace-class. 
\item A bornological \(R\)-module is called \emph{strongly nuclear} if it is a \(\Q\)-indexed colimit of Banach \(R\)-modules with injective, trace-class transition maps.  Equivalently, a bornological \(R\)-module is strongly nuclear if it is an \(\N\)-indexed colimit with injective, factorisably trace-class transition maps.  
\end{itemize}
\end{definition}

The definitions of (factorisably) trace-class maps and (strongly) nuclear modules make sense in the generality of any compactly generated, stable \(\infty\)-category \(\bC\) with a closed symmetric monoidal structure and compact tensor unit. 

\begin{definition}\label{def:nuclear}
Let \(\bC\) be a compactly generated, closed symmetric monoidal, stable \(\infty\)-category with compact unit object \(1\). 

\begin{itemize}
\item A map \(f \colon X \to Y\) is called \emph{trace-class} if it is in the image of the canonical map \(\pi_0(\mathbf{Hom}(1, M^\vee \otimes N)) \to \mathbf{Hom}(M,N)\);
\item An object \(X \in \bC\) is called \emph{basic nuclear} if it is an \(\N\)-indexed colimit of compact objects with trace-class transition maps;
\item An object is called \emph{\(\Q\)-basic nuclear} if it is a \(\Q\)-indexed colimit of compact objects with trace-class transition maps. 
\end{itemize}
\end{definition}

Denote by \(\mathbf{Nuc}(\bC)\) (resp. \(\mathbf{Nuc}^\infty(\bC)\)) the full subcategory of \(\bC\) generated under colimits by basic (resp. \(\Q\)-basic) nuclear objects.

In this generality, we have the following:

\begin{proposition}\label{prop:abstract-nuc}\cite{CS3}
Let \(\bC\) be as in Definition \ref{def:nuclear}. 
\begin{enumerate}
\item The category \(\mathbf{Nuc}(\bC)\) is \(\omega_1\)-compactly generated and stable. Then \(\omega_1\)-compact objects are precisely the basic nuclear objects. It is rigid if and only if every trace-class map is factorisably trace-class;
\item The essential image of the inclusion \(\bC^\mathrm{rig} \to \bC\) coincides with the full subcategory \(\mathbf{Nuc}^{\infty}(\bC)\) of infinitely nuclear objects of \(\bC\).  In particular, \(\mathbf{Nuc}^\infty(\bC)\) is rigid, and the \(\omega_1\)-compact objects are precisely the \(\Q\)-basic nuclear objects. 
\end{enumerate}
\end{proposition}

We apply Definition \ref{def:nuclear} and Proposition \ref{prop:abstract-nuc} to our category \(\mathbf{D}(R)\). To relate the categories \(\mathbf{Nuc}(R) = \mathbf{Nuc}(\bD(R))\) and \(\mathbf{Nuc}^{\infty}(R) = \mathbf{Nuc}^\infty(\bD(R))\) to stable infinity categories generated under suitable colimits by nuclear bornological modules, we have the following:

\begin{proposition}\label{prop:nuclear-borno}\cite{Kelly-M} The following holds
\begin{enumerate}
\item The category \(\mathsf{Nuc}(\Q_p) \subset \mathsf{CBorn}(\Q_p)\) is an exact full subcategory, with exact structure restricted from \(\mathsf{CBorn}(\Q_p)\);
\item The category \(\mathbf{Nuc}(\bD(\Q_p))\) is rigid, and we have an equivalence of categories \(\mathbf{Nuc}(\bD(\Q_p)) \simeq \bD(\mathsf{Nuc}(\Q_p))\);
\item The category \(\mathsf{Nuc}^\infty(\C) \subset \mathsf{CBorn}(\C)\) of strongly nuclear bornological \(\C\)-vector spaces is an exact full subcategory, with exact structure induced from \(\mathsf{CBorn}(\C)\);
\item The category \(\bD(\mathsf{Nuc}^\infty(\C))\) is rigid and we have a fully faithful functor \[\bD(\mathsf{Nuc}^\infty(\C)) \to \mathbf{Nuc}^\infty(\bD(\C)).\]
\end{enumerate}
\end{proposition}

The factorisably trace-class maps between Banach \(\C\)-vector spaces are precisely the \emph{order \(0\)-trace-class} maps defined in \cite{grothendieck1955produits}, characterised in a manner similar to trace-class maps, but with sequences of rapid decay in place of \(l^1\)-summable sequences. In the nonarchimedean setting, the distinction between trace-class and factorisably trace-class maps disappears as the former are characterised by a null-sequence of singular values representing a trace-class maps.

\subsection{Condensed mathematics}

In this section, we describe another convenient formalism for functional analysis, namely, Clausen and Scholze's condensed mathematics (\cite{CS1,CS2}). The starting point and the building blocks of the theory is the category of \emph{light profinite sets} defined as follows:

\begin{definition}\label{def:light-profinite}[Light profinite sets]
A \emph{light profinite} set is a functor \(F \colon I^\op \to \mathsf{Fin}\), where \(I\) is a countable directed set, and \(\mathsf{Fin}\) is the category of finite sets. We shall denote such an inverse system by \(``lim_{i \in I}" F_i\), where \(F_i = F(i)\), and the actual inverse limit in topological spaces by \(\varprojlim_{i \in I} F_i\).   
\end{definition}

\begin{example}
We consider two canonical examples of light profinite sets:

\begin{enumerate}
\item Consider the set \(\N \cup \{\infty\} = \varprojlim_n \{0,1, \dotsc, \infty\}\) where the structure maps are the obvious base-point preserving projections \[\{0,1,\dotsc, n, \infty\} \twoheadrightarrow \{0,1, \dotsc, n-1, \infty\}\] that maps \(n\) to \(\infty\) and the other numerals identically.  
\item The Cantor set \(\{0,1\}^\N = \varprojlim_n \{0,1\}^n\).
\end{enumerate}
\end{example}

Light profinite sets form a category \(\mathsf{Pro}_{\N}(\mathsf{Fin})\) with morphisms given by \[\mathsf{Hom}(``\lim_i" X_n, ``\lim_j" Y_n) := \lim_{j \in J} \mathrm{colim}_{i \in I}\mathsf{Hom}(X_i,Y_j).\] 

By an appropriate modification of the Stone duality, we have the following:

\begin{theorem}[Light Stone duality]
The following categories are equivalent:
\begin{enumerate}
\item \(\mathsf{Pro}_{\N}(\mathsf{Fin})\);
\item Metrizable, totally disconnected compact Hausdorff spaces;
\end{enumerate}
\end{theorem}

\begin{proof}
Let \(X = ``\varprojlim_{i \in J}" X_i\) be a light profinite set. Its inverse limit is a closed subspace of the product \(\prod_{i \in J} X_i\), which is compact as each of the \(X_i\)'s is finite and discrete. To see that it is metrisable, we equip the resulting inverse limit with the metric \(d_n(x,y) = \inf_n \{2^{-n} \vert p_n(x) = p_n(y)\}\), which generates the topology of the inverse limit. Conversely, given a metrisable, totally disconnected compact Hausdorff space \(X\), we may write this as an inverse limit as follows: for each \(n\in \N\), consider the equivalence relation where \(x \simeq_n y\) if and only if \(d(x,y) \leq 2^{-n}\). Then \(X_n := X/ \simeq_n\) is a finite discrete space and its inverse limit is homeomorphic to \(X\). The required functor is then given by the assignment \(X \mapsto ``\varprojlim_{n \in \N}" X_n\); we leave it for the reader to check that this assignment is indeed functorial. This establishes the equivalence between (1) and (2). 
\end{proof}

By virtue of the Theorem above, we often identify (light) profinite sets \(X = ``\lim_{i \in I}" X_i\) with the underlying topological space defined by taking the inverse limit of the finite discrete sets \(X_i\) in the pro-system \(X\).

\begin{definition}\label{def:light-condensed}
A presheaf \(X \colon \mathsf{Pro}_{\N}(\mathsf{Fin})^\op \to \mathsf{Sets}\) is called a \emph{light condensed set} if 
\begin{enumerate}
\item \(X(\phi) = \{\mathrm{pt}\}\);
\item \(X(S \coprod T) \overset{\simeq}\to X(S) \times X(T)\);
\item for a surjection \(T \to S\), we have an equaliser diagram \[X(S) \overset{\simeq}\to \mathrm{Eq}(X(T) \rightrightarrows X(T \times_S T))\] in the category of sets.   
\end{enumerate}
\end{definition}

In other words, a light condensed set is a sheaf for the Grothendieck topology on \(\mathsf{Pro}_\N(\mathsf{Fin})\) whose covers are finite families \((S_i \to S)\) of maps that are jointly surjective. Denote by \(\mathsf{Cond}_\N(\mathsf{Set})\) the category of light condensed sets.

\begin{example}
Let \(X\) be a fixed topological space. Then the assignment \(\underline{X} \defeq S \mapsto C(S,X)\) on light profinite sets is a light condensed set. We thus get a functor from the category of topological spaces to light condensed sets, taking \(X \mapsto \underline{X}\). 
\end{example}

Given a condensed set \(X\), we may extract out of it a set in the obvious manner by evaluating it at a point \(X(*)\).

\begin{proposition}\label{prop:condensed-top}
The functor \[\mathsf{Top} \to \mathsf{Cond}_\N(\mathsf{Set}), \quad X \mapsto \underline{X}\] is right adjoint to the functor \(X \mapsto X(*)\), where \(X(*)\) is equipped with the quotient topology from \(\coprod_{S \to X} S \to X(*)\). This functor is fully faithful when restricted to the full subcategory of sequential topological spaces. 
\end{proposition}

\begin{proof}
Let \(X\) be a condensed set and denote by \(X(*)_{\mathrm{top}}\) the topology on the underlying set \(X(*)\) induced by condensed sets mapping to \(X\). An application of the Yoneda lemma then shows that for any topological space \(Y\), continuous maps \(X(*)_{\mathrm{top}} \to Y\) are equivalent to maps from (light) profinite sets \(S \to X\), such that the composition \(S \to X(*) \to Y\) is continuous, which in turn is equivalent to a map \(X \to \underline{Y}\) of (light) condensed sets. 

That the functor \(\mathsf{Top} \to \mathsf{Cond}_\N(\mathsf{Set})\) is faithful is easy to see. Since this functor has a left adjoint, it will be fully faithful on a full subcategory \(\mathcal{B}\) if the counit map \(\underline{X}(*)_{\mathrm{top}} \to X\) is a homeomorphism of topological spaces for \(X \in \mathcal{B}\). Let us investigate the quotient map \(\coprod_{S \to X} S \to X\). Using the fact that any light profinite set is a quotient of the Cantor set, this quotient is equivalent to the quotient \(\coprod_{C \to X} C \to X\), where \(C\) is the Cantor set. But since the Cantor set \(C\) is itself a sequential topological space, we may write refine the quotient above to a quotient map \(\coprod_{\N \cup \{\infty\} \to X} \N \cup \{\infty\} \to X\). This is equivalent to \(X\) being a sequential topological space. 
\end{proof}

By Proposition \ref{prop:condensed-top}, we may embed several known categories of topological spaces into (light) condensed sets. 


We now consider linearisations of the category of light condensed sets. Since \(\mathsf{Cond}_\N(\mathsf{Set})\) is a topos, we may speak of abelian group objects inside it. Denote this category by \(\mathsf{Cond}_\N(\mathsf{Ab})\). We may turn a light condensed set into a (light) condensed abelian group via the following free-construction that works in any topos: consider the forgetful functor \(\mathsf{Cond}_\N(\mathsf{Ab}) \to \mathsf{Cond}_\N(\mathsf{Set})\). This has a left adjoint given by the assignment \(\mathcal{F} \mapsto \Z[\mathcal{F}] \colon L(S \mapsto \Z[\mathcal{F}(S)])\), where \(L\) denotes sheafification. In particular, by the Yoneda lemma, for any light profinite set \(S\) and a condensed abelian group \(M\), we have a bijection \[\mathsf{Hom}_{\mathsf{Cond}_\N(\mathsf{Ab})}(\Z[S], M) \cong M(S)\] that is natural in light profinite sets.

Again, as in any sheaf category, abelian sheaves form a Grothendieck abelian category.  Furthermore, we may define the tensor product of two condensed abelian groups as follows: for \(M\), \(N \in \mathsf{Cond}_\N(\mathsf{Ab})\), \[M \otimes N \colon L(S \mapsto M(S) \otimes_\Z N(S)).\]  For a condensed abelian group \(N \in \mathsf{Cond}_\N(\mathsf{Ab})\), the functor \(- \otimes N \colon \mathsf{Cond}_\N(\mathsf{Ab}) \to \mathsf{Cond}_\N(\mathsf{Ab})\) is left adjoint to the internal Hom functor defined by \[\underline{\mathsf{Hom}}(M,N)(S) \defeq \mathsf{Hom}_{\mathsf{Cond}(\mathsf{Ab})}(\Z[S] \otimes M, N).\]

\begin{example}\label{ex:convergent-sequence}
Of particular relevance is the free condensed abelian group \(\Z[\N \cup \{\infty\}]\), called the \emph{condensed abelian group classified by a convergent sequence}. More elaborately, if \(M\) is a condensed abelian group, then by the Yoneda lemma, maps \(\Z[\N \cup \{\infty\}] \to M\) are equivalent to \(M(\N \cup \{\infty\})\), so that if \(M = \underline{A}\) is a topological abelian group, then \(\underline{A}(\N \cup \{\infty\}) = C(\N \cup \{\infty\}, A)\), which is nothing but convergent sequences in \(A\).  
\end{example}

\begin{example}\label{ex:null-sequence}
The second important example of a condensed abelian group is the one which classifies null sequences. Consider the quotient \(\Z[\N \cup \{\infty\}]/ \Z[\{\infty\}]\) taken in the category of condensed abelian groups. Intuitively, for a condensed abelian group \(M\), a map of condensed abelian groups \(\Z[\N \cup \{\infty\}]/ \Z[\{\infty\}] \to M\) corresponds to a map \(\N \cup \{\infty\} \to M\) taking \(\infty \mapsto 0\). It is then an easy exercise to verify that if \(\underline{A}\) is a topological abelian group (viewed as a condensed abelian group), then \(\mathsf{Hom}_{\mathsf{Cond}_\N(\mathsf{Ab})}(\Z[\N \cup \{\infty\}]/\Z[\{\infty\}],A) \cong C_0(\N,A)\). 
\end{example}

As \(\mathsf{Cond}_\N(\mathsf{Ab})\) form a Grothendieck abelian category, we may form its derived category \(\mathsf{D}(\mathsf{Cond}_\N(\mathsf{Ab}))\), and finally its derived \(\infty\)-category enhancement \[\mathbf{D}(\mathsf{Cond}_\N(\mathsf{Ab})) = N(\mathsf{Ch}(\mathsf{Cond}_\N(\mathsf{Ab}))[W^{-1}]),\] by localising at the quasi-isomorphisms \(W\). Summarily, we apply to general result of Proposition \ref{prop:D-stable} to get the following:

\begin{theorem}\label{thm:condensed-stable}
The category \(\mathbf{D}(\mathsf{Cond}_\N(\mathsf{Ab}))\) is a presentable, stable \(\infty\)-category. It inherits a symmetric monoidal structure from light condensed abelian groups. 
\end{theorem}

\subsubsection{Completeness from the condensed viewpoint}

So far we have defined a category that contains sequential topological abelian groups as a full subcategory. In this subsection, we consider a full subcategory of condensed abelian groups of \emph{complete} objects, that correspond to the complete sequential topological abelian groups.

Recall that an object in an exact catergory (so in particular, an abelian category) is projective if and only if \(\mathsf{Ext}^i(M,P) = 0\) for all \(M\) and \(i>0\). But since we have a closed symmetric monoidal category, we may also consider the \emph{internal derived Hom spaces} by taking the \(i\)-th right derived functors \[\underline{\mathsf{RHom}^i}(M,-) \colon \mathsf{Cond}_\N(\mathsf{Ab}) \to \mathsf{Cond}_\N(\mathsf{Ab}), \quad N \mapsto \underline{\mathsf{Ext}^i}(M,N) := \underline{\mathsf{RHom}^i}(M,N)\] of the internal Hom functor. 

\begin{definition}\label{def:internally-projective}
A condensed abelian group \(M\) is said to be \emph{internally projective} if \(\underline{\mathsf{Ext}^i}(M,N) = 0\) for all \(i >0\) and \(N \in \mathsf{Cond}_{\N}(\mathsf{Ab})\).
\end{definition}

\begin{proposition}
The object \(\Z[\N \cup \{\infty\}]\) is internally projective. 
\end{proposition}

\begin{proof}
It suffices to prove that the quotient \(P = \Z[\N \cup \{\infty\}]/\Z[\{\infty\}]\) is internally projective. To see this, consider an epimorphism of condensed abelian groups \(M \to N\) and an arbitrary map \(f \colon \Z[S] \otimes P \to N\), where \(S\) is a light profinite set. Note that the latter is an arbitrary map in the internal Hom. Such a map is equivalent to a map \(S \sqcup \N \sqcup \{\infty\} \to N\) mapping \(S \sqcup \infty \mapsto 0\). By definition of an epimorphism in a sheaf topos, there is a surjection \(S' \to S\) of light profinite sets and a map \(g \colon S' \to M\) making the obvious diagram commute. Consider the pullback \(S''\) of the diagram \(S' \to \N \cup \{\infty\} \leftarrow \{\infty\}\), which is closed in \(S'\). This is a light profinite set, and using the fact that any light profinite set is an injective object in profinite sets, there is a retraction \(r \colon S' \to S''\).  Then the composition \(g - g \circ r \colon S' \to M\) then induces the required lift \( \N \cup \{\infty\} \to M\), mapping \(\infty \mapsto 0\),  which extends to \(P \to M\) lifting the original map \(f\).  
\end{proof}

Consider the shift function \(S \colon \N  \to \N \), \(n \mapsto n + 1\). This map is clearly proper, so it induces a continuous map \( S \colon \N \cup \{\infty\} \to \N \cup \{\infty\}\), and therefore a map \[S \colon P \to P\] of condensed abelian groups, with \(P = \Z[\N \cup \{\infty\}]/\Z[\{\infty\}]\). Since \(P\) is internally projective, for any condensed abelian group \(M\), the map \(\sigma = 1 - S \colon P \to P\) induces a map \[\sigma^* \colon \underline{\mathsf{Hom}}(P, M) \to \underline{\mathsf{Hom}}(P, M)\] of condensed abelian groups. 

\begin{definition}\label{def:solid}
A light condensed abelian group \(M\) is called \emph{solid} if the map \(\sigma^*\) above is an isomorphism. 
\end{definition}

To see why this has anything to do with a completion, consider a complete metric space with an abelian group structure \(\underline{A}\). Then \(\mathsf{Hom}(P, \underline{A}) \cong C_0(\N, A)\) as we have just seen. But then for the map \[\sigma^* \colon C_0(\N,A) \to C_0(\N, A), \quad \sum_{n \in \N} a_n \delta_n \mapsto \sum_{n \in \N} a_n (\delta_{n} - \delta_{n+1})\] to be an isomorphism, we require that any null sequence in \(A\) comes from a Cauchy and hence convergent sequence. This is the prototypical situation in \(p\)-adic analysis, and as one should expect, no real topological vector space ever satisfies this condition.

Denote by \(\mathsf{Solid} \subset \mathsf{Cond}(\mathsf{Ab})\) the full subcategory of solid abelian groups. Note that at this point we have no reason to believe this category has any nontrivial object. The following result takes care of this:

\begin{proposition}\label{prop:Z-solid}
The group of integers \(\Z\) with the discrete topology is a solid abelian group. 
\end{proposition}

\begin{proof}
We have \(
\underline{\mathsf{Hom}}(P, \underline{\Z})(S) = \mathsf{Hom}(P \otimes \Z[S], \underline{\Z})
\cong \mathsf{Hom}(P, \underline{\mathsf{Hom}}(\Z[S], \underline{\Z})).\) The relations \begin{multline*}
\underline{\mathsf{Hom}}(\Z[S],\underline{\Z})(T) \cong \mathsf{Hom}(\Z[S] \otimes \Z[T], \underline{Z}) \cong \mathsf{Hom}(\Z[S \times T], \underline{\Z}) \\
 \cong \underline{C(S \times T, \Z)} \cong \underline{C(T, \mathsf{C}(S,\Z)},
 \end{multline*} 
imply that \(\underline{\mathsf{Hom}}(\Z[S], \underline{Z}) \cong \underline{C(S,\Z)}\). This shows that \(\underline{\mathsf{Hom}}(P, \underline{\Z}) \cong \mathsf{Hom}(P, \underline{\Z}) \cong C_0(\N, \Z)\). Furthermore, for the induced map \[\sigma^* \colon C_0(\N, \Z) \to C_0(\N, \Z), (x_n) \mapsto (x_0 - x_1, x_1 - x_2, \cdots)\] to be an isomorphism, it suffices that any null sequence in \(\Z\) in the usual absolute value be summable, which is obvious. The inverse is then given by \((x_n) \mapsto (\sum_{i \geq n} x_i)_n\).  
\end{proof}

With at least one example of a solid abelian group, the following result shows how to build more. 

\begin{proposition}\label{prop:solid-properties}
The full subcategory inclusion \(\mathsf{Solid} \subset \mathsf{Cond}_{\N}(\mathsf{Ab})\) is closed under kernels, cokernels, extensions, limits and colimits, and (derived) internal \(\mathsf{Hom}\). Furthermore, the inclusion has a left adjoint functor \[(-)^\blacksquare \colon \mathsf{Cond}_{\N}(\mathsf{Ab}) \to \mathsf{Solid}.\] Finally, there is an induced colimit-preserving symmetric monoidal structure on \(\mathsf{Solid}\), and  the functor \((-)^\blacksquare\) is symmetric monoidal.
\end{proposition}

\begin{proof}
Let \(M \to N\) be a morphism of solid abelian groups. Then the kernel \(K\) is a condensed abelian group. To see that it is solid, we need to prove that the map \(\sigma^* \colon \underline{\mathsf{Hom}}(P, K) \to \underline{\mathsf{Hom}}(P, K)\) is an isomorphism. But since \(P\) is internally projective, we have by left exactness \(\underline{\mathsf{Hom}}(P,K) = \mathrm{ker}(\underline{\mathsf{Hom}}(P,M) \to \underline{\mathsf{Hom}}(P,N))\). Then \(\sigma^* \colon \underline{\mathsf{Hom}}(P,K) \to \underline{\mathsf{Hom}}(P,K)\) is an isomorphism using that \(M\) and \(N\) are solid and the naturality of kernels. One similarly uses the right exactness of \(\underline{\mathsf{Hom}}(P,-)\) to prove that the cokernel of a map of solid abelian groups is solid. 

For internal-Hom, consider a solid abelian group \(M\) and an arbirary condensed abelian group \(N\). Then \(\sigma^* \colon \underline{\mathsf{Hom}}(P, \underline{\mathsf{Hom}}(N,M)) \to \underline{\mathsf{Hom}}(P, \underline{\mathsf{Hom}}(N,M))\) is equivalent to the map \(\underline{\mathsf{Hom}}(P \otimes N, M) \to \underline{\mathsf{Hom}}(P \otimes N, M)\), which in turn is equivalent to \[\underline{\mathsf{Hom}}(N, \underline{\mathsf{Hom}}(P,M)) \to \underline{\mathsf{Hom}}(N, \underline{\mathsf{Hom}}(P,M)),\] which is an isomorphism as \(M\) is solid. The same argument works for the internal derived-Hom. 

The existence of the left adjoint follows from the fact that \(\mathsf{Solid}\) is an accessible Bousfield localisation of \(\mathsf{Cond}(\mathsf{Ab})\) at the set of maps of the form \(\mathrm{id}_{\Z[T]} \otimes \sigma\), for light profinite sets \(T\).  The symmetric monoidal structure is given by applying the solidification functor to the tensor product in condensed abelian groups, that is, \[M \otimes^{\blacksquare} N := (M \otimes N)^\blacksquare,\] which manifestly preserves colimits in both variables as solidification is a left adjoint functor.  Finally, to see that solidification is symmetric monoidal, we want the natural map \[M \otimes^\blacksquare N \to (M^\blacksquare \otimes N^\blacksquare)^\blacksquare\] to be an isomorphism for condensed abelian groups \(M\) and \(N\). And to see this, consider for a solid abelian group \(P\), \begin{multline*}\mathsf{Hom}((M^\blacksquare \otimes N^\blacksquare)^\blacksquare, P) \cong  \mathsf{Hom}(M^\blacksquare \otimes N^\blacksquare, P) \cong \mathsf{Hom}(M^\blacksquare, \underline{\mathsf{Hom}}(N^\blacksquare, P)) \\ 
\mathsf{Hom}(M, \underline{\mathsf{Hom}}(N^\blacksquare, P)) \cong \mathsf{Hom}(N^\blacksquare, \underline{\mathsf{Hom}}(M, P)) \\
 \cong \mathsf{Hom}(N, \underline{\mathsf{Hom}}(M,P)) \cong \mathsf{Hom}(M \otimes N, P) \cong \mathsf{Hom}(M \otimes^\blacksquare N, P)
\end{multline*} as required. 
\end{proof}

The functor \((-)^\blacksquare\) is called the \emph{solidification} functor. It behaves like a completion in the following sense:

\begin{proposition}
A light condensed abelian group \(M\) is solid if and only if for any map of condensed sets \(S \to M\) from a light profinite set, there is a unique extension \(\Z[S]^\blacksquare \to M\). 
\end{proposition}

\begin{proof}
By the adjunction in Proposition \ref{prop:solid-properties} and the free-forgetful adjunction, we have \[\mathsf{Hom}_{\mathsf{Solid}}(\Z[S]^\blacksquare, M) \cong \mathsf{Hom}_{\mathsf{Cond}(\mathsf{Ab})}(\Z[S],M) \cong \mathsf{Hom}_{\mathsf{Cond}(\mathsf{Set})}(S,M),\] as required. 
\end{proof}

Before moving on, we consider a few examples of solidification.

\begin{example}\label{ex:main-solid}
Consider the internally projective object \(P = \Z[\N \cup \{\infty\}]/ \Z[\{\infty\}]\) classfying null sequences. We may identify its solidification with the product \(\prod_{\N} \Z\) taken in condensed abelian groups. To see this, consider the space of bounded sequences \[\prod_{\N}^{\mathrm{bdd}} \Z := \bigcup_{n \in \N} \prod_{\N} (\Z \cap [-n,n]) \subset \prod_{\N} \Z,\] equipped with the subspace topology from the countable product of \(\Z\) with its discrete topology. To clarify what this is as a condensed abelian group, we assign to a light profinite set \(S\), the abelian group \((\prod_{\N}^{\mathrm{bdd}} \Z)(S) := C_b(S, \prod_{\N} \Z)\) of continuous functions with uniformly bounded range.  There is a natural map \[P \to \prod_{\N}^{\mathrm{bdd}} \Z, \quad [n] \mapsto (0,\dotsc,1,0,\dotsc).\] This induces an isomorphism upon taking solidification. Furthermore, we also have an identification \[(\prod_{\N}^{\mathrm{bdd}} \Z)^\blacksquare  \overset{\simeq}\to (\prod_\N \Z)^\blacksquare \cong \prod_{\N} \Z,\] where the last equivalence follows from the fact that \(\prod_{\N} \Z\) is already solid, and the first equivalence follows from the exactness of solidification, applied to the exact sequence \[0 \to \prod_{\N}^{\mathrm{bdd}} \Z \to \prod_{\N} \Z \to \prod_{\N} \Z / \prod_{\N}^{\mathrm{bdd}} \Z \to 0\] of condensed abelian groups, and using also that for any solid \(M\), \[\mathsf{Hom}(\prod_{\N} \Z / \prod_{\N}^{\mathrm{bdd}} \Z, M) = 0.\]  

Combining the two equivalences, we have that \[P^\blacksquare \simeq (\prod_{\N}^{\mathrm{bdd}}\Z)^\blacksquare  \simeq \prod_{\N} \Z.\]
\end{example}

\begin{remark}
In fact, an even stronger statement holds. For any solid abelian group, \[\mathsf{Ext}_{\mathsf{Solid}}^i(\prod_{\N} \Z, M) \overset{\simeq}\to \mathsf{Ext}_{\mathsf{Solid}}^i(P^\blacksquare, M) \overset{\simeq}\to \mathsf{Ext}_{\mathsf{Cond}(Ab)}^i(P, M) = 0\] for all \(i >0\). Consequently, \(\prod_{\N} \Z\) is a projective object. In fact, by the equivalence above, since \(\mathsf{Hom}(P, -)\) preserves colimits, \(\prod_{\N} \Z\) is also compact.   
\end{remark}


\begin{theorem}\label{thm:solid-gen}
The category \(\mathsf{Solid}\) has a single compact projective generator, namely, \(\prod_{\N} \Z\). 
\end{theorem}

\begin{proof}
The idea is that for any infinite light profinite set \(S\), one can define a morphism (using the axiom of choice) \(P \to \Z[S]\), such that \[P^\blacksquare \overset{\simeq}\to \Z[S]^\blacksquare.\] We have already seen that \(P^\blacksquare \cong \prod_{\N} \Z\) is a compact and projective object. Using the identification above, we have for any solid abelian group \(M\), \[\mathsf{Hom}_{\mathsf{Solid}}(\prod_{\N} \Z, M) = \mathsf{Hom}_{\mathsf{Solid}}(\Z[S]^\bullet, M) = \mathsf{Hom}_{\mathsf{Cond}(\mathsf{Ab})}(\Z[S], M) = 0\] if and only if \(M(S) = 0\) for all infinite light profinite sets, if and only if \(M = 0\).  
\end{proof}

Since \(\mathsf{Solid}\) is an extension closed subcategory of \(\mathsf{Cond}_{\N}(\mathsf{Ab})\), it is itself an abelian category. And by Theorem \ref{thm:solid-gen}, it is also a Grothendieck abelian category. In particular, it has a derived \(\infty\)-category that embeds \(\mathbf{D}(\mathsf{Solid}) \subset \mathbf{D}(\mathsf{Cond}_{\N}(\mathsf{Ab})\) fully faithfully inside the derived category of condensed abelian groups. This has a left adjoint \[\mathbf{D}(\mathsf{Cond}_{\N}(\mathsf{Ab})) \to \mathbf{D}(\mathsf{Solid}),\] called the \emph{derived solidification functor}.

Consider a ring object \(A \in \mathsf{Alg}(\mathsf{Cond}_{\N}(\mathsf{Ab}))\); we shall call this a \emph{condensed ring}. Suppose the underlying condensed abelian group of \(A\) is solid, we may consider the category \[\mathsf{Mod}_A(\mathsf{Solid}) := \mathsf{Mod}_A( \mathsf{Cond}_\N(\mathsf{Ab})) \times_{\mathsf{Cond}_\N(\mathsf{Ab})} \mathsf{Solid}\] of \emph{solid} \(A\)-modules.

As in Proposition \ref{prop:solid-properties}, we again have the following:

\begin{proposition}\label{prop:relative-solid}
Let \(A\) be a condensed ring whose underlying abelian group is solid. Then the full subcategory inclusion \(\mathsf{Mod}_A(\mathsf{Solid}) \subset \mathsf{Mod}_A(\mathsf{Cond}_\N(\mathsf{Ab}))\) is closed under kernels, cokernels, extensions, limits and colimits, and (derived) internal \(\mathsf{Hom}\). 
\end{proposition}

\begin{example}\label{ex:Zp-Qp}
Consider \(\Z_p\) as a topological \(\Z_p\)-module with its canonical \(p\)-adic topology. As a compact metric space, it is a light profinite set, so that it may be viewed as a condensed ring via the Yoneda embedding. To see that it is solid, by an argument similar to Example \ref{ex:main-solid}, we are reduced to showing that \[\sigma^* \colon C_0(\N, \Z_p) \to C_0(\N, \Z_p), \sum_{n \in \N} a_n \delta_n \mapsto \sum_{n \in \N} a_n(\delta_n - \delta_{n+1})\] is an isomorphism, which follows from the ultrametric property of the \(p\)-adic norm. A similar argument shows that \(\Q_p\) is a solid ring.  
\end{example}

In particular, we get categories of ``complete" topological \(\Z_p\)-modules and \(\Q_p\)-vector spaces working in \[\mathsf{Mod}_{\Z_p}(\mathsf{Solid}) \subset \mathsf{Mod}_{\Z_p}(\mathsf{Cond}(\mathsf{Ab})), \quad \mathsf{Mod}_{\Q_p}(\mathsf{Solid}) \subset \mathsf{Mod}_{\Q_p}(\mathsf{Cond}(\mathsf{Ab})),\] with left adjoint solidification functors playing the role of ``nonarchimedean" completions. These categories contain Banach \(\Z_p\)-modules and Banach \(\Q_p\)-vector spaces. Likewise, we have derived solidification functors that are left adjoint to the inclusions \[\mathbf{D}(\mathsf{Mod}_A(\mathsf{Solid})) \to \mathbf{D}(\mathsf{Mod}_A(\mathsf{Cond}_{\N}(\mathsf{Ab}))).\]

\subsubsection{Gaseous modules}

In this section, we consider archimedean completions. We view \(\R\) with its archimedean topology as a condensed abelian group in the usual way: \(S \mapsto C(S, \R)\).  The reason why archimedean topological vector spaces need to be treated separately is the following:

\begin{proposition}\label{lem:real-solid}\cite[Corollary 6.1]{CS1}
The solidification of \(\R\) is zero. Consequently, the solidification of any complex \(M\in \mathbf{D}(\mathsf{Mod}_\R(\mathsf{Cond}(\mathsf{Ab})))\) trivialises. 
\end{proposition}
\begin{proof}
Consider the null sequence \(x = (1,1/2,1/2,1/4,1/4,1/4,1/4,\dotsc)\) in \(\R\). This induces a null sequence \(x\) in the solidification via the map \[\mathsf{Hom}_{\mathsf{Cond}(\mathsf{Ab})}(P, \R) \to \mathsf{Hom}_{\mathsf{Cond}(\mathsf{Ab})}(P, \R^\blacksquare).\]  Since \(\R^\blacksquare\) is solid, the map \(\sigma^* \colon \mathsf{Hom}_{\mathsf{Cond}(\mathsf{Ab})}(P, \R^\blacksquare) \to \mathsf{Hom}_{\mathsf{Cond}(\mathsf{Ab})}(P, \R^\blacksquare)\) is an isomorphism, so that we get a unique lift \(y \in \mathsf{Hom}_{\mathsf{Cond}(\mathsf{Ab})}(P, \R^\blacksquare)\). Now consider the first coordinate map \(\Z \to P\) induced by  \(\{0\} \to \N \cup \{\infty\}\); postcomposing yields a map \[\mathsf{Hom}_{\mathsf{Cond}(\mathsf{Ab})}(P, \R^\blacksquare) \to \mathsf{Hom}_{\mathsf{Cond}(\mathsf{Ab})}(\Z, \R^\blacksquare) \cong \R^{\blacksquare}.\] Applying this to \(y\) yields an element \(y_0\) satisfying \(y_0 = 1 + y_0\). Since \(\R^{\blacksquare}\) is a ring object in \(\mathsf{Cond}(\mathsf{Ab})\), this implies \(\R^{\blacksquare} = 0\). 
\end{proof}

\begin{remark}
In the proof above, the lift \(y\) is intuitively supposed to be thought of as the sequence \((\sum_{n=0}^\infty x_n, \sum_{n=1}^\infty x_n, \sum_{n=2}^\infty x_n, \cdots)\), with \(x\) as in the proof. The first coordinate projection is then \(y_0 = 1 + 1/2 + 1/2 + 1/4 + 1/4 + 1/4 + 1/4 + \cdots = 1 + y_0\). To translate this intuition and other operations one can perform on sequences as operations on \(P\), we direct the reader to \cite[Remark 6.1.3]{kedlaya}.    
\end{remark}

\begin{definition}\label{def:gaseous}
A condensed \(\R\)-module \(M \in \mathsf{Mod}_{\R}(\mathsf{Cond}(\mathsf{Ab}))\) is called \emph{gaseous} if the canonical map \[\sigma^* \colon \underline{\mathsf{Hom}}(P, M) \overset{(1 - \frac{S}{2})^*}\to \underline{\mathsf{Hom}}(P,M)\] induced by the shift map \(S \colon \N \to \N\), \(n \mapsto n +1\) is an isomorphism. 
\end{definition}

To see what this condition means, consider a topological abelian group \(M\), so that \(\underline{\mathsf{Hom}}(P, M) = C_0(\N, M)\). Then the requirement that \(\sigma^*\) is an isomorphism says that for a null sequence \((a_n) \in C_0(\N, M)\), there is a null sequence \((b_n) \in C_0(\N, M)\) such that \(b_n - \frac{b_{n+1}}{2} = a_n\). By recursion, we have \(b_n = \sum_{j=0}^\infty 2^{-j} a_{n+j}\), which is absolutely convergent, for instance when \(M\) is a Banach \(\R\)-vector space.

Denote by \(\mathsf{Mod}_{\R}(\mathsf{Gas}) \subset \mathsf{Mod}_{\R}(\mathsf{Cond}(\mathsf{Ab}))\) the full subcategory of gaseous \(\R\)-modules. As in the solid theory, we have the following:

\begin{theorem}\label{thm:gaseous-properties} 
We have the following:
\begin{enumerate}
\item The category \(\mathsf{Mod}_{\R}(\mathsf{Gas})\) is abelian, closed under (derived) limits, (derived) colimits,and (derived) internal \(\mathsf{Hom}\);
\item The inclusion \(\mathsf{Mod}_{\R}(\mathsf{Gas}) \subset \mathsf{Mod}_{\R}(\mathsf{Cond}(\mathsf{Ab}))\) has a left adjoint, \(M \mapsto M^{\mathrm{gas}}\). 
\item There is a unique symmetric monoidal structure making \(M \mapsto M^{\mathrm{gas}}\) symmetric monoidal. 
\end{enumerate}
\end{theorem}

\begin{proof}
The claims in (1) follow formally using the internal projectivity of \(P\). (2) follows from the observation that \(\mathsf{Mod}_{\R}(\mathsf{Gas})\) is the accessible Bousfield localisation of \(\mathsf{Mod}_{\R}(\mathsf{Cond}(\mathsf{Ab}))\) at \(\mathrm{id}_{\R[T]} \otimes (1-\frac{S}{2})\). 
\end{proof}

Using the fact that \(\mathsf{Mod}_{\R}(\mathsf{Gas})\) is a Grothendieck abelian category, we obtain a derived category \(\bD(\mathsf{Mod}_{\R}(\mathsf{Gas})) \subset \bD(\mathsf{Mod}_{\R}(\mathsf{Cond}(\mathsf{Ab}))\), which formally admits a left adjoint functor, \[\bD(\mathsf{Mod}_\R(\mathsf{Cond}(\mathsf{Ab}))) \to \bD(\mathsf{Mod}_{\R}(\mathsf{Gas}))\] intepreted as a derived completion functor. 

We end this section with a particularly pleasant feature of the gaseous framework. Recall that in the complex bornological setting, there was a distinction between trace-class and factorisably trace-class maps between Banach spaces, ultimately leading to the more complicated category of infinitely nuclear modules. In the gaseous setting, this distinction disappears:

\begin{theorem}[Clausen-Scholze]
The category \(\mathbf{Nuc}(\bD(\mathsf{Mod}_{\C}(\mathsf{Gas})))\) is rigid. 
\end{theorem}
\begin{proof}
Trace-class maps are characterised by singular values with quasi-exponential decay, which is preserved under taking square roots. Consequently, trace-class maps are factorisably trace-class, so that \(\mathbf{Nuc}(\bD(\mathsf{Mod}_{\C}(\mathsf{Gas})))\) is rigid by part (1) of \ref{prop:abstract-nuc}. 
\end{proof}

\section{Higher categorical perspectives on operator algebras}

\subsection{Bivariant \(K\)-theory}

In this section, we recall the definition of bivariant \(K\)-theory in the classical sense of Kasparov. Let \(A\) be a \(C^*\)-algebra. A \emph{pre-Hilbert} \(A\)-module is a right \(A\)-module \(E\) with a sesquilinear form \(\langle -, - \rangle \colon E \times E \to  A\) such that \[\langle x, \lambda y \rangle = \lambda \langle x,y \rangle, \quad \langle x,  y a \rangle = \langle x, y \rangle a, \quad \langle x, y \rangle^* = \langle y,x \rangle, \quad \langle x, x \rangle \geq 0\] for all \(x\), \(y \in E\), \(\lambda \in \mathbb{C}\) and \(a \in A\). A \emph{pre-Hilbert} \(A\)-module is called a \emph{Hilbert module} if it is complete and separated in the seminorm defined by \(\vert x \vert = \vert \langle x, x \rangle\vert^{1/2}\). 

Let \(E\) and \(F\) be Hilbert modules over a \(C^*\)-algebra \(A\). We call a map \(T \colon E \to F\) between Hilbert \(A\)-modules an \emph{adjointable operator} if there exists \(T^*\) such that \(\langle Tx, y \rangle = \langle x, T^*y \rangle\). When a map is adjointable, it is automatically \(\mathbb{C}\)-linear, \(B\)-linear and norm bounded. Denote the space of all adjointable operators by \(\mathcal{L}(E,F)\). When \(E = F\), we simply denote this space by \(\mathcal{L}(E)\), which is a \(C^*\)-algebra with the operator norm.

A \emph{\(\mathbb{Z}_2\)-graded Hilbert \(A\)-module} is a Hilbert \(A\)-module \(E\) with a decomposition \(E \cong E_0 \oplus E_1\), where \(E_i\) are Hilbert \(A\)-submodules for \(i =0,1\).  An element of \(E\) is said to be \emph{homogeneous of degree \(i\)} if it belongs to \(E_i\). An operator \(T \in \mathcal{L}(E)\) is \emph{homogeneous of degree \(j\)} if for any homogeneous element \(x \in E_i\), \(T(x)\) is homogeneous of degree \(i + j\), for \(i,j \in 0,1\).

Now consider two separable \(C^*\)-algebras \(A\) and \(B\). A \(C^*\)-correspondence \((E, \pi)\) is a \(\Z_2\)-graded Hilbert \(B\)-module with a \(*\)-homomorphism \(\pi \colon A \to \mathcal{L}(E)\), such that for all \(a \in A\), \(\pi(a)\) is a homogeneous operator of degree \(0\). Consider the triple \((E, \pi, F)\), where \((E, \pi)\) is a \(C^*\)-correspondence, and \(F \in \mathcal{L}(E)\) is a homogeneous operator of degree \(1\) satisfying \[\pi(a)(F^2 - 1) = 0, \quad [\pi(a), F] = 0, \quad \pi(a)(F - F^*) = 0\] for all \(a \in A\). The collection of all such triples is denoted by \(E_0[A,B]\). We then define \(KK_0^{\mathrm{top}}(A,B)\) to be the homotopy classes of elements in \(E_0[A,B]\), where a homotopy is an element in \(E_0(A, B \hat{\otimes}_{C^*} C[0,1])\). We call this the \(0\)-th bivariant \(K\)-theory group, with group structure defined by taking direct sums. 

To define \(KK^{\mathrm{top}}_1\), we first define \(E_1(A,B)\) as the collection of all triples \((E,\pi,F)\), where \(E\) is a trivially graded Hilbert \(A\)-\(B\)-module and \(F \in \mathcal{L}(E)\) satisfies \[[\pi(a), F] \in \mathcal{K}(E), \quad \pi(a)(F^2 - 1) \in \mathcal{K}(E)\] for all \(a \in A\). Then \(KK_1^{\mathrm{top}}(A,B)\) is again the group of homotopy classes in \(E_1(A,B)\). As a consequence of Bott periodicity, we have \[KK_n^{\mathrm{top}}(A, B) \cong KK_{n+2}^{\mathrm{top}}(A,B)\] for all \(n \in \Z\), so that \(KK_0^{\mathrm{top}}\) and \(KK_1^{\mathrm{top}}\) suffice to define \(\Z_2\)-graded bivariant \(K\)-theory groups.

The definition of \(KK\)-theory just provided is very concrete and useful in many applications. From a homotopy-theoretic point of view, we may ask, however, whether there is an underlying additive category whose morphisms are the graded abelian groups above. This is indeed the case as we have a category that we denote by \(KK^{\mathrm{top}}\) whose objects are \(C^*\)-algebras and whose morphisms are given by \[KK^{\mathrm{top}}(A,B) = \bigoplus_{n \in \N} KK_0^{\mathrm{top}}(A,C_0(\R^n,B)).\]   

\noindent Then with the suspension \(A[1] := C_0(\R, A)\) and exact triangles given by \[A[1] \to B[1] \to \mathsf{cone}(f) \to A,\] where \(f \colon A \to B\) is a \(*\)-homomorphism, and \[\mathsf{cone}(f) = \{(a,b) \in A \oplus C_0((0,1], B) \vert b(1) = f(a),\] the category \(KK^{\mathrm{top}}\) is triangulated (\cite{meyer2007homological}). In fact, \(KK^{\mathrm{top}}\) refines to a stable \(\infty\)-category that we denote by \(\mathbf{KK}^{\mathrm{top}}\)\footnote{This notation is somehwat nonstandard; the reason for it is (1) that the notation \(KK\) appears later to mean the mapping spectrum in noncommutative motives, and (2) the when \(A = \C\), we get complex topological \(K\)-theory, while \(KK\) corepresents algebraic \(K\)-theory}.

\begin{theorem}\label{thm:bivariant-K-properties}\cite{bunke2021stable}
Bivariant \(K\)-theory is a presentable, stable \(\infty\)-category. It has the following universal property: there is a natural functor \(C_s^*\mathsf{Alg} \to \mathbf{KK}^{\mathrm{top}}\) on the category of separable \(C^*\)-algebras, such that if we have any functor \(C^*\mathsf{Alg} \to \bD\) into a stable \(\infty\)-category that is homotopy invariant, \(\mathcal{K}\)-stable and split exact, then we get a unique left exact functor \(\mathbf{KK}^{\mathrm{top}} \to \bD\). 
\end{theorem}
\begin{proof}
The proof proceeds by successively forcing the properties for which \(\mathbf{KK}^{\mathrm{top}}\) is universal (see \cite{bunkehtpy}). As we discuss this in detail in the following section on \(E\)-theory, we do not include a proof here. 
\end{proof}

\subsubsection{Bootstrap class}

Recall that for \(A = \C\), by Bott periodicity, we get a natural equivalence \(\mathbf{KK}^{\mathrm{top}}(\C, A) \simeq K^{\mathrm{top}}(A) \in \mathbf{Sp}\). Using the composition and symmetric monoidal structure of \(\mathbf{KK}\), we then see that the that the stable \(\infty\)-category \(\mathbf{KK}\) is enriched in \(\mathbf{KK}^{\mathrm{top}}(\C, \C) \simeq \mathbf{KU}\)-modules, where \(\mathbf{KU}\) is the ring spectrum representing topological \(K\)-theory. In other words, we get a functor 
\[
K^{\mathrm{top}} \colon C_s^*\mathsf{Alg} \to \mathbf{Mod}_{\mathbf{KU}}
(\mathbf{Sp}),
\] which is symmetric monoidal for the minimal and maximal tensor product of \(C^*\)-algebras. This functor descends to a symmetric monoidal, limit and colimit preserving functor 
\[
K^{\mathrm{top}} \colon \mathbf{KK}^{\mathrm{top}} \to \mathbf{Mod}_{\mathbf{KU}}
(\mathbf{Sp}).
\]

\begin{theorem}\cite{bunkehtpy}
The functor above is the right adjoint of a symmetric monoidal right Bousfield localisation. The essential image of the left adjoint \[\mathbf{Mod}_{\mathbf{KU}}(
\mathbf{Sp}) \to \mathbf{KK}^{\mathrm{top}}\] is the UCT class. 
\end{theorem}

\subsection{\(E\)-theory}

In this section, we recall the construction of \(E\)-theory from a homotopy-theoretic point of view, appearing in \cite{bunke2024theory}.

Let \(C_s^*\mathsf{Alg}\) denote the category of separable \(C^*\)-algebras. We list a collection of desirable properties for functors out of \(C^*\)-algebras: let \(F \colon C_s^*\mathsf{Alg} \to \bD\), where \(\bD\) is a countably cocomplete, stable \(\infty\)-category.  

\begin{enumerate}
\item \emph{Homotopy invariance:} the canonical map \(A \to C([0,1], A)\) induces an equivalence \(F(A) \simeq F(A \otimes C[0,1])\) for all \(A \in C_s^*\mathsf{Alg}\);
\item \emph{Stability:} the corner inclusion \(A \to \mathcal{K}(A)\) induces an equivalence \(F(A) \simeq F(A \otimes \mathcal{K})\), where \(\mathcal{K}\) denotes the \(C^*\)-algebra of compact operators on a separable Hilbert space;
\item \(F\) preserves countable filtered colimits;
\item \(F\) preserves direct sums in the sense that for a countable family \((A_i)_{i \in I}\) of \(C^*\)-algebras, the canonical map \[\coprod_{i \in I} F(A_i) \to F(\bigoplus_{i \in I} F(A_i)\] is an equivalence; 
\item \(F\) is exact in the sense that for an exact sequence of \(C^*\)-algebras \[0 \to A \to B \to C \to 0,\] we have \(F(0) \simeq 0\) and the induced diagram 
\[
\begin{tikzcd}
F(A) \arrow{r}{} \arrow{d}{} & F(B) \arrow{d}{} \\
0 \arrow{r}{} & F(C) 
\end{tikzcd}
\] is a pullback diagram in \(\bD\). 
\end{enumerate}

It turns out that these properties may be successively forced in a controlled manner. We start with homotopy invariance. Recall that the mapping set between two \(C^*\)-algebras \(A\) and \(B\) may be topologically enriched using the compact open topology. Denote the resulting compactly generated topological space by \(\underline{\mathsf{Hom}}(A,B)\). Let \(C_s^*\mathbf{Alg}_h\) denote the Dwyer-Kan localisation of the category \(C_s^*\mathsf{Alg}\) at the homotopy equivalences. Concretely, this is given by the homotopy coherent nerve of the associated simplicial category, whose morphism space is given by \(\mathsf{Sing}(\underline{\mathsf{Hom}}(A,B))\), which is a Kan complex, so that the resulting simplicially enriched category of \(C^*\)-algebras is locally fibrant; this is needed to show that the homotopy coherent nerve is indeed a quasi-category. By definition and the universal property of Dwyer-Kan localisations, the canonical functor \(L_h \colon C_s^*\mathbf{Alg} \to C_s^*\mathbf{Alg}_h\) restricts to an equivalence 
\[\mathbf{Fun}(C_s^*\mathbf{Alg}_h, \bC) \to \mathbf{Fun}^h(C_s^*\mathbf{Alg}, \bC)\] between functor \(\infty\)-categories, where the right hand side denotes homotopy invariant functors and \(\bC\) is an arbitrary \(\infty\)-category. 

Next, we consider the Yoneda embedding \(y \colon C_s^*\mathbf{Alg}_h \to \mathbf{Ind}_{\omega_0}^{\omega_1}(C_s^*\mathbf{Alg}_h)\). For a countable filtered system of \(C^*\)-algebras \(A = (A_n)_{n \in \N}\), we have a canonical map \[i_A \colon \colim_n (y \circ L_h)(A_n) \to (y \circ L_h)(\colim_n A_n).\] Let \(W_R = \{i_A: A \colon \N \to C_s^*\mathbf{Alg}\}\) and \(\mathbf{A}_s\) the Dwyer-Kan localisation of \(\mathbf{Ind}_{\omega_0}^{\omega_1}(C_s^*\mathbf{Alg}_h)\) at \(W_R\). This comes with a natural functor \[R \colon \mathbf{Ind}_{\omega_0}^{\omega_1}(C_s^*\mathbf{Alg}_h) \to \mathbf{A}_s.\]

Arguably the most striking ingredient for the dualisability of \(E\)-theory is the following: 

\begin{proposition}\label{prop:Bousfield-loc}\cite[Theorem 3.41]{bunke2024theory}
\begin{enumerate}
\item The functor \(R\) is the right adjoint of a right Bousfield localisation \[L \colon \mathbf{A}_s \to \mathbf{Ind}_{\omega_0}^{\omega_1}(C_s^*\mathbf{Alg}_h).\]
\item The category \(\mathbf{A}_s\) is pointed, left exact, and has countable filtered colimits that are preserved by \(R\).
\end{enumerate}
\end{proposition}

\begin{proof}
We outline the proof of the first statement as this crucially uses Blackadar's shape theory for separable \(C^*\)-algebras. To show that the Dwyer-Kan localisation \(R\) is actually a right Bousfield localisation, we use the general criterion (\cite[Remark 3.42]{bunke2024theory}) that if a Dwyer-Kan localisation \(R \colon \bC \to \bD\) is such that for every object \(x \in \bC\), there is a map \(y \to x\) that gets mapped to an equivalence in \(\bD\) and \(y\) is \(W\)-local, then \(R\) is a right Bousfield localisation. For a separable \(C^*\)-algebra \(A\), shape theory \cite[Theorem 4.3]{blackadar1985shape} provides a countable inductive system \((A_n)_{n \in \N}\) with (explicit) semi-projective structure maps \(A_n \to A_{n+1}\) and colimit \(A\). The condition of being semi-projective says that every \(C^*\)-algebra is \(S\)-exhaustible, where \(S\) is the precompact ideal of maps with the shifted lifting property (and are therefore weakly compact). Using this, one shows that \(colim_n y \circ L_h(A_n)\) is \(W_R\)-local and the map \(colim_n y \circ L_h(A_n) \to y \circ L_h(A)\) is manifestly in \(W_R\). The second statement is formal. 
\end{proof}

Denote the composition of \(R\), \(y\) and \(L_h\)  by \(a \colon C_s^*(\mathsf{Alg}) \to \mathbf{A}_s\). The following follows essentially by the construction of \(a\):

\begin{lemma}\label{lem:universal-filtered-E}
The functor \(a \colon C_s^*(\mathsf{Alg}) \to \mathbf{A}_s\) is initial among functors \(F \colon  C_s^*(\mathsf{Alg}) \to \bC\) into a countably cocomplete categories that satisfy homotopy invariance and preserve countable filtered colimits. 
\end{lemma}

It now remains to force exactness and \(\mathcal{K}\)-stability. To force \(\mathcal{K}\)-stability, we consider the functor \[L_K \colon \mathbf{A}_s \to \mathbf{A}_s, \quad A \mapsto A \otimes \mathcal{K}.\] Since \(\C \to \mathcal{K}\) is an idempotent object (already) in \(C_s^*(\mathbf{Alg})_h\), then \(L_\mathcal{K}\) is a Dwyer-Kan localisation of \(\mathbf{A}_s\) at the set \(\{A \to A \otimes \mathcal{K}:A \in C_s^*(\mathsf{Alg})\}\), which we denote by \(L_{\mathcal{K}}\mathbf{A}_s\). In fact, \(L_\mathcal{K}\) is the left adjoint of a left Bousfield localisation, where the right adjoint preserves countable filtered colimits (using that the tensor product commutes with countable filtered colimits). The resulting composition \[L_{h, \mathcal{K},e} \colon C_s^*(\mathsf{Alg}) \to L_{\mathcal{K}}\mathbf{A}_s\] is homotopy invariant, \(\mathcal{K}\)-stable. The functor \(L_h\) takes a Schochet exact sequence to a pullback square in the homotopy category, and the other localisations preserve finite limits, so that the composition is Schochet exact. This implies exactness by \cite[Proposition 3.58]{bunke2024theory}.

It therefore only remains to force stability. Consider the Bott map as in \cite{GHT}: \[\beta \colon L_{h, \mathcal{K},e}(S^2(\C)) \to L_{h, \mathcal{K},e}(\C)\] obtained from applying the functor \(L_{h, \mathcal{K},e}\) to the reduced Toeplitz extension \[\mathcal{K} \to \mathcal{T}_0 \to S(\C).\] We have natural equivalences \begin{multline*}
\Omega^2(-) \simeq L_{h, \mathcal{K},e}(S^2(\C)) \otimes - \colon L_{\mathcal{K}}(\mathbf{A}_s) \to L_{\mathcal{K}}(\mathbf{A}_s) \to L_{\mathcal{K}}(\mathbf{A}_s) \to L_{\mathcal{K}}(\mathbf{A}_s) \\
 \mathrm{id} \simeq L_{h, \mathcal{K},e}(S(\C)) \otimes - \colon L_{\mathcal{K}}(\mathbf{A}_s) \to L_{\mathcal{K}}(\mathbf{A}_s),
 \end{multline*} inducing a natural transformation \(\beta' = \beta \otimes - \colon \Omega^2(-) \to \mathrm{id} \colon L_{\mathcal{K}}(\mathbf{A}_s) \to L_{\mathcal{K}}(\mathbf{A}_s)\). Taking the Dwyer-Kan localisation at the maps \(\beta'(A) = \Omega^2(A) \to A\), we arrive at a stable \(\infty\)-category and a functor 
 
 \[
e \colon C_s^*\mathsf{Alg} \to L_{\mathcal{K}}\bA_s \to  \bE
\] representing \(E\)-theory. By construction, we have the following:

\begin{theorem}\cite[Proposition 3.55]{bunke2024theory}
The functor \(e \colon C_s^*\mathsf{Alg} \to \bE\) is the initial homotopy invariant, \(\mathcal{K}\)-stable, Schochet exact, countable filtered colimit-preserving functor into a countable cocomplete stable \(\infty\)-category. 
\end{theorem}

We now turn to the main theorem of \cite{bunke2024theory}:


\begin{theorem}\label{thm:E-compactly-ass}(Bunke-D\"unzinger)
\(E\)-theory is dualisable.
\end{theorem}

\begin{proof}
The category \(\mathbf{Ind}_{\omega_0}(C_s^*\mathsf{Alg}_h)\) is compactly generated, so in particular, compactly assembled. The proof then relies on the subsequent localisations being left or right Bousfield localisations. These need not preserve compact assembly, but when the right adjoint preserves countable filtered colimits, the Bousfield localisation is again compactly assembled by \cite[Proposition 3.65, 3.66]{bunke2024theory}.
\end{proof}

\section{Definition of Algebraic K-theory}

The goal of this section is to give an introduction to algebraic K-theory. Classically, K-theory is an invariant of rings taking the following format: to any (associative) ring $R$ we assign the a sequence of abelian groups
\[
K_n(R) \qquad n \in \mathbb{Z}
\]
that is covariantly functorial in $R$. We have explicit descriptions for the low dimensional $K$-groups, namely $K_0$ is the group completion of the monoid of isomorphisms classes of finitely generated projective $R$-modules, or in a generator and relations description:
\[
K_0(R) = \frac{\langle [P] \mid P \text{ finitely generated projective $R$-module} \rangle}{[P \oplus Q] = [P] + [Q]} .
\]
Similarly we have a description for $K_1$ in terms of automorphisms of finitely generated projective $R$-modules:
\[
K_1(R) = \frac{\langle [f] \mid f: P \xrightarrow{\simeq} P \rangle}{[f \oplus g] = [f] + [g], [fg] = [f] + [g] } \ .
\]
\begin{exercise}
Show that in the above description of $K_1$ it is irrelevant if we ask $P$ to be f.g. projective or f.g. free (i.e. $P = R^n$). 
\end{exercise}
\begin{exercise}
Compute the $K$-groups $K_0$ and $K_1$ for $R$ a field and for $R = \mathbb{Z}$.
\end{exercise}

It was an insight of Quillen, that higher $K$-groups $K_n(R)$ for $n \geq 2$ can be defined as the homotopy groups of a $K$-theory space $K(R)$, i.e.
\begin{equation}\label{K-groups}
K_n(R) = \pi_n(K(R)) \ .
\end{equation}
In fact, the space $K^{\Omega^\infty}(R)$ will indeed be the space underlying a spectrum, which we denote by $K(R)$ and the formula \eqref{K-groups} will then hold for all $n \in \Z$. Before we explain how to define the $K$-theory space/spectrum, let us just mention that algebraic $K$-groups are a rather subtle subject. They show up in many areas of mathematics and are home to some very important invariants, e.g. Whitehead torsion. But to compute them is actually a very complicated undertaking. For example for the initial ring $R = \Z$ the following is known: 

\begin{center}
\begin{tabular}{rl@{\hskip 2cm}rl}
$K_{-n}(\mathbb{Z})$ &= $0$ for $n > 0$ &
$K_0(\mathbb{Z})$ &= $\mathbb{Z}$ \\
$K_1(\mathbb{Z})$ &= $\mathbb{Z}/2$ &
$K_2(\mathbb{Z})$ &= $\mathbb{Z}/2$ \\
$K_3(\mathbb{Z})$ &= $\mathbb{Z}/48$ &
$K_4(\mathbb{Z})$ &= $0$ \\
$K_5(\mathbb{Z})$ &= $\mathbb{Z}$ &
$K_6(\mathbb{Z})$ &= $0$ \\
$K_7(\mathbb{Z})$ &= $\mathbb{Z}/240$ &
$K_8(\mathbb{Z})$ &= $0$ \\
$K_9(\mathbb{Z})$ &= $\mathbb{Z} \oplus \Z/2$ &
$K_{10}(\mathbb{Z})$ &= $\Z/2$ \\
$K_{11}(\mathbb{Z})$ &= $\mathbb{Z}/1008$ &
$K_{12}(\mathbb{Z})$ &= $???$ (= 0) \\
$K_{13}(\mathbb{Z})$ &= $\mathbb{Z}$ &
& ...
\end{tabular}
\end{center}
Here note that the group $K_{12}(\Z)$ is not known, but conjectured to be zero. In fact, it is known that it is a finite torsion group and the torsion has to be $n$-torsion for $n \geq 20.000$. So for most practical purposes this suffices. In fact, all $K$-groups are known, subject to the Kummer-Vandiver conjecture - an infamous conjecture in number theory, which ends up being equivalent to the assertion that $K_{4n}(\Z) = 0$ for all $n$. 

Other calculations of $K$-groups are known, e.g. for finite fields by Quillen or in recent work by the first named author with Antieau and Krause for rings $\Z/n$. But generally the computation of $K$-groups is generally hard (and not the subject of this lecture). 

\subsection{Definition of positive K-groups}

In order to give the definition of $K$-theory, we would like to note that $K_0$ and $K_1$ clearly don't really depend on the ring $R$, but only on its category of modules. In particular Morita equivalent rings have isomorphic $K$-groups. This is crucial for the Definition that we are going to take. In fact, we will take an even stronger step and immediately pass to the derived $\infty$-category $\mathcal{D}(R)$ of the ring, in which the ordinary category of $R$-modules embedds fully faithfully. We will define the $K$-theory in such a way that it only depends on  $\mathcal{D}(R)$. 

\begin{exercise}
Recall the definition of compact objects and that the compact objects in $\mathcal{D}(R)$ are precisely those complexes quasi-isomorphic to a complex of finite length, where each term is f.g. projective.
\end{exercise}

In order to understand the step from finitely generated projective modules to perfect chain complees let us first give the definiton of $K_0$ in this generality. 

\begin{definition}\label{def}
Let $\mathbf{C}$ be a compactly generated stable $\infty$-category (e.g. $\mathcal{D}(R)$) with compact objects $\mathbf{C}^\omega \subset \mathbf{C}$. 
\[
K_0(\mathbf{C}) = \frac{\langle [c] \mid c \in \mathbf{C}^\omega \rangle}{[b] = [a] + [c] \text{ for } a \to b \to c \text{ fibre sequence}}
\]
\end{definition}

\begin{exercise}
Show that for a ring $R$ the old and new Definiton of $K_0$ agree, by exhibiting an explicit isomorphisms (using the Euler characteristic).

Also show that one can simply take the monoid of isomorphism classes of objects $c \in \mathbf{C}^\omega$ and quotient by the relation and this already gives the correct group. The inverse of $[c]$ is given by $[\Sigma c]$. 
\end{exercise}

Now we come to the $K$-theory space $K(R)$. This will be defined as the geometric realisation of a span category. This definition is basically Quillen's Q-construction cast in the language of $\infty$-categories, as realised by Barwick.

\begin{definition}
For a compactly generated, stable $\infty$-category $\mathbf{C}$
consider the category
$
\mathrm{Span}(\mathbf{C}^\omega)
$ of spans in $\mathbf{C}^\omega$. 
Then we set 
\[
K^{\Omega^\infty}(\mathbf{C}) := \Omega | \mathrm{Span}(\mathbf{C}^\omega) | \ .
\]
\end{definition}

\begin{exercise}
Show that $\pi_0(K^{\Omega^\infty}(\mathbf{C})  = \pi_1(  | \mathrm{Span}(\mathbf{C}^\omega) |)$ is isomorphic to $K_0(\mathbf{C})$ as in Definition \ref{def}. This is a hard exercise, so it is a good start to construct a map 
\[
K_0(\mathbf{C}) \to \pi_0(K^{\Omega^\infty}(\mathbf{C}))
\]
\end{exercise}

\begin{remark}
Let us explain a bit the intuition behind the span category that we have. We think of it as an analogue of an algebraic `cobordism' category (following a philosophy of Steimle, Hebestreit and Raptis). For example, if we have a cobordism of compact manifolds
\[
M \to W \leftarrow N
\]
where $W$ is an $n$-manifold with boundary $M \amalg \overline{N}$ where $M$ and $N$ are $(n-1)$-manifolds. Then we can take singular chains with values in $\Z$
and get a span
\[
C^*(M) \leftarrow C^*(W) \rightarrow C^*(N)
\]
in $\mathcal{D}(Z)^\omega$. Perfectness is a consequence of compactness of the manifolds. This can be made into a functor
\[
\mathrm{Cob}(n) \to \mathrm{Span}(\mathcal{D}(\Z)^\omega) . 
\]
Upon geometric realization this give a map
\[
|\mathrm{Cob}(n)| \to \Sigma K(\Z)
\]
\end{remark}

There are also negative $K$-groups. In fact, there exists a spectrum $K(R)$ with $\Omega^\infty K(R) = K^{\Omega^\infty}(R)$. We will not go into the definition here but note that this is essential forced by the positive ones if one wants the fibration theorem to be true, that we will discuss below (Theorem \ref{fib_thm}).  

\subsection{The Calkin category}

Now we want to generalise $K$-theory to the real of dualisable stable $\infty$-categories. The idea is to find a replacement for $\mathbf{C}^\omega$ which in this case can consist of only the zero object.

Recall that a dualisable stable $\infty$-category is a presentable, stable $\infty$-category
which has enough compact morphisms. Recall that a morphism $f: X \to Y$ is compact if for every map $Y \to Z = \colim_{i \in I} Z_i$ with $i$ filtered we have that the composite $X \to Y \to Z$ factors over a finite stage $Z_i \to Z$. 

\begin{exercise}
Show that if $\mathbf{C}$ is compactly generated, then a morphism $X \to Y$ is compact iff if factors through a compact objects, that is there is a compact object $K \in \mathbf{C}^\omega$ and a factorization $X \to K \to Y$ of $f$. 

In particular for $\mathbf{C} = \mathbf{D}(R)$ contemplate why this is the natural version of `finite rank' operators. For example by showing that a morphism $\bigoplus_{\N} R \to \bigoplus_{\N}  R$ is compact in $\mathbf{D}(R)$ if it factors through a finite sum, i.e. has finite rank. 

Show that  $\mathbf{Map}^{\mathcal{K}}_\mathbf{C}(X,Y)$ is an ideal in  $\mathbf{Map}_\mathbf{C}(X,Y)$, i.e. that the composition of a compact morphism with an arbitrary morphism is again compact.
\end{exercise}

Our first goal now is to define a space of compact morphisms $\mathbf{Map}^{\mathcal{K}}(X,Y)$. One could define this to be simply the subspace of the mapping space $\mathbf{Map}(X,Y)$ on all compact morphisms, but this turns out to be slightly two naive, since the factorization through $Z_i$ is an extra structure that has to be taken into account properly higher categorically. 

To this end recall that one characterization of dualisability for a category $\mathbf{C}$ is that the colimit functor
\[
k: \mathbf{Ind}(\mathbf{C}^{\omega_1}) \to \mathbf{C}
\]
has a left adjoint, which we denote by $\jhat$. The right adjoint, which is the constant ind-object will be denoted $j$. We set for objects $X,Y \in \mathbf{C}$:
\[
\mathbf{Map}^{\mathcal{K}}_\mathbf{C}(X,Y) =  \mathbf{Map}_{\mathbf{Ind}\mathbf{C}^{\omega_1}}(jX, \jhat Y)
\]
to be the spectrum of compact morphisms from $X$ to $Y$. There is a canonical map
\[
\mathbf{Map}^{\mathcal{K}}_\mathbf{C}(X,Y) \to \mathbf{Map}_\mathbf{C}(X,Y) 
\]
induced by the functor $k$ using that $kjX \cong X$ and $k \jhat Y \cong Y$. 

\begin{proposition}
Assume $\mathbf{C}$ is dualisable. Then 
a morphism $f: X \to Y$ in $\mathbf{C}$ is compact precisely if it can be lifted through the map
\[
\mathbf{Map}^{\mathcal{K}}_\mathbf{C}(X,Y) \to \mathbf{Map}_\mathbf{C}(X,Y)  \ .
\]
\end{proposition}

\begin{exercise}
Show that if $X$ or $Y$ are compact, then
\[
\mathbf{Map}^{\mathcal{K}}_\mathbf{C}(X,Y) = \mathbf{Map}_\mathbf{C}(X,Y)  \ .
\]
Show that in the case that $\mathbf{C}$ is compactly generated, the space 
$
\mathbf{Map}^{\mathcal{K}}_\mathbf{C}(X,Y)$
is equivalent to the colimit 
\[
\colim_{i \in I} \mathbf{Map}^{\mathcal{K}}_\mathbf{C}(X,Y_i)
\]
for writing $Y = \colim_{i \in I} Y_i$ as a filtered colimit of compact objects $Y_i$ (which is unique as an ind object) and the map $\mathbf{Map}^{\mathcal{K}}_\mathbf{C}(X,Y) \to \mathbf{Map}_\mathbf{C}(X,Y) $ is the canonical map out of this colimit induced by $Y_i \to Y$.
\end{exercise}

We note that all categories in question are actually stable $\infty$-categories such that in particular the spaces $\mathbf{Map}_\mathbf{C}(X,Y)$ and $\mathbf{Map}^{\mathcal{K}}_\mathbf{C}(X,Y)$ canonically refine to spectra which we denote as follows:
\[
\mathbf{Map}^{\mathcal{K}}_\mathbf{C}(X,Y) \to \mathbf{Map}_\mathbf{C}(X,Y) \ .
\]

Now the idea is to consider the cofibre of this map. In analogy with operator algebras, we think of this cofibre as the Calkin algebra.

\begin{defprop}
For every dualisable, stable $\infty$-category $\mathbf{C}$ 
there is a stable $\infty$-category $\mathbf{Calk}(\mathbf{C})$ with objects given by objects in $\mathbf{C}^{\omega_1}$ and morphisms
\[
\mathbf{Map}_{\mathbf{Calk}(\mathbf{C})}(X,Y) = \mathbf{Map}_\mathbf{C}(X,Y)  / \mathbf{Map}^{\mathcal{K}}_\mathbf{C}(X,Y) \ .
\]
\end{defprop}

Morally, this means that we can define the composition in this category. This informally follows from the fact that $\mathbf{Map}^{\mathcal{K}}_\mathbf{C}(X,Y)$ is an ideal in  $\mathbf{Map}_\mathbf{C}(X,Y)$, i.e. that the composition of a compact morphism with an arbitrary morphism is again compact.
\begin{proof}
Define
$\mathbf{Calk}(\mathbf{C})$ as the full subcategory of 
$\mathbf{Ind}(\mathbf{C}^{\omega_1})$ spanned by the objects of the form $jX / \jhat X$ for $X \in \mathbf{C}^{\omega_1}$. Then using the cofibre sequence $\jhat X \to jX \to j X / \jhat X$ we get a fibre sequence
\[
\mathbf{Map}_{\mathbf{Ind}(\mathbf{C}^{\omega_1})} (j X / \jhat X, jY / \jhat Y) 
\to
\mathbf{Map}_{\mathbf{Ind}(\mathbf{C}^{\omega_1})} (j X, jY / \jhat Y) 
\to 
\mathrm{Map}_{\mathbf{Ind}(\mathbf{C}^{\omega_1})} (\jhat X, jY / \jhat Y) 
\] 
We claim that the last term is zero. Indeed, since $\jhat$ is left adjoint to $k$ this follows since 
$k (jY / \jhat Y) = 0$. Thus the mapping space in question is equivalent to 
\[
\mathbf{Map}_{\mathbf{Ind}(\mathbf{C}^{\omega_1})} (j X, jY / \jhat Y) 
\]
which sits in the relevant fibre sequence using that $\mathbf{Map}_{\mathbf{Ind}(\mathbf{C}^{\omega_1})} (j X, jY) = \mathbf{Map}_{\mathbf{C}}(X,Y)$.  
\end{proof}
Clearly the category $\mathbf{Calk}(\mathbf{C})$ comes with a functor $p: \mathbf{C}^{\omega_1} \to \mathbf{Calk}(\mathbf{C})$. 
By analogy with operator theory, we would call an equivalence in $\mathbf{Calk}(\mathbf{C})$ a Fredholm operator (really the analogous thing would be the image of a Fredholm operator in the Calkin algebra). 

\begin{warning}
We warn the reader that not every Fredholm operator is induced by a morphism in $\mathbf{C}^{\omega_1}$ since the cofibre defining the mapping spectrum in 
 $\mathbf{Calk}(\mathbf{C})$ can pick up terms from $\pi_{-1}\mathbf{Map}^{\mathcal{K}}_\mathbf{C}(X,Y)$. In fact one can show that there are generally very few Fredholm operators that lift, if $\mathbf{C}$ is not compactly generated, as we will do in the next exercise.
\end{warning}

\begin{exercise}\label{ex_fred}
Show that for a morphisms $f: x \to y$ in a compactly generated stable $\infty$-category $\mathbf{C}$ the induced morphism in in $\mathbf{Calk}(\mathbf{C})$ becomes an equivalence, precisely if the fibre of $f$ is compact. In particular, if $\mathbf{C}$ has no compact objects, then this forces $f$ to be an equivalence. 
\end{exercise}

We can even define a category of Fredholm operators, whose objects are given by triples $X, Y$ and $f: X \to Y$ a Fredholm operator, i.e. an invertible morphisms in $\mathbf{Calk}(\mathbf{C})$. 

\begin{definition} 
Let $\mathbf{C}$ be a dualisable, stable $\infty$-category. We define the category $\mathbf{Fred}(\mathbf{C})$ as the pullback
\[
\mathbf{Fred}(\mathbf{C}) = \mathbf{C}^{\omega_1} \times_{\mathbf{Calk}(\mathbf{C})} \mathbf{C}^{\omega_1}  \ .
\]
\end{definition}

We note that $\mathbf{Fred}(\mathbf{C})$ is a stable $\infty$-category. Note that a morphism $(X,Y, f)$ to $(X',Y',f')$ in 
$\mathbf{Fred}(\mathbf{C})$ is given by a pair of morphisms $f: X \to X'$ and $g: Y \to Y'$ in $\mathbf{C}^{\omega_1}$ and a commutative square
\[
\begin{tikzcd}
X \arrow{d}\arrow{r}& X' \arrow{d}\\
Y \arrow{r} & Y'
\end{tikzcd}
\] 
in $\mathbf{Calk}(\mathbf{C})$. 
\begin{exercise}\label{ex_fred2}
Show that for a dualisable category $\mathbf{C}$ there is a fully faithful functor
\[
\mathbf{C}^{\omega} \to \mathbf{Fred}(\mathbf{C})
\]
sending 
$X$ to the morphism $X \to 0$. Show that if $\mathbf{C}$ is compactly generated, then every Fredholm operator, i.e. equivalence in $\mathbf{Calk}(\mathbf{C})$ can be represented by a pair of morphisms
\[
X \to Z \leftarrow Y
\] 
in $\mathbf{C}^{\omega_1}$ such that the fibres of $X \to Z$ and $Y \to Z$ are both compact (represented means that under the functor $\mathbf{C}^{\omega_1} \to \mathbf{Calk}(\mathbf{C})$ this gives the desired equivalence. 
\end{exercise}

\subsection{Efimov K-theory}

Let $\mathbf{C}$ be a dualisable, stable $\infty$-category. We define
\[
K(\mathbf{C}) := K(\mathbf{Ind}(\mathbf{Fred}(\mathbf{C})))
\]

Note that we have to form $\mathbf{Ind}$ first simply because the way we have defined $K$-theory so far, it applies to compactly generated categories. But since the first step is to pass to compact objects, we really take $K$-theory of Fredholm operators.

\begin{proposition}\label{prop_last}
\begin{enumerate}
\item
We have that $K_0(\mathcal{C}) - \pi_0 K(\mathcal{C}) $ is given by equivalence classes of Fredholm operators, modulo fibre sequences. More precisely classes $[f]$ for $f: X \to Y$ a Fredholm operator modulo the relation that $[f] = [f'] + [f'']$ for a fibre sequence in $\mathbf{Fred}(\mathbf{C})$. 

\item
If $\mathbf{C}$ is compactly generated, then the map 
\[
\mathbf{C}^{\omega} \to \mathbf{Fred}(\mathbf{C})
\]
of Exercise \ref{ex_fred2}
induces an equivalence in $K$-theory, so that the old and new definitions agree in this case. 
\item
We have that $K(\mathbf{C})$ is also equivalent to $\Omega K(\mathbf{Calk}(\mathbf{C}))$.  
\end{enumerate}
\end{proposition}

The first statement is clear. 
The last one is proven by showing that the defining pullback square for $\mathbf{Fred}(\mathbf{C})$:
\[
\begin{tikzcd}
\mathbf{Fred}(\mathbf{C}) \arrow{r} \arrow{d}& \mathbf{C}^{\omega_1} \arrow{d}\\
\mathbf{C}^{\omega_1} \arrow{r} &  \mathbf{Calk}(\mathbf{C})
\end{tikzcd}
\] 
induces a pullback square in $K$-theory (of the respective $\mathrm{Ind}$-categories). The $K$-theory of  $\mathbf{Ind}(\mathbf{C}^{\omega_1})$ is zero, because of an Eilbenberg-Swindle. The second claim follows from the third. 
\begin{exercise}
Show that the class of the identity operator $f: X \to X$ for any $X \in \mathbf{C}^{\omega_1}$ represents the zero object in $K_0(\mathbf{C})$. (Hint: Consider the infinte direct sum of $\bigoplus f$ and use an Eilenberg-Swindle. )
Deduce that a morphism $f$ in $\mathbf{C}^\omega$ which becomes Fredholm in $\mathbf{Calk}(\mathbf{C})$ the class of $[f]$ is given by the class of $\mathrm{fib}(f)$ (see Exercise \ref{ex_fred}). Moreover use the description of Fredholm operators in the compactly generated case from Exercise \ref{ex_fred2} to show the $\pi_0$ version of Proposition \ref{prop_last}(2). 
\end{exercise}

\begin{exercise}
Show that every class in $K_0$ can be represented by a Fredholm operator with source and target isomorphic. Hint: Use the previous exercise and consider for $f: X \to Y$ the operator $f \oplus \mathrm{id}_Z$ for $Z = \bigoplus_\mathbb{N} (X \oplus Y)$. Use this to prove that $K_0(\mathbf{C})$ is generated by automorphisms in  $\mathbf{Calk}(\mathbf{C})$ (no lift of source and target needed) modulo one relation for short exact sequences and for composition:
\[
[fg] = [f] + [g] \ .
\]
 using only the automorphisms as generators.
\end{exercise}

\begin{example}
Let $X$ be a locally compact Hausdorff space. 
For any dualisable category $\mathbf{C}$ we have the category $\mathbf{Shv}(X;\mathbf{C})$ of sheaves with values in $\mathbf{C}$. This category is dualisable, but rarely compactly generated (only if $\mathbf{C}$ is compactly generated and $X$ totally disconnected, compact). In particular even for sheaves of chain complexes this is not compactly generated.
\end{example}

Thus $K(\mathbf{Shv}(X; \mathbf{C}))$ is well-defined and a genuinely new object. We have the following result by Efimov.

\begin{theorem}[Efimov]\cite{efimov2024k}\label{thm:efimov} We have
\[
K(\mathbf{Shv}(X; \mathbf{C})) \simeq \Gamma_c(X, K(\mathbf{C})),
\]
that is, $K$-theory of $\mathbf{Shv}(X; \mathbf{C})$ is compactly supported sheaf cohomology of $X$ with values in the $K$-theory spectrum of $\mathbf{C}$.
\end{theorem}

\begin{corollary}\label{cor_deloop}
We find that 
\[
K(\Shv(\mathbb{R}^n;\bC)) = \Omega^n K(\bC) \ .
\]
In particular we get that 
\[
K_n(\bC) = K_0(\Shv(\mathbb{R}^n;\bC)) \ .
\]
This give a formula to compute higher $K$-theory for an arbitrary $\infty$-category as $K_0$, for which we had a (somewhat) explicit formula.
\end{corollary}

\begin{example}[\(K\)-theory of equivariant sheaves]

Next we consider an equivariant generalisation of Theorem \ref{thm:efimov} above. Let \(G\) be a finite group, \(\mathrm{Orb}_G\) the category of \(G\)-orbits, and \(E \colon \mathrm{Orb}_G^\op \to \bC\) a presheaf valued in a dualisable \(\infty\)-category \(\bC\). Let \(t \colon \mathrm{Orb}_G \to \mathrm{Top}_G\) be the canonical inclusion of orbits into \(G\)-spaces. Equipping \(\mathrm{Top}_G\) with the open covers topology, restriction along \(t\) is right adjoint to left-Kan extensiona nd sheafification \[t^* \colon \mathbf{PSh}(\mathrm{Orb}_G^\op, \bC) \leftrightarrows \mathbf{Shv}(\mathsf{Top}_G, \bC) \colon t_*.\]

\begin{definition}
Let \(X\) be a \(G\)-space. The \emph{Bredon sheaf cohomology} of \(X\) with coefficients in \(E\) is defined as \(\Gamma^G(X,E) \defeq t^*(E)(X)\).
\end{definition}

Let \(X\) be a locally compact Hausdorff \(G\)-space. We may define a sheaf on the coarse moduli space \(X/G\) by restricting the sheaf \(t^*(E) \in \Shv(\mathsf{Top}_G, \bC)\) to \(G\)-invariant open subsets of \(X\) (or equivalently, open subsets of \(X/G\). Denote the resulting sheaf by \(E_X\). We then define \emph{compactly supported Bredon sheaf cohomology} of \(X\) with coefficients in \(E\) as \[\Gamma_c^G(X,E) \defeq \Gamma_c(X/G, E_X).\]

Now consider the category \[\mathbf{Shv}^G(X;\bC) = \colim_{[n] \in \Delta^\op}^{\mathbf{Cat}^{\mathrm{dual}}} \mathbf{Shv}(G^n \times X, \bC)\] of \(G\)-equivariant sheaves valued in \(\bC\); this is a dualisable \(\infty\)-category. 

\begin{theorem}\cite{bredon}\label{thm:K-equivariant}
Let \(X\) be a locally compact Hausdorff \(G\)-space, and \(K^G \colon \mathrm{Orb}_G^\op \to \mathbf{Sp}\) the presheaf \(G/H \mapsto K(\mathbf{Sp}^{BH})\). We have an equivalence 
\[
K(\mathbf{Shv}^G(X,\mathbf{Sp})) \simeq \Gamma_c^G(X, K^G).
\] 
\end{theorem}

An analogous result holds for topological \(K\)-theory, where \(\mathbf{Shv}^G(X;\bC)\) is replaced with the crossed product \(C^*\)-algebra \(G \ltimes C_0(X)\).  The important insight due to Clausen in both Theorem \ref{thm:efimov} and Theorem \ref{thm:K-equivariant} is the following conservativity result:

\begin{theorem}\cite{bredon}\label{thm:conservativity}
Let \(\bC\) be a dualisable category and \(\mathbf{Fun}^{o,cc}(\mathrm{LCHaus}_G^\op, \bC) \subset \mathbf{Fun}(\mathrm{LCHaus}_G^\op, \bC)\) be the subcategory of functors satisfying open descent and compact cofiltered codescent (that is, functors taking a cofiltered limit of compact Hausdorff \(G\)-spaces to filtered colimits). Then restriction 
\[
t_* \colon \mathbf{Fun}^{o,cc}(\mathrm{LCHaus}_G^\op, \bC) \to \mathbf{Fun}(\mathrm{Orb}_G^\op, \bC)
\] is an equivalence of categories, with inverse defined by Bredon sheaf cohomology \(E \mapsto \Gamma^G(-, E)\). 
\end{theorem}

The result above says that Bredon sheaf cohomology is the unique functor satisfying open descent and compact cofiltered descent. In fact, when one considers the span category \(\mathrm{LCHaus}_G^{pdp}\) of locally compact Hausdorff spaces with morphisms given by spans of equivariant open inclusions and proper maps, we get an equivalence between functors on \(\mathrm{LCHaus}_G^{pdp}\) satisfying open codescent and compact cofiltered descent, and functors on \(G\)-compact Hausdorff spaces \(\mathrm{CHaus}_G\) satisfying closed and compact cofiltered descent. One then uses Theorem \ref{thm:conservativity} to show that compactly supported Bredon sheaf cohomology is the unique functor with these properties. Since the counit of the adjunction \(t^*\dashv t_*\) induces a natural map \[\Gamma^G(-, K^G) \to K(\Shv^G(-, \bC)),\] and the functor \(K(\Shv^G(-, \bC))\) shares these properties, we have the desired equivalence of Theorem \ref{thm:K-equivariant}.
\end{example}

\begin{example}[$K$-theory of $\Z_p$] 
Consider the category $\bC = \bD(\Z)^\wedge_p \subseteq \bD(\Z)$. This is a symmetric monoidal category 
which is compactly generated (by $\mathbb{F}_p$). Each compact object is dualisable, but the tensor unit is not dualisable. 

The $K$-theory of $\bD(\Z)^\wedge_p$ is by Quillen's Devissage equivalent to $K(\mathbb{F}_p)$.

A morphism $f: M \to N$ in $\mathbf{D}(\Z)^\wedge_p$ is trace class, 
precisely if for every $n$ the induce morphism $f/p^n: M \otimes_{\Z} \Z/p^n  \to N \otimes_{\Z} \Z/p^n$ factors trough a compact morphisms in $\mathbf{D}(\Z/p^n)$. 
For example the identity on $\Z_p$ is trace class. The trace class maps in this case are factorisable.

Now consider the full subcategory of 
\[
\mathbf{Nuc}(\Z_p)^{\omega_1} \subseteq \mathbf{Ind}(\mathbf{D}(\Z)^\wedge_p) 
\]
consisting of those ind diagrams 
\[
M_0 \to M_1 \to M_2 \to ...
\]
where each morphism is trace class. The objects in there are called \emph{basic nuclears}. This contains objects like 
\[
\Z_p = \colim( \Z_p \xrightarrow{\mathrm{id}} \Z_p \xrightarrow{\mathrm{id}} ...),
\]
but also things like
\[
\Q_p =  \colim( \Z_p \xrightarrow{\cdot p} \Z_p \xrightarrow{\cdot p} ...) \ .
\]

The colimit closure of $\mathbf{Nuc}(\Z_p)^{\omega_1}$ inside  $\mathbf{Ind}(\mathbf{D}(\Z)^\wedge_p) $
 is $\mathbf{Nuc}(\Z_p)$.
 It also admits other descriptions, for example it is the rigidification of $\mathbf{D}(\Z)^\wedge_p$ or the dualisable inverse limit
 \[
 ...  \to \mathbf{D}(\Z/p^3) \to  \mathbf{D}(\Z/p^2)  \to \mathbf{D}(\Z/p)
 \]
\end{example}

We have that 
\begin{theorem}[Efimov]\cite{efimov2025localizing}
We have that $K(\mathbf{Nuc}(\Z_p)) = \lim K(\Z/p^n)$.
\end{theorem}

\begin{example}\label{almost}
Let $R$ be a ring (or ring spectrm) with an ideal $I$ such that the map $I \otimes_R I \to I$ is an equivalence. 
\footnote{Here the tensor product means the derived tensor product. We generally adopt the convention that all tensor products are derived and that everything in fact also work for ring spectra in place of rings. }
For example we could take $R = \mathbb{Z}[x^{1/p^\infty}]$
and $I = (x^{1/p^\infty})$ for some prime $p$.

Then we consider $\mathbf{D}(R,I) := \mathrm{ker}(\mathbf{D}(R) \to \mathbf{D}(R/I))$. This is called the almost category and agrees with the category of $H$-unital $I$-modules that we will discuss below. It is dualisable. In fact, one can show that all dualisable, stable $\infty$-categories are of this form, for some ring spectrum $R$ with such an ideal $I$. 

The fibre sequence $\mathbf{D}(R,I) \to \mathbf{D}(R) \to \mathbf{D}(R/I)$ is the prototypical example of a Verdier sequence (as will discuss soon) and thus induces a fibre sequence
\[
K(\mathbf{D}(R,I)) \to K(R) \to K(R/I) \ .
\] 
We will see that this fibre is in fact $K(I)$ for the non-unital ring $I$. 
\end{example}

\section{Properties of K-theory}

The goal of this section is the investigate and explain some properties of algebraic K-theory such as Excision and Localization sequences.

But let us first investigate the functoriality properties of Efimov K-theory. To understand this, we say that a functor $f: \mathbf{C} \to \mathbf{D}$ between presentable $\infty$-categories is strongly continuous if it admits a right adjoint (i.e. $f$ is left adjoint) and the right adjoint admits a further right adjoint.

\begin{exercise}
Asume that $\mathbf{C}$ and $\mathbf{D}$ are compactly generated, i.e. $\mathbf{C} = \mathbf{Ind}(\mathbf{C}^\omega)$ and  $\mathbf{D} = \mathbf{Ind}(\mathbf{D}^\omega)$. Then a functor $\mathbf{C} \to \mathbf{D}$ is strongly continuous precisely if it is of the form $\mathbf{Ind}(f')$ for some finite colimit preserving functor 
$f': \mathbf{C}^\omega \to \mathbf{D}^\omega$. Equivalently: if $f$ is left adjoint and sends compact objects to compact objects. 

Deduce that for rings $R,S$ a strongly continous functor
\[
\mathbf{D}(R) \to \mathbf{D}(S)
\]
is given by tensoring with a derived $S-R$-bimodule $M$ which is compact as an $S$-module.
\end{exercise}

\begin{lemma}
A left adjoint functor between dualisable, stable $\infty$-categories is strongly continuous precisely if it sends compact morphisms to compact morphisms. This is the case precisely if it commutes with $\hat j$ (in the obvious sense).
\end{lemma}

As a result we get from a strongly continuous functor
$
f: \mathbf{C} \to \mathbf{D}
$
induced functors 
\[
\mathbf{Calk}(\mathbf{C}) \to \mathbf{Calk}(\mathbf{D}) \qquad \text{and}  \qquad \mathbf{Fred}(\mathbf{C}) \to \mathbf{Fred}(\mathbf{D})
\]
and therefore and induced map
\[
K(\mathbf{C}) \to K(\mathbf{D}) \ .
\]
If we let $\Prld$ be the $\infty$-category of stable, dualisable $\infty$-categories and strongly left adjoint functors, then we get a functor 
\[
K : \Prld \to \mathbf{Sp} \ .
\]

\subsection{Verdier sequences}

Now we come to the exactness properties of $K$-theory. 

\begin{definition}
A sequence $\bC \xto{i} \bD \xto{p} \bE$ in $\Prld$ is a \emph{Verdier sequence} if it is a fibre and cofibre sequence. Explicitly this is equivalent to:
\begin{itemize}
\item
$i$ is fully faithful and the kernel of $p$.
\item
The right adoint $R_p$ of $p$ is fully faithful.
\end{itemize}
\end{definition}

\begin{example}
Let $R$ be a ring and $x \in R$ a central element. Then we have the sequence
\[
\bD(R)^{x-\text{tor}} \to \bD(R) \to \bD(R[x^{-1}]) 
\]
is a Verdier sequence, where $ \bD(R)^{x-\text{tor}} \subseteq \bD(R)$ is the full subcategory on the $x$-torsion modules. Here an object $M \in \bD(R)$ is called $x$-torsion if the colimit
\[
M \xto{x} M \xto{x} M \xto{x} ...
\]
is zero, i.e. this is precisely the kernel of the base change $\bD(R) \to \bD(R[x^{-1}])$. The right adjoint of this is simply the inclusion of those modules on which $x$ acts invertibly, in particular it is fully faithful.
\end{example}

\begin{example}
For any dualisable, stable $\infty$-category $\bC$ the sequence
\[
\bC \to \mathbf{Ind}(\bC^{\omega_1}) \to \mathbf{Calk}(\bC)
\]
is a Verdier seqeunce.
\end{example}

\begin{theorem}\label{fib_thm}[Quillen, Waldhausen, Thomasson,...]
Let $\mathbf{C} \to \mathbf{D} \to \mathbf{E}$ be a Verdier sequence in $\Prld$. Then the induced sequence
\[
K(\mathbf{C}) \to K(\mathbf{D}) \to K(\mathbf{E})
\]
is a fibre sequence of spectra. In particular we get a long exact sequence 
\[
\ldots  \to K_1(\mathbf{D}) \to K_1(\mathbf{E}) \to 
K_0(\mathbf{C}) \to K_0(\mathbf{D}) \to K_0(\mathbf{E}) \to K_{-1}(\mathbf{C}) \to ...
\]
\end{theorem}

\begin{corollary}
$K$-theory satisfies Zariski descent as a functor 
\[
\{\text{Affine Opens in }\mathrm{Spec}(R)\}^{\mathbf{op}} \to \mathbf{Sp} \ . 
\]
Concretely if we have central elements $x, y \in R$ such that $1 \in (x,y)$. Then
\[
\begin{tikzcd}
K(R) \arrow{r} \arrow{d}& K(R[x^{-1}])\arrow{d} \\
K(R[y^{-1}] \arrow{r} & K(R[x^{-1}, y^{-1}])
\end{tikzcd}
\] 
\end{corollary}
\begin{proof}
For the diagram of categories
\[
\begin{tikzcd}
\bD(R) \arrow{r} \arrow{d}& \bD(R[x^{-1})\arrow{d} \\
\bD(R[y^{-1}]) \arrow{r} & \bD(R[x^{-1},y^{-1}])
\end{tikzcd}
\]
the induced functor on horizontal kernels is an equivalence since for an $R$-module $M$ which is $x$-torsion, we find that the $R/y$ module $M/y$ has $x$-invertible (as $1 = (x)$ and since it is also $x$-torsion it has to be zero. But $M/y = 0$ means that $y$ acts invertibly, i.e. $M$ is an $R[y^{-1}]$-module. 
\end{proof}

\begin{remark}
One can more generally deduce Nisnevich descent for $K$-theory \cite{thomason2013higher} with the same argument. 
\end{remark}

We also see that the sequence
\[
\bC \to \mathbf{Ind}(\bC^{\omega_1}) \to \mathbf{Calk}(\bC)
\]\
together with the fact that $K(  \mathbf{Ind}(\bC^{\omega_1}) ) = 0$ implies that $K(\bC) = \Omega K(\mathbf{Calk}(\bC))$ as stated in Proposition \ref{prop_last}(3).

\begin{remark}
Since the formula 
$K(\bC) = \Omega K(\mathbf{Calk}(\bC))$ 
holds for any $\bC$ we can also iterate it. This means that we can get that
\[
K(\bC) = \Omega^n K(\mathbf{Calk}^n(\bC))
\]
In particular we find that 
\[
K_{-n}(\bC) = K_0(\Calk^n(\bC)) \ .
\]
This essentialy determines negative $K$-theory and can be used to define it (which we haven't done so far). Together with Corollary \ref{cor_deloop} we get a formula for each $K$-group of $\bC$ as a $K_0$-group of a category. Said differently: $\Calk(-)$ behave like a suspension and $\Shv(\R;-)$ like a loop functor on $\Prld$, at least through the eyes of $K$-theory. 
\end{remark}

\begin{example}
The sequnce $\mathbf{D}(R,I) \to \mathbf{D}(R) \to \mathbf{D}(R/I)$ of Example \ref{almost} is a Verdier sequnce. Here note that $\mathbf{D}(R,I)$ is generally not compactly generated. 
\end{example}

\subsection{Non-unital rings}

We will now deal with non-unital rings $R$. A lot of things that we will do now depend on a base ring $k$, which is a commutative ring spectrum. One example is that $k = \mathbb{S}$ is the sphere spectrum, in which case we are simply dealing with usual ring spectra. If $k = \mathbb{Z}$ is the integers, then we get $\Z$-linear ring spectra, which can be modelled by DGA's. It is perfectly fine to think of those. For the moment, it is also fine to think of ordinary rings. but later we will definitely need things that are homotopically not discrete, i.e. have higher homotopy resp. homology groups. 

For a non-unital $k$-algebra $R$, there is the unitalisation 
\[
R_k^+ = R \oplus k
\]
which is a unital $k$-algebra.  We note that the unitalisation depends on the fixed base $k$.
$R^+_k$ comes with a map of unital $k$-algebras $R_k^+ \to k$, i.e. an augmentation. 
In fact, one can show (see \cite{HA}) that this construction consitutes and equivalence between non-unital $k$-algebra and augmented unital $k$-algebras. Sometimes it is even better to think of non-unital $k$-algebras by means of the map $R_k^+ \to k$. 

\begin{definition}
Let $R$ be a non-unital $k$-algebra. We define
\[
K(R) := \mathrm{fib}( K\left(R_k^+\right) \to K(k)).
\]
\end{definition}

Note that this definition does depend on the base-ring $k$. We suppress this from the notation. We could for example always use $k = \mathbb{S}$  for a definitive definition.  We will soon start investigating, in which generality this does in fact not depend on the base and find that in favourable situations it does not.

\begin{exercise}
Show that if $R$ is unital this definiton agrees with the definition of $K$-theory for unital rings and in particular does not depend on the base $k$.
\end{exercise}

\begin{example}
Consider $R = \mathbb{Z}$ as a ring with the trivial multiplication. Then we have that the unitalisation over $k = \Z$ is given by
\[
R^+ = \Z[x] / x^2 
\]
with the map to $k$ given by evaluation at $0$. If we unitalise with respect to the sphere, we get a different ring \(\mathbb{S} \oplus \mathbb{Z}\). One can show that 
in this case the $K$-theory indeed depends on the base. 
\end{example}

We say that a ring $R$ is locally unital, if it is a filtered colimit of unital rings along non-unital maps. 
\begin{lemma}
Being locally unital for a ring is equivalent to the assertion that 
it has local units, that is for any finite family of elements $r_1,...,r_n \in R$ there exists an idempotent $e \in R$ such that 
\[
e r_i = r_i e = r_i \qquad \text{for all $i$} \ .
\]
\end{lemma}
\begin{proof}
Clearly if $R$ is a filtered colimits of unital rings $R_i$, every finite sequence $r_1,...,r_n \in \pi_*(R)$ lies in some $R_i$ and thus we can take $e$ to be the unit of $R_i$. 

Conversely for an idempotent consider the subring $eRe \subseteq R$
which is a unital ring. Moreover we can consider the poset of idempotents  in $R$ ordered by 
\[
e \leq e'  \quad \text{if} \quad ee' = e'e = e \ .
\]
This is a poset as one easily verifies and it is filtered, since for two idempotents $e, e'$ there exists by assumption a third $e''$ that unitalizes both, i.e. $e \leq e''$ and $e' \leq e''$. For $e \leq e'$ we moreover have and inclusion $eBe \subseteq e'Be'$ . 
Then we consider the ring
\[
\colim_{e} eBe \subseteq R .
\]
This included injectively into $R$ and since each $r \in R$ is unitalised by some $e$, i.e. of the form $r = ere$ it is also surjective.
\end{proof}

\begin{lemma}
Assume that $R$ is locally unital. Then the definition of $K$-theory does not depend on the base.
\end{lemma}
\begin{proof}
Write $R = \colim R_i$. 
It is easy to see that $K(R) = \colim K(R_i)$, but the pieces do not depend on the base. 
\end{proof}

\begin{definition}
A non-unital ring $R$ is called $H$-unital if $R \otimes_{R^+} R \to R$ is an equivalence, that is if it is an almost ideal. We recall again that all tensor products are derived and in fact everything here also works for ring spectra.
\end{definition}

Note that this definition a priori also depends on the base, thus we make it over the sphere. But one can show that this is in fact independent of the base, at least in the connectiv situation. Without connectictity assumptions, the condition over the sphere implies it for all other bases.

\begin{example}\label{example_H}
Locally unital rings are $H$-unital. All $C^*$-algebras when considered as rings are $H$-unital (Wodzicki). If $I \subseteq R$ is an ideal that is idempotent, then $I$ as a non-unital ring is $H$-unital.\footnote{Here if we work with ring spectra, we have to assume that \(I\) and \(R\) are connective.}
\end{example}

For a $H$-unital ring $R$ we define the $H$-unital derived category as
\[
\bD^H(R) := \bD(R^+_k, R) = \mathrm{ker}(\bD(R^+) \to \bD(k)) \ .
\]
That is a module $M$ over $R^+$ is $H$-unital if and only if $R \otimes_{R^+} M \to M$ is an equivalence.

\begin{theorem}[Tamme]\cite{Tamme}
The category $\bD^H(R)$ is independent of the base and dualisable. We have that 
\[
K(R) = K(\bD^H(R))
\] 
which is independent of the base. 
\end{theorem}

More generally we have:
If $R \to S$ is a morphism of ring spectra whose fiber $I$ is $H$-unital, then $I$ is also idempotent in $R$ (in the sense of Example  \ref{almost}) and 
\[
\bD(R,I) = \bD^H(I).
\]
In particular, the category $\bD(R,I)$ only depends on $I$ and not on $R$.

\begin{corollary}
If we have $I \to R \to R/I$ with $H$-unital ideal $I$,  then we have an excision fibre sequence
\[
K(I) \to K(R) \to K(S) \ .
\]
\end{corollary}

The problem is that in general this will not be satisfied in cases of interest, for example for the classical case 
\[
R = \Z[C_p] = \frac{\Z[x]}{x^p -1}  \quad \text{with} \quad I = (x-1) \ .
\]
In this case the ideal is not $H$-unital, so that we do not get a fibre sequence.

\subsection{General Excision}

Let $I \to A \to A/I$ be an ideal sequence of (non-unital) rings. One of the questions one can ask if the induced sequence
\[
K(I) \to K(A) \to K(A/I)
\]
is a fibre sequence. This is generally not the case. If $I, A$ and $A/I$ are $H$-unital it is the case though, since then the functor
\[
\bD^H(A) \to \bD^H(A/I) 
\]
is a Verdier quotient with fibre the H-unital $I$-modules. In general the situation is a bit more complicated and clarified by the following result of Land--Tamme:

\begin{theorem}[Land-Tamme]\label{thm_Land}\cite{land2019k}
For an ideal $I$ in a non-unital ring spectrum $A$ there is another non-unital ring spectrum $A/ I^\circ$ with a factorization $A \to A/ I^\circ \to A/I$ of the projection and a nullhomotopy of the composite from $I \to A \to A/ I^\circ $ (factoring the one for $I \to A \to A/I$) such that
the induced sequence
\[
K(I) \to K(A) \to K(A / I^\circ)
\]
is a fibre sequence. Moreover as a spectrum we have that 
\[
I^\circ = A^+ \otimes_{I^+} I \ . 
\] 
\end{theorem}
\begin{proof}
We apply Land-Tamme's construction to the square 
\[
\begin{tikzcd}
I^+ \arrow{r} \arrow{d}& A^+ \arrow{d} \\
k \arrow{r} &  (A/I)^+
\end{tikzcd}
\] 
and note that the result of there construction, called the circle dot ring, is a ring spectrum $B$ over $(A/I)^+$, thus it fibre over $k$ is a non-unital ring spectrum,
 i.e. $B = (A/I^{\circ})^+$ with the desired properties. For the formula we want to compute
 \[
 I^\circ = \mathrm{fib}(A \to A/I^{\circ}) = \mathrm{fib}(A^+ \to (A/I^{\circ})^+) 
 \]
We have by Land-Tamme's results that as a spectrum
\[
B = A^+ \otimes_{I^+} k
\]
with the map from $A^+ \to B$ given by the inclusion in the left tensor factor, i.e. writing $A^+$ as 
\[
A^+ = A^+ \otimes_{I^+} I^+
\]
and projecting in the second factor. Thus the fibre of this map is given by
\[
 I^\circ =  A^+ \otimes_{I^+} I
\]
as desired. 
\end{proof}

\begin{remark}
The ring sprectrum $A/ I^\circ $ depends on a choice of base $k$. 
\end{remark}

\begin{exercise}
Show that for ordinary rings  $I$ and $A$ the map $A / I^\circ \to A/I$ is an isomorphism on $\pi_0$. This is in fact true for all connective ring spectra with connective ideal
\end{exercise}

As a result, we get that the map 
$K(A / I^\circ) \to K(A/I)$ is an isomorphism in degrees $\leq 1$ provided the base $k$ is connective, since $K$-theory increases connectivity by $1$. In general one should see the failure of the map 
$A / I^\circ \to A/I$ being an equivalence in $K$-theory as quantising the failure of excision. 

In fact, it turns out that Theorem \ref{thm_Land} holds much more generally than for $K$-theory. 

\begin{definition}
A localizing invariant is a functor
\[
E:  \Prld \to \mathbf{Sp}
\]
which sends Verdier sequences to fibre sequences.
\end{definition}

The real theorem of Land-Tamme is that for any localizing invariant $E$ we have a fibre sequence
\[
E(I) \to E(A) \to E(A / I^\circ) . 
\]

\subsection{Homotopy K-theory}

In this section we would like to define homotopy $K$-theory. The key is to define a simplicial commutative ring spectrum
\[
S \in \Delta \mapsto \mathbb{S}^{\Delta^S}
\]
Usually this is done over the integers by defining 
\[
\Z^{\Delta^n} = \frac{\Z[x_0,...,x_n]}{\sum x_i -1}
\]
The problem with this approach however is, that this does not make sense for $\mathbb{E}_\infty$-rings over the sphere since one cannot quotient by arbitrary 
polynomials.\footnote{The only strict elements in $\pi_0(\mathbb{S}[x_0,...,x_n]) = \Z[x_0,...,x_n]$ are $0$ and monomials. These are the elements by which one can quotient in an 
$\mathbb{E}_\infty$ way.} Instead we will take a slighly different approach. Recall that an element in $\Delta$ is a non-empty, finite, linearly ordered set. All of them are of the form 
\[
[n] = \{0<1<...<n\}
\]
for some natural number $n$. An interval is a finite linearly ordered set with a minimum and maximum that are different. All of those are of the form $[n]$ for $n \geq 1$. 
While the category $\Delta$ is given by order preserving maps, the category of intervals is given by intervals and maps that send the minimum to the minimum and the maximum to the maximum. One has that the category of intervals is equivalent to the opposite category $\Delta^\op$ where the functor sends $S \in \Delta$ to the intervall $\mathrm{Cut}(S) = \{ S = T \sqcup T' \mid T < T'\}$ of partitions of $S$ into two subsets.
\begin{definition}
For an intervall $I$ we define a commutative ring spectrum
\[
\widetilde{\mathbb{S}}[I] = \frac{\mathbb{S}[x_i \mid i \in I]}{x_\mathrm{min} = 0, x_{\mathrm{max}} = 1} \ .
\]
For an object $S \in \Delta$ we define 
\[
\mathbb{S}^{\Delta^S} =  \widetilde{\mathbb{S}}[\mathrm{Cut}(S)]  \ .
\]
\end{definition}

\begin{exercise}
Show that over $\Z$ the two definitions of $\Z[\Delta^n]$ agree naturally in $\Delta$. Hint: consider the map
\[
 \frac{\mathbb{Z}[y_0,\ldots,y_{n+1}]}{y_0 = 0, y_{n+1} = 1} \to 
\frac{\Z[x_0,\ldots,x_n]}{\sum x_i -1} 
\]
sending $y_k$ to $\sum_{i < k} x_i$ with inverse sending $x_k$ to $y_{k+1} - y_k$. Convince yourself that this is natural in $\Delta^{\mathrm{op}}$ (best try to write down the map for a general $S$ naturally in $S$). 

\end{exercise}

\begin{definition}
For every dualisable stable $\infty$-category $\mathbf{C}$ we define a simplicial category by
\[
\Delta^{\mathrm{op}} \ni S  \mapsto \mathbf{C}^{\Delta^S} := \mathbf{C} \otimes \bD(\mathbb{S}^{\Delta^S}) \ .
\]
\end{definition}

The following definition follows Weibel \cite{weibelhomotopy} and Cisinksi-Khan\cite{Cisinski-Khan}.
 
\begin{definition}
We define \emph{homotopy $K$-theory} as
\[
K_H(\mathbf{C}) = \colim_{S \in \Delta} K( \mathbf{C}^{\Delta^S} ) \ .
\]
For a unital ring $R$ we set $K_H(R) := K_H(\bD(R))$ and for a non unital ring $R$ we set $K_H(R) = \mathrm{fib}(K_H(R^+_k) \to K(k))$. 
\end{definition}

\begin{proposition}
We have that the canonical map
\[
KH(\mathbf{C}) \to KH(\mathbf{C}^{\Delta^n}) 
\]
induced by $\Delta^n \to \Delta^0$ is an equivalence
for every $n$, in particular
\[
K_H(\bC) =  \colim_{S \in \Delta} K_H( \mathbf{C}^{\Delta^S} )
\]
\end{proposition}

\begin{proposition}[Weibel, Cisinski-Khan, Land--Tamme]\label{corollary}
Homotopy $K$-theory is a localizing invariant which commutes with filtered colimits and that is
truncating in the sense that for any connective ring spectrum the map
\[
K_H(R) \to K_H(\pi_0(R))
\]
is an equivalence.
\end{proposition}

\begin{corollary}
Homotopy $K$-theory satisfied excision for connective ring spectra, that is
\[
K_H(I) \to K_H(A) \to K_H(A/I)
\]
is a fibre sequence. Moreover for non-unital rings the value doesn't depend on the base (as long as the base $k$ is connective).
\end{corollary}
\begin{proof}
Follows since 
\[
A/I^{\circ} \to A/I
\]
is an isomorphism on $\pi_0$ and since for $R$ non-unital and any base we have the fibre sequence
\[
R \to R^+_k \to k \qedhere
\]
\end{proof}

For a localizing invariant $E: \Prld \to \bD$ we set $E(R)$ to mean $E(D(R))$ for $R$ connective and unital. For non-unital rings we set as before $E(R) = \mathrm{fib}(E(R^+_k) \to E(k))$.

\begin{definition}
A functor $E: \Prld \to \mathbf{Sp}$ is called \emph{truncating}, if $E(R) \to E(\pi_0R)$ is an equivalence for each connective ring spectrum $R$. 
\end{definition}

The assertion of Corollary \ref{corollary}  holds for $K_H$ replaced by an truncating, localizing invariant. 

\begin{exercise}
Show that for $E$ is is equivalent that $E(R) \to E(\pi_0R)$ is an equivalence for all (possibly non-unital) rings versus only for unital rings (as usual under the assumption that the base is connected. 
\end{exercise}

\section{Non-commutative motives}

Now we would like to introduce the category of non-commutative motives \cite{BGT}. To this end recall that a localizing invariant was a functor
\[
E: \mathrm{Pr}^L_\mathrm{dual} \to \bE
\]
which sends Verdider sequences to fibre sequences. We say that a localizing invariant is \emph{finitary} if it commutes with filtered colimits, i.e. sends filtered colimits in 
$\mathrm{Pr}^L_\mathrm{dual}$ to filtered colimits of spectra. Examples of filtered colimits in $\mathrm{Pr}^L_\mathrm{dual}$ include
\[
\dirlim \mathbf{Mod}_{R_i} = \mathbf{Mod}_R
\]
for a filtered diagram $R = \dirlim R_i$ of rings or ring spectra. 

\begin{example}
$K$-theory and homotopy $K$-theory are finitary localizing invariants. Similarly topological Hochschild homology is a finitary localizing invariant. Topological periodic and cyclic homology are localizing invariants that are not finitary. 
\end{example}

The goal of this section is to investigate localizing invariants. We will do this following the philosophy of Grothendieck by introducing a universal category of motives, over which all of them factor. Note that following the philosophy of Kontsevich, Tabuada \cite{tabuada-motives}, Kaledin etc, we should think of stable $\infty$-category (or DG-categories in the approach relative to $\Z$) as non-commutative schemes. This is similar to how we think of $C^*$-algebras as non-commutative topological spaces. Therefore why call this category non-commutative motives.

\subsection{Definition of motives}

\begin{defprop}
There is a universal, finitary, localizing invariant 
\[
M: \Prld \to \mathbf{NCMot} \ .
\]
In other words: $\mathbf{NCMot}$ is a stable, presentable $\infty$-category with finitary localizing invariant which assigns to every dualisable stable $\infty$-category $\mathbf{C}$ its motive $M\mathbf{C}$. Any other finitary localizing invariant
\[
E: \Prld \to \mathbf{E}
\]
factors uniquely as
\[
\begin{tikzcd}
\Prld \arrow{r}{E} \arrow{d}{M} & \bE\\
\mathbf{NcMot} \arrow{ru}{E'} &  
\end{tikzcd}
\]

where $E'$ is colimit preserving.
\end{defprop}
The construction of the category $\NcMot$ is very formal, it is a localization of the category of functors $\mathrm{Fun}((\Prld)^{\mathrm{op}}, \mathrm{Sp})$ that preserve cofiltered limits (i.e. send filtered colimits in $\Prld$ to limits of spectra. Then the functor $M$ is the Yoneda embedding followed by the localization.

Note that $K$-theory itself is a finitary localizing invariants, thus induces a functor 
\[
K: \mathbf{NCMot}  \to \mathbf{Sp} \ .
\]
The non-formal part is the following, which is an analogue of a statement in the $C^*$-world.
\begin{theorem}[Barwick, Blumberg--Gepner--Tabuada]
$K$-theory is corepresentable by the motive of the category of spectra, that is we have that 
\[
K(\mathbf{C}) = \mathbf{Map}_{\mathbf{NCMot}}(M\mathbf{Sp}, M\mathbf{C})
\]
In particular we also have 
\[
K_n(\mathbf{C}) = [M\mathbf{Sp}, M\mathbf{C}[n]]
\]
\end{theorem}

\begin{corollary}[Uniqueness of $K$-theory]
Assume we have any localizing invariant $E$ with $\pi_0(E) \cong K_0(E)$ as functors to abelian groups. Then $E \simeq K$ as functors to spectra. In other words: there is a unique localizing extension of $K_0$ to a functor of spectra. 
\end{corollary}
\begin{proof}
First of all, we note that it follows that $E$ is finitary (since this can be checked on $K_0$ as we have shift operators given by $\Calk(-)$ and $\Shv(\R,-)$.
By Yoneda (for functors with values to abelian groups on the homotopy category of $\mathbf{NcMot}$), the transformation $\pi_0K \to \pi_0E$ is given by a class $a \in \pi_0(E(\mathbf{Sp}))$ which using Yoneda for functors to spectra gives rise to a map $K \to E$ that induces an isomorphism on $\pi_0$. Then by the formulas for looping and delooping we get that it is an isomorphism on all $\pi_i$. 
\end{proof}

\begin{definition}
\emph{Bivariant $K$-theory} is defined as 
\[
\mathbf{KK}^\alg(\mathbf{C}, \mathbf{D}) = \mathbf{Map}_{\mathbf{NCMot}}(M\mathbf{C}, M\mathbf{D}) \ .
\]
For rings $R,S$ we define 
\[
\mathbf{KK}^\alg(R, S) :=   \mathbf{KK}^\alg(\bD(R), \bD(S))
\]
\end{definition}

\begin{definition}
A morphism in $\Prld$ is called a \emph{motivic equivalence}, if it induces an equivalence in $\mathbf{NcMot}$, 
Motivic equivalence. A functor $\Prld \to \bE$ is called a \emph{motivic invariant}, if it sends motivic equivalence to equivalences in $\mathbf{E}$. 
\end{definition}

\begin{example}
Every finitary localizing $E$ invariant is motivic, since it factors over $\mathbf{NcMot}$.
Also $\mathrm{TC}$ and $\mathrm{TP}$ are motivic invariant, even though they are not finitary. This is seen by the fact that they are build functorially from $\mathrm{THH}$, which in turn is finitary and thus motivic.

One can in fact show, that also all $\omega_1$-finitary localizing invariants are motivic (Efimov--Sosnilo--Ramzi--Winges), but we will not need this.
\end{example}

The upshot is that virtually all invariants we care about are motivic. 

\begin{theorem}[Sosnilo--Ramzi--Winges]
The category $\NcMot$ is the localization of $\Prld$ at the motivic equivalences, i.e.
\[
\NcMot = \Prld[W_{\mathrm{mot}}^{-1}]
\] 
\end{theorem}

In particular we an really think of $\NcMot$ as having the same objects as $\Prld$ and as morphisms the KK-spectra (similar to the $C^*$-situation). 
Also note that the functor $\bC \mapsto \Calk(\bC)$ acts as a suspension (i.e. upshift) on $\NcMot$ and $\bC \mapsto \Shv(\mathbb{R},\bC)$ as a loopfunctor (i.e. downshift). \footnote{Note that $\Shv(\mathbb{R}, \Shv(\mathbb{R}^{n-1}, \bC)) \simeq \Shv(\mathbb{R}^n,\bC))$.}

One can in fact define an $\infty$-categorical version of a category of cofibrant objects structure on $\Prld$ with cofibrations the fully faithful inclusions and weak equivalences the motivic equivalences.

\begin{remark}
One can in fact show that $\NcMot$ is already a Dwyer-Kan localization of the category of either unital ring spectra or $H$-unital ring spectra (with some subtleties) as done in joint work with Krause and P\"utzst\"uck \cite{nkp}. 
\end{remark}

\begin{corollary}
Every motivic invariant 
$E: \mathrm{Pr}^L_\mathrm{dual} \to \bE$
induces a unique functor 
\[
E': \NcMot \to \bE \ .
\]
If $E$ is localizing, then $E'$ preserves finite colimits, i.e. is exact. If $E$ is finitary, then $E'$ preserves filtered colimits. 
\end{corollary}

In general we shall not distinguish between $E$ and $E'$ and denote both by $E$ (i.e. evaluate at a motive or a category whenever convenient). Sometimes if we need to be careful we will make this distinction. 

\subsection{Sheaves as a non-commutative motive}

Now we shall analyse the motive $M(\Shv(X,\bC))$ for a locally compact Hausdorff space $X$.  We already know that 
\[
K(\Shv(X,\bC)) = \Gamma_c(X, \bC) \ .
\]
As a short digression and to properly state the result, we start by reviewing/describing the six functor formalism of locally compact Hausdorff spaces, which is intimately related to the result. 

We consider sheaves on $X$ with values in any presentable stable $\infty$-category $\bE$. Here $\bE$ a priori does not need to be dualisable. At some point we need this, but not for the moment, so we change from $\bC$ to $\bE$.

Let \(f \colon X \to Y\) be a continuous map of locally compact Hausdorff spaces. To such a functor, we may associate the \emph{pushforward} functor \[f_* \colon \Shv(X, \bE) \to \Shv(Y, \bE), \quad \mathcal{F} \mapsto (U \mapsto \mathcal{F}(f^{-1}(U)),\] which is right adjoint to the \emph{pullback} functor \[f^* \colon \Shv(Y, \bE) \to \Shv(X, \bE).\] The second adjoint pair is given by the \emph{proper pushforward} functor \begin{multline*}
f_{!} \colon \Shv(X, \bE) \to \Shv(Y, \bE), \\
\mathcal{F} \mapsto f_!(\mathcal{F})(U) = \colim_{K \subset f^{-1}(U): K \to U \text{ proper }}\mathsf{fib}(\mathcal{F}(f^{-1}(U)) \to \mathcal{F}(f^{-1}(U) \setminus K))),
\end{multline*} whose (abstractly defined) right adjoint \(f^! \colon \Shv(Y, \calE) \to \Shv(X, \calE)\) is called the \emph{exceptional image} functor. Specialising to the map \(f \colon X \to *\), we get \(\Shv(*, \calE) \simeq \calE\) and functors
\begin{align*}
& \Gamma(X,\mathcal{F})  = f_*(\mathcal{F}) = \mathcal{F}(X) \\
& \underline{E}  = f^*(E) \\
& \Gamma_c(X,\mathcal{F}) = f_!(\mathcal{F}) =  \colim_{K \subset X: K \text{ compact }}\mathsf{fib}(\mathcal{F}(X) \to \mathcal{F}(X \setminus K))
\\ 
& f^!E
\end{align*}

\noindent of which the first and the third are called the \emph{global sections} and \emph{global sections with compact support}, respectively. We shall also note that 
for a spectrum $E \in \Sp$ the compositions gives rise to know things: 
\begin{align*}
&\Gamma(X, \underline{E}) = f_*f^*(E)  & \text{sheaf comology of $X$ with values in $E$} \\
&\Gamma_c(X, \underline{E}) = f_!f^*(E)  & \text{compactly supported sheaf comology of $X$ with values in $E$} \\
&f_!f^!(E)  & \text{sheaf homology of $X$ with values in $E$} \\
&f_*f^!(E)   & \text{locally finite sheaf homology of $X$ with values in $E$} \\
\end{align*}
The latter is also known as Borel-Moore homology. 
For a nice space $X$ (e.g. a CW complex or more generally a space of singular shape) the cohomology is literally the limit over the $\infty$-groupoid $\mathrm{Sing}X$ (i.e. singular cohomology) and the homology is the colimit (i.e. singular homology). In general there is an pro-homotopy type $\Pi_\infty(X) \in \mathrm{Pro}(\mathcal{S})$ called the shape of $X$ such that 
\[
f_*f^*E = E^{\Pi_\infty(X)} \ .
\]
where the right hand side is define as a colimit of limits in $\calE$ (the colimit over the indexing category of the pro-object and the limit over the terms) and similar formulas for the others. This goes to say that  generally these functors can be evaluated purely in terms of the category $\bE$ and an invariant associated with a space $X$.

For a left adjoint functor $g: \bE \to \bE'$ one gets an induced base change functor $\bar g: \Shv(X, \bE) \to \Shv(X, \bE')$ in the obvious way (involving a sheafification) and one can show that the functors $f_!$ and $f^*$ commutes with this base change in the obvious way: 
\[
\bar g f_! = f_! \bar g \qquad \bar g f^* = f^* \bar g \ .
\]
\begin{theorem}[Efimov]\cite{efimov2024k}
For any finitary localizing invariant $E: \Prld \to \bC$ we have
\[
E(\Shv(X,\bC)) = f_!f^* E(\bC) = \Gamma_c(X, \underline{E(\bC)})
\]
\end{theorem}

In particular we have that the motive $M(\Shv(X,\bC)) =  f_!f^* M(\bC)$ and in view of the fact that a finitary localizing invariant gives rise to a functor $\NcMot \to \bE$ the result is in fact equivalent to the computation for $M$. 

Let us say some words about the proof: the crucial idea is to consider the assigment $X \mapsto K(\Shv(X, \bC)$ as a $\mathcal{K}$-sheaf on the category of compact Hausdorff spaces itself. The stalks are simply given by $K(\bC)$ and the idea is to show that this sheaf is indeed constant. One can produce a map from the constant sheaf $\underline{K(\bC)}$ that is an equivalence on stalks. In order to conclude that this is enough one crucially uses two facts:
\begin{enumerate}
\item The category $\NcMot$ is dualisable itself (another deep result by Efimov).
\item Every compact Hausdorff space an be embedded into a Hilbert cube, hence a space that is hypercomplete, so that equivalences can be tested on stalks. 
\end{enumerate}
These facts are the crucial facts that are used.

\begin{corollary}\label{cor_kk}
For any $\bC \in \Prld$ we have that 
\[
\mathbf{KK}^\alg(\Shv(X;\Sp), \bC) \simeq f_*f^! K(\calC) \ .
\]
is given by Borel Moore homology. More generally we have that the inner Hom in motives is given by
\[
\underline{\mathbf{Hom}}_{\NcMot}(M\Shv(X;\Sp), M\calC) = f_*f^! M\calC \ .
\]
\end{corollary}
Note that 
$K( \underline{\mathbf{Hom}}_{\NcMot}(M\Shv(X;\Sp), M\calC)) = \mathbf{KK}^\alg(\Shv(X;\Sp), \calC)$
and that $K$ theory $\NcMot \to \Sp$ is right adjoint, thus commutes with $f_*f^!$. Therefore the first statement is really a special case of the second. 
\begin{proof}
We will prove the first, the second works similar using inner homs and is mostly notationally more complicated (plus uses some properties of six-functor formalisms that we have not talked about). 
Using the adjunctions we have
\begin{align*}
\mathbf{Map}_{\NcMot}(M\Shv(X;\Sp), M\calC) &= \mathbf{Map}_{\NcMot}(f_!f^*M\Sp, M\calC) \\
& = \mathbf{Map}_{\Shv(X,\NcMot)}(f^*M\Sp, f^*M\calC)\\
& = \mathbf{Map}_{\Shv(X,\NcMot)}(M\Sp, f_*f^!M\calC) \\
& = K(f_*f^!M\calC) = f_*f^!K(M\calC)  
\end{align*}
and since $K$-theory is limit-preserving it commutes with all the functors in question ($f_*$ and $f^!$). 
\end{proof}

In particular for compact Hausdorff spaces $X$ we get that 
\[
\mathbf{KK}^\alg(\Shv(X, \Sp), \calC) = \Pi_\infty(X) \otimes \calC
\]
For a locally compact Hausdorff space $X$ we get that 
\[
\colim_{K \subseteq X \text{ compact}} \mathbf{KK}^\alg(\Shv(K, \Sp), \calC)  = \Pi_\infty(X) \otimes \calC
\]
which for $X = BG$ is the source of the assembly map.

\subsection{$\mathbb{A}^1$-invariant motives}

Note that if $E$ is any finitary localizing invariant, we can form a homotopy version of $E$ as
\[
EH(\calC) = \colim_{S \in \Delta^{\mathrm{op}}} E(\calC^{\Delta^S}) 
\]
similar to homotopy $K$-theory. And similar to homotopy $K$-theory, we see that $EH$ is homotopy invariant in the sense that $EH(\calC) \xto{\simeq} EH(\calC^{\Delta^S})$, in particular $EH \to (EH)H$ is an equivalence.  Morever $EH$ is a finitary, localzing invariant.

\begin{definition}
A finitary, localzing invariant $E$ is called \emph{$\mathbb{A}^1$-invariant} if the map 
\[
E(\calC) \to E(\calC^{\Delta^1}) 
\]
is an equivalence. This is equivalent to the assertion that the map $E \to EH$ is an equivalence.
\end{definition}

\begin{proposition}
There is a universal, finitary localizing, $\mathbb{A}^1$-invariant functor
\[
M_{\mathbb{A}^1}: \Prld \to \NcMot_{\mathbb{A}^1} \ .
\]
The canonical map $\NcMot \to \NcMot_{\mathbb{A}^1}$ is a Verdier Quotient map, in particular a localization. It is smashing, i.e. given by tensoring with an idempotent in $\NcMot$ and this idempotent is given by the colimit
\[
M_{\mathbb{A}^1}\Sp = \colim_{S \in \Delta^{\mathrm{op}}} M\calD(\mathbb{S}^{\Delta^S}) \ .
\]
in commutative algebras in $\NcMot$. 
\end{proposition}

Note that the object $M_{\mathbb{A}^1}\Sp$ is clearly a commutative ring object in $\NcMot$. Part of the assertion is that it is idempotent, i.e. that
\[
M_{\mathbb{A}^1}\Sp \otimes M_{\mathbb{A}^1}\Sp \to M_{\mathbb{A}^1}\Sp
\]
is an equivalence and that $\NcMot_{\mathbb{A}^1}$ is simply modules over this idempotent algebra in $\NcMot$.

\begin{proof}
The fact that the universal invariant exists and is a further Bousfield localzing of $\mathrm{NcMot}$ is formal. Restriction along the functor
 \[
 \NcMot \xto{p} \NcMot_{\mathbb{A}^1} 
 \]
 is given by 
Then the functor
\[
\Fun^L(\NcMot_{\mathbb{A}^1}, \bE) \to \Fun^L(\NcMot, \bE)
\]
is given by the inclusion of finitary localizing, $\mathbb{A}^1$-invariants into finitary localizing invariants.
This functor has a left adjoint $E \mapsto EH$ given by
\[
E_H(M\calC) = \colim_{S \in \Delta^{\mathrm{op}}} E(M\calC^{\Delta^S}) = E(M\calC \otimes M_{\mathbb{A}^1}\Sp) \ .
\]
This implies by the Yoneda Lemma that 
\[
 \NcMot \xto{p} \NcMot_{\mathbb{A}^1} 
\]
has right adjoint $R_p$ in $\Pr^L$, such that the composition $R_p \circ p$ is given by tensoring with $M_{\mathbb{A}^1}\Sp$. This implies the claim and exhibits 
$\NcMot_{\mathbb{A}^1} $ as modules over $M_{\mathbb{A}^1}\Sp$.
\end{proof}

\begin{corollary}
Homotopy $K$-theory is corepresented in $\NcMot_{\mathbb{A}^1}$ by $M_{\mathbb{A}^1}\Sp$.
\end{corollary}

In fact, we can even define bivariant $K_H$-theory as
\[
\mathrm{KK_H}(\calC, \calD) = \map_{\NcMot_{\mathbb{A}^1}} (M_{\mathbb{A}^1}\calC, M_{\mathbb{A}^1}\calD) =  \map_{\NcMot}(M \calC, M \calD \otimes 
M_{\mathbb{A}^1}\Sp)
\]

\begin{example}
We have that $M_{\mathbb{A}^1} \Shv(X,\calC) = f_! f^* M_{\mathbb{A}^1} \calC$ and 
\[
K_H( \Shv(X,\calC)) = \Gamma_c(X, K_H(\calC)) \ .
\]
That is for regular rings $R$ the $K$-theory and homotopy $K$-theory of $ \Shv(X,\calD(R))$ agree (in fact the motives already do). 
\end{example}

\begin{definition}
A motivic invariant with corresponding functor $E: \NcMot \to \calE$ is called \emph{strongly $\mathbb{A}^1$-invariant} if 
\[
E(M\calC) \to E(M\calC \otimes M_{\mathbb{A}^1}\Sp)
\]
is an equivalence. 
\end{definition}

We can reformulate this condition as follows: a morphism $\calC \to \calD$ is called $\mathbb{A}^1$-motivic equivalence if the induced map $M \calC \to M\calD$ becomes an equivalence after tensoring with $M\Sp_\Delta$, equivalently if $M_{\mathbb{A}^1}\calC \to M_{\mathbb{A}^1}\calD$ is an equivalence. 
Then a strongly $\mathbb{A}^1$-invariant motivic invariant $E$ is a functor that sends $\mathbb{A}^1$-motivic equivalences to equivalences, in particular induces a functor 
\[
E': \NcMot_{\mathbb{A}^1} \to \calE \ .
\]
If $E$ is localizing then $E'$ is exact, if $E$ is finitary, then $E'$ preserves filtered colimits.  

\begin{warning}
It is for a general motivic invariant $E$ not true that strong $\mathbb{A}^1$-invariance is equivalent to 
\[
E(\calC) \to E(\calC^{\Delta^1})
\]
being an equivalence. This is only true for finitary, localizing $E$. This is why we call it \emph{strong}  $\mathbb{A}^1$-invariant, where the latter would be called \emph{weak} $\mathbb{A}^1$-invariance. Thus distinction will be become important soon.
\end{warning}

We can make a general motivic invariant $E$ strongly $\mathbb{A}^1$-invariant by 
\[
E_H(M\calC) = E(M\calC \otimes M_{\mathbb{A}^1}\Sp)
\]

\begin{warning}
For a general motivic invariant, it is not true that 
\[
E_H(\calC) = \colim_{S \in \Delta^{\mathrm{op}}} E(\calC \otimes \calD(\mathbb{S}^{\Delta^S}))
\]
This is only true if $E$ is finitary and localizing. Otherwise we can not commute the colimit  outside the $E$. 
\end{warning}

\subsection{$\mathbb{A}^1$-acyclicity}

\begin{definition}
A motivic invariant $E$ is called \emph{$\mathbb{A}^1$-acyclic}, if $E_H = 0$. 
\end{definition}

\begin{lemma}
The category of $\mathbb{A}^1$-acyclic, motivic invariants is closed under all limits and colimits in motivic invariants.
\end{lemma}
\begin{proof}
The functor $E \mapsto E_H$ evidently commutes with all limits and colimits in $E$. Thus the kernel is closed under limits and colimits.  
\end{proof}

\begin{example}
$\mathrm{THH}$ is $\mathbb{A}^1$-acyclic. We have to show that homotopy $\THH$ vanishes. Since $\THH$ is finitary and localizing, we have the `usual' formula for homotopyfication. Moreover to show that homotopy $\THH$ vanishes it is enough to check it on the sphere, since everything else is a module over that. Thus we need to see that homotopy $\THH$ of rings is zero. This in turn now amounts to proving that 
\[
\colim_{S \in \Delta^{\mathrm{op}}} \THH(R^{\Delta^S}) = 0
\]
But $\THH$ itself is defined using the cyclic Bar construction, in particular it is a colimit each of which terms are tensor powers of $R$. Since geometric realizations commute with colimits and tensor powers it therefore suffices to show that for any ringspectrum $R$ we have that 
\[
\colim_{S \in \Delta^{\mathrm{op}}} R^{\Delta^S}
\]
vanishes. This is left as an exercise.
\end{example}

\begin{exercise}
Finish the previous proof, that is show that 
\[
\colim_{S \in \Delta^{\mathrm{op}}} R^{\Delta^S} = 0
\]
for every ring $R$. 
\end{exercise}

\begin{corollary}
Topological negative, topological periodic and topological cyclic homology are $\mathbb{A}^1$-acyclic. 
\end{corollary}
\begin{proof}
They are obtained from $\THH$ by limits, colimits and fibres. 
\end{proof}

\begin{remark}
It might be rather surprising that periodic homology is $\mathbb{A}^1$-acyclic since it is sometime (e.g. in rational situations) even (weakly) $\mathbb{A}^1$-invariant.
But since it is not finitary this is not a contradiction at all and even highlights the subtleties in these notions. 
\end{remark}

We let $I$ be the non-commutative motive given by $I := \mathrm{fib}(M\mathbb{S} \to M_{\mathbb{A}^1}\mathbb{S})$ for the next result.

\begin{proposition}
For every motivic invariant $E$ there is a tranformation $E \to E_C$ such that $E_C$ is $\mathbb{A}^1$-acyclic and initial among all such under $E$. 
We have 
\begin{align*}
E_C(M\calC) & = E( \underline{\mathrm{Hom}}_{\NcMot}(I, M\calC)) \\
& E( \lim_{S \in \Delta}  \underline{\mathrm{Hom}}_{\NcMot}(M \mathbb{S}^{\Delta^S}, M\calC))
\end{align*}
Moreover the square
\[
\begin{tikzcd}
E\arrow{r} \arrow{d}& E_H \ar[d]\\
E_C\arrow{r} &  (E_H)_C
\end{tikzcd}
\]
is a pullback for each $E$.
\end{proposition}

\begin{proof}
Since $M_{\mathbb{A}^1}\mathbb{S}$ is an idempotent in $\NcMot$ we get by abstract non-sense a recollement. The construction we have given above is the dual picture. 

We will use that since $M_{\mathbb{A}^1}\mathbb{S}$ is an idempotent we have that $I \otimes M_{\mathbb{A}^1}\mathbb{S} = 0$ and $I \otimes I \to I$ is an equivalence. 
Thus to see that 
\[
E_C(M\calC) = E( \underline{\mathrm{Hom}}_{\NcMot}(I, M\calC))
\]
is the universal $\mathbb{A}^1$-acyclic invaiant under $E$ we first observe that it is indeed acyclic since
\[
\underline{\mathrm{Hom}}_{\NcMot}(I, M\calC \otimes M_{\mathbb{A}^1}\mathbb{S}) = 0
\]
and to see that is indeed universal we use that if $E$ is already $\mathbb{A}^1$-acyclic, then $E \to E_C$ is an equivalence.
\end{proof}

Note that $(E_C)_H = 0$ by definition of acyclicity. Also note that the square looks like a general completion/localization fracture square. In this way of viewing it, $E_H$ plays the role of the `completion' of $E$ and $E_C$ the role of the localization of $E$. 

\begin{remark}
From the fracture perspective one sees that the fibre $E \to E_C$ is the universal $\mathbb{A}^1$-torsion approximation, where a motivic invariant $E$ is called $\mathbb{A}^1$-torsion if $E_C = 0$.  

There is an equivalence between $\mathbb{A}^1$-torsion invariants and strong $\mathbb{A}^1$-invariants given by sending a torsion invariant $E$ to $E_H$ and conversely a  $\mathbb{A}^1$-invariant to the fibre of $E \to E_C$. 
\end{remark}

\begin{remark}\label{rem_limit}
One can try to be more explict about the motive $\underline{\mathrm{Hom}}_{\NcMot}(M \mathbb{S}^{\Delta^S}, M\calC)$. The reduced version It can indeed be written as an inverse limit (in motives) of motives of categories, namely 
\[
\underline{\mathrm{Hom}}_{\NcMot}(M \mathbb{S}[x_1,...,x_k], M\calC) = \lim_{n_1,..,n_k} \Omega ^k M \frac{\mathbb{S}[x_1\ldots ,x_k]}{x_1^{n_1},\ldots,x_k^{n_k}}
\]
If the functor $E$ for example commutes with limits as a functor $\NcMot \to \mathcal{E}$ then we get that 

\end{remark}

\begin{theorem}[Hesselholt, Betley--Schlichtkrull, McCandless, Efimov]\cite{Hesselholt, McCandless, efimov-rigidity, betley-schlichtkrull}
For connective ring spectra $R$ we have 
\[
K_C(R) \simeq \mathrm{TC}(R) \ .
\]
(More generally this holds for dualisable, stable $\infty$-category induced from additive $\infty$-categories). 
For general categories it is equivalent to Bloch's version of $\mathrm{TC}$ build from $K$-theory of curves, i.e.
\[
K_C(\calC) = \mathrm{TC}^{\mathrm{cur}}(\calC)
\]
\end{theorem}

This in particular shows that $ \mathrm{TC}$ satisfies a universal property under $K$-theory, namely that is the universal way of making $K$-theory $\mathbb{A}^1$-acyclic. Said 
differently, $\mathrm{TC}$ is the initial pointed, localizing $\mathbb{A}^1$-acyclic invariant.  
Here and henceforth we do not distinguish carefully between $\mathrm{TC}$ and  $\mathrm{TC}^{\mathrm{cur}}(\calC)$ since they agree on connective ring spectra. Generally it is always the latter that satisfies all the nice properties. 

We could also say that the fibre $K^{\mathrm{inv}}(\calC) := \mathrm{fib}(K(\calC) \to \mathrm{TC}^{\mathrm{cur}}(\calC))$ is the universal $\mathbb{A}^1$-acyclic approximation to $K$-theory.

\section{Assembly maps}

In this section we would like to explain how one can use the technology developed in the last sections to give concrete models for assembly maps and give some applications.

To this end, recall that the classical assembly map in $K$-theory is the map
\[
BG \otimes K(\mathbb{Z}) \to K(\mathbb{Z}[G])
\]
obtained as the colimit interchange map (since $\calD(\Z[G]) = BG \otimes \calD(\Z)$ in $\Prld$).

One of the reasons we care about this map is that the Farrell-Jones conjecture asserts that for torsion free groups $G$, the assembly map
\[
BG \otimes K(\mathbb{Z}) \to K(\mathbb{Z}[G])
\]
is an equivalence. This question plays an important role in many applications ranging from ring theory to geometric topology. For example we can consider the following conjecture.

\begin{conjecture}[Borel]
Assume that $M$ and $N$ are closed, aspherical topological manifolds $M$ and $N$ (aspherical means that the universal cover is contractible) with isomorphic fundamental groups. Then $M$ and $N$ are homeomorphic. \footnote{More precisely every homotopy equivalence between $M$ and $N$ is homotopic to a homeomorphism}.
\end{conjecture}

\begin{exercise}
Show that for a closed, aspherical manifold the fundamental group is torsion free. Hint: Assume the fundamental group had torsion and consider the homology.
\end{exercise}
It turns out, that for dimension $n \geq 5$ this conjecture can be translated through surgery theory to the question whether the assembly map for the fundamental group is an equivalence, i.e. into the Farrell-Jones conjecture for $\pi_1(M)$ (actually one really needs the $L$-theoretic version.). 

We finally note that in fact, there is a more general version of assembly: one can replace $BG$ by any homotopy type $X$ and the $\calD(\Z)$ by any dualisable, stable $\infty$-category $\calC$. Note that $\calD(\Z[G]) = \calD(\mathbb{Z})^{BG} = \Fun(BG, \calD(\mathbb{Z}))$. Generally we get the assembly map as
\[
X \otimes K(\calC) \to K(\calD(\calC^X)) \ .
\]

\subsection{The assembly map using algebraic KK}\label{sec81}

Recall that we have  
\[
\colim_{K \subseteq X} \mathbf{KK}^\alg(\Shv(K), \mathbf{C}) \;\simeq\; \Pi_\infty(X) \otimes K(\mathbf{C})
\]
for a locally compact Hausdorff space \(X\).\footnote{Here $\Pi_\infty(X)$ is the shape. We shall not dwell onto this point. The reader might just imagine that $X$ is a nice space (e.g. CW complex) in which case this is equivalent to the singular complex $\mathrm{Sing}(X)$.}
This already represents a significant step forward: in general, it is quite difficult to construct elements in the homotopy groups of \(\Pi_\infty(X) \otimes K(\mathbf{C})\). Once we can realise this object as the \(K\)-theory of a category---or as here, in terms of KK-theory---we gain access to concrete ways of constructing elements.
In our present situation this means the following:  
if, for every compact \(K \subseteq X\), we have a strongly continuous functor  
\[
F_K \colon \Shv(K) \longrightarrow \mathbf{C}
\]
natural in \(K\), i.e.\ compatible with restriction along inclusions of compact subsets of \(X\), then we obtain an element in  
$
\pi_0\!\left(\colim_{K \subseteq X} \mathbf{KK}^\alg(\Shv(K), \mathbf{C})\right),
$
that is, a class in  
$
\pi_0\!\big(\Pi_\infty(X) \otimes K(\mathbf{C})\big).
$
This already yields a rather pleasant formalism for producing elements in the source of the assembly map which is quite useful, for 
for instance when attempting to build an inverse to the assembly map. 

However, we must also understand how these elements are mapped under the assembly map. 
To this end, we will now give a description of the assembly map in the same language. The idea is to make use of a ``diagonal" class in KK-theory and then proceed in close analogy with familiar constructions from operator theory, using Kasparov's topological KK-category.

In order to define the diagonal class, it is helpful to think of functors 
$X \to \Sp$
as \emph{local systems} on \(X\). More precisely, the following theorem provides a vast generalization of classical covering space theory. As in the classical setting, certain niceness assumptions on \(X\) are required (e.g., locally path-connected and semi-locally simply connected). In our setting, we assume that \(X\) has \emph{locally contractible shape}.
This condition holds, for example, if \(X\) is locally contractible and hypercomplete---such as when \(X\) has finite homotopy dimension. The class of such spaces includes ANRs (absolute neighborhood retracts).

\begin{theorem}[...,Lurie]
Let \(X\) have locally contractible shape. Then for any \(\infty\)-category \(\mathbf{C}\), the functor category 
\[
\Fun(X, \mathbf{C})
\]
is equivalent to the category of locally constant sheaves on \(X\), regarded as a full subcategory of \(\Shv(X; \mathbf{C})\).
\end{theorem}

\begin{definition}
Let $X$ be a locally compact Hausdorff space which is of locally constant shape. We define an object 
\[
D_X \in \Shv(X, \Sp) \otimes \Sp^X \simeq \Fun(X, \Shv(X, \Sp) )
\]
called the diagonal class as follows: we consider the inclusions 
\[
\Sp^{X \times X} = \Sp^X \otimes \Sp^X \subseteq  \Shv^{\mathrm{loc}}(X, \Sp) \otimes \Sp^X 
\]
and take the element in the source given by the pushforward $\Delta_\natural(\mathbb{S})$ of the constant sphere. 
\end{definition}

Here note that $\Delta_\natural$ is the left adjoint to $\Delta^*$. Also note that concretely, we can consider the diagonal
$
X \to X \times X
$ as an anima over $X \times X$ and thus by straightening as a functor $X \times X \to \mathrm{An}$ whose suspension spectrum we take. In other words we take the functor
\[
X \times X \to \Sp \quad (x,y) \mapsto \Sigma^\infty_+ P_{x,y} \ .
\]
where $P_{x.y}$ is the space of path from $x$ to $y$. It is very useful to think of $\Shv(X, \Sp) \otimes \Sp^X$ as sheaves on $X \times X$ which are locally constant in the direction of the second variable (but not the first one). The diagonal element however is locally constant in both directions.  

\begin{lemma}
The diagonal class represents a compact object in $\Shv(X, \Sp) \otimes \Sp^X$.
\end{lemma}
\begin{proof}
The compact objects in $\Shv(X, \Sp) \otimes \Sp^X$ are by general theory of sheaves (see \cite{HTT}) given by those elements $E$ in $\Sp^{X \times X} \subseteq \Shv(X, \Sp) \otimes \Sp^X$ which have the property that for every $x \in X$ the restriction $E|_{\{x\} \times X} \in \Sp^{\{x\} \times X} = \Sp^X$ is a compact object. In our case we have that $D_X|_{\{x\} \times X}$ is given the constant sphere over $X$. This is a compact object since $X$ is compact. 
\end{proof}

Note that this lemma totally fails if we replace $\Shv(X, \Sp) \otimes \Sp^X$ by $\Shv(X, \Sp) \otimes \Shv(X, \Sp)$ and consider the similar class. 
As a result we can think of the diagonal class as a strongly continuous functor  
\[
\Sp \to  \Shv(X, \Sp) \otimes \Sp^X
\]
since generally strongly continuous functors from spectra to any presentable, stable $\infty$-category are precisely given by compact objects. We will also denote this functor by $D_X$. 
\begin{definition}\label{def_As_eins}
For $X$ compact and of locally constant shape and $\calC$ dualisable we  define a map 
\begin{align*}
\mathbf{KK}^\alg(\Shv(X), \mathbf{C}) &\xto{- \otimes \Sp^X}  \mathbf{KK}^\alg(\Shv(X) \otimes \Sp^X, \mathbf{C} \otimes \Sp^X) \\
&  \xto{D^*_X} \mathbf{KK}^\alg(\Sp, \mathbf{C} \otimes \Sp^X) = K(\mathbf{C}^X).
\end{align*}
\end{definition}

\begin{proposition}
Assume that $X$ is compact and of locally contractible shape. Then 
the map of the previous definition is equivalent to the assembly map. In particular it can be made natural in $X$ and therefore also induces a map  for $X$ of locally compact shape but not necessarily compact:
\[
\colim_{K \subseteq X} \mathbf{KK}^\alg(\Shv(K), \mathbf{C})  \to K(\mathbf{C}^X) \ .
\]
which also is the assembly map. 
\end{proposition}

This statement already is strong enough the identify the image of certain classes upon applying the assembly map and will be used later. 

\subsection{Cosheaves}

Recall that in Section \ref{cor_kk} we not only presented a model of assembly in terms of KK-theory, but also gave a more concrete description: the inner hom in motives
\[
\underline{\mathrm{Hom}}_{\NcMot}(M\Shv(X;\Sp), M\calC)
\]
is equivalent to $f_*f^! M\calC = \Pi_\infty X \otimes M\calC$. This is nice, since for example the $K$-theory of this motive is then for compact spaces $X$ given by the source of assembly. It would be nice to have this motive as being the motive of an explicit category (abstractly we of course that every motive is the motive of a category).

What we will do is to employ the inner hom in $\Prld$. To this end we note that there is such an inner hom
$\underline{\mathrm{Hom}}_{\Prld}(\calC, \calD)$ for any pair of dualisable categories. By definition it is the right adjoint of tensoring with $\calC$.
It is a bit tricky to understand the inner hom, since it it certainly not the inner hom in $\Pr^\mathrm{L}$. 
But one can nevertheless deduce a bunch of things and give a sort of fomula. We will refrain from doing so here and refer to the literature. 

\begin{definition}
For a locally compact Hausdorff space $X$ and a dualisable, stable $\infty$-category $\calC$ we define 
\[
\coS(X;\calC) := \underline{\mathrm{Hom}}_{\Prld}(\Shv(X;\Sp), \calC)
\]
\end{definition}

\begin{proposition}\label{prop_co}
\begin{enumerate}
\item
The compact objects $\coS(X;\calC)^\omega$ are given by cosheaves on $X$ (i.e. contravariant functors $\mathrm{Open}(X) \to \calC$ satisfying the dual of the descent conditon for sheaves) which satisfy the following condition:  for an inclusion  of open sets $U \subseteq V$ in $X$ for which there exists a compact $K$ with $U \subseteq K \subseteq V$ the induced map $F(U) \to F(V)$ is compact. 
\item
The category $\coS(X;\calC)^\omega$ is covariantly functorial in proper maps and contravariantly functorial in local homeomorphisms (as in object of $\Prld$)
\item
The dual of $\coS(X;\calC)$ in  $\Pr^{\mathrm{L}}$ is given by $\coS(X, \calC^\vee)$ where $\calC^\vee$ is the dual of $\calC$. For example, wenn $\calC = \calD(R)$ then $\calC^\vee = \calD(R^{\mathrm{op}})$. 
\end{enumerate}
\end{proposition}

Note that by definition as an inner Hom, we can apply the forgetful functor $U: \Prld \to \Pr^{\mathrm{L}}$ which is strong symmetric monoidal to it and get a canonical interchange map
\[
U \underline{\mathrm{Hom}}_{\Prld}(\Shv(X;\Sp), \calC) = \underline{\mathrm{Hom}}_{\Pr^{\mathrm{L}}}(U\Shv(X;\Sp), U\calC) 
\]
that is a left adjoint functor
\begin{equation}\label{underlying}
\coS(X;\calC) \to \mathbf{coShv}(X; \calC) 
\end{equation}
which is not strongly continuous.

\begin{proposition}\label{left}
If $X$ is compact and has locally contractible shape, then 
this functor \eqref{underlying} has a left adjoint 
\[
i: \mathbf{coShv}(X; \calC)  \to \mathbf{coShv}(X;\calC)
\]
that is fully faithful. This functor sends a cosheaf $F$ satisfying the condition of Proposition \ref{prop_co})(1) to itself when considered as a compact object $\coS(X;\calC)^\omega$. 
\end{proposition}

Note that $i$ is then strongly continuous. This functor is in some sense the sheaves variant of the functor
\[
\calD(\Z)^\wedge_p \to \mathbf{Nuc}(\Z_p)
\]
and we like to think of it like a `completion'.
\begin{definition}
For a locally compact Hausdorff space $X$ and a dualisable, stable $\infty$-category $\calC$ we define 
\[
\coS_{\mathrm{cs}}(X;\calC) := \colim_{K \subseteq X} \coS(K;\calC)  \in \Prld
\]
where the colimit ranges over all compact subsets of $X$.
\end{definition}

\begin{proposition}
The category $\coS_{\mathrm{cs}}(X;\calC)$ is covariantly functorial in all continuous maps (as an object of $\Prld$) and the compact objects can be described as 
those cosheaves $F$ on $X$ that are supported on a compact subset $K \subseteq X$ and there satisfy the condition that 
\[
U \subset \!\subset V \Rightarrow F(U) \to F(V)   \text{   compact}.
\]
\end{proposition}

\begin{theorem}
Assume $X$ is countable at $\infty$, then
$M\coS(X;\calC) = f_*f^! M\calC$ and in particular K-theory is given by locally finite homology of $X$:
\[
K(\coS(X;\calC)) = f_*f^! K\calC \ .
\]
Without assumptions on $X$ we have that 
$M\coS_{\mathrm{cs}}(X;\calC)  = f_!f^! M\calC$ and  and in particular $K$-theory is given by homology of $X$:
\[
K(\coS_{\mathrm{cs}}(X;\calC)) = f_!f^! K\calC = \Pi_\infty X \otimes K\calC
\]
\end{theorem}

\subsection{The assembly map as a functor}

Now that we have modelled the source of the assembly map by means of $K$-theory of a category $\coS_{\mathrm{cs}}(X;\calC)$ we would like to also describe the assembly map itself as $K$-theory of a functor. We have already done so in terms of $KK$-theory. 
 To this end, recall that the target would have to be the category $\calC^X = \Fun(X, \calC)$ of functors from $X$ to $\calC$, which in the case $X = BG$ just reduces to objects of $\calC$ with a $G$-action. 

\begin{theorem}[Bartels--Nikolaus]
If $X$ has locally contractible shape then there is a natural, strongly continous functor
\[
A: \coS_{\mathrm{cs}}(X;\calC)) \to \mathbf{Fun}(X, \calC)
\]
such that upon taking $K$-theory it induces the assembly map 
\[
\Pi_\infty X \otimes K(\calC) \to K(  \mathbf{Fun}(X, \calC))
\]
Moreover $A$ is a Verdier quotient, i.e. a localization on a a full subcategory that can be explicitly described. 
\end{theorem}

We note that since the assembly map and also $A$ are natural and everything is a filtered colimit of restrictions to compact $K \subseteq X$, it is in fact enough to construct this functor for compact $X$. We will explain how this is done now. The construction is an exact translation of Definition \ref{def_As_eins}:
%
%
%
We define the assembly map 
\[
A: \coS_{\mathrm{cs}}(X;\calC)) =  \underline{\mathbf{Hom}}_{\Prld}(\Shv(X;\Sp), \calC)  \to \mathbf{Fun}(X, \calC)
\]
as the composition
\begin{align*}
\underline{\mathbf{Hom}}_{\Prld}(\Shv(X), \mathbf{C}) &\xto{- \otimes \Sp^X}  \underline{\mathbf{Hom}}_{\Prld}(\Shv(X) \otimes \Sp^X, \mathbf{C} \otimes \Sp^X) \\
&   \xto{ \underline{\mathbf{Hom}}(D_X, -)}  \underline{\mathbf{Hom}}_{\Prld}(\Sp, \mathbf{C} \otimes \Sp^X) = \mathbf{C}^X.
\end{align*}
%
where for the last step we have used that $\calC \otimes \Sp^X \simeq \calC^X$ as well as the fact that the inner hom out of the tensor unit (i.e. the category of spectra) is equivalent to the target. 
Clearly, this is just a translation of the description of the assembly map from Section \ref{sec81} to the internal hom, so clearly this map induces the assembly map 
on homology.

There is another and more digestible description of the assembly map that we would like to explain now. Namely one has a factorization
\begin{equation}\label{factorization}
\begin{tikzcd}
\coS(X;\calC) \arrow{rd}{A} \arrow{d} & \\
 \mathbf{coShv}(X; \calC) \arrow{r} &  \calC^X
\end{tikzcd}
\end{equation}
where the left map is the canonical map \ref{underlying} which has a fully faithful right adjoint. 
The idea is to describe the right adjoint of $A$ which is a fully faithful inclusion. The lower horizontal one is a functor that we would like to describe 
now. It also has a fully faithful right adjoint, namely the right adjoint includes  $\calC^X$ as locally constant cosheaves into 
$\mathbf{coShv}(X; \calC)$. Here a cosheaf $F$ is called locally constant, if there is an open cover, such that the restriction of $F$ 
to each of the opens is constant. But note that being constant for a cosheaf means, it is the cosheafification of a constant copresheaf. 
We warn the reader that under Verdier duality this does not correspond to constant sheaves (it rather corresponds to sheaves of the form $t^!(c)$ for some object $c \in \calC$ for $t: X \to \mathrm{pt}$. Thus the left adjoint functor is a functor that universally turns a cosheaf $F$ into a locally constant cosheaf. 

\begin{warning}
The factorization \ref{factorization} is only a factorization as left adjoint functors, not as strongly continuous functors. Thus it does not induce a factorization on the level of $K$-theory. It is however true, that the lower map in this factorization is strongly continuous and that the left vertical map has a strongly continuous left adjoint as started in Proposition \ref{left} . That leads to a commutative diagram
\[
\begin{tikzcd}
\coS(X;\calC) \arrow{rd}{A} & \\
 \mathbf{coShv}(X; \calC) \arrow{u}{i} \arrow{r} &  \calC^X
\end{tikzcd}
\]
of strongly continuous functors with the same lower horizontal map. If we think of $i$ as exhibiting $\coS(X;\calC)$ as some sort of completion of $\mathbf{coShv}(X; \calC)  $ (this is a purely moral statement) then the assembly map is simply the map that makes a cosheaf universally locally constant. 
\end{warning}

\end{document}